\documentclass[11pt,a4paper,oneside]{article}
\usepackage[utf8]{inputenc}
\usepackage[T1]{fontenc}
\usepackage{amsmath,amssymb,amsthm}
\usepackage{mathrsfs}
\usepackage{enumerate}
\usepackage{wasysym}
\usepackage{cite}
\usepackage{graphicx}
\usepackage{braket}
\usepackage[dvipsnames]{xcolor}
\usepackage{tikz}
\usepackage{hyperref}
\usepackage{geometry}

\newtheorem{theorem}{{Theorem}}[section]
\newtheorem{lemma}[theorem]{{Lemma}}
\newtheorem{definition}[theorem]{{Definition}}
\newtheorem{korollar}[theorem]{{Corollary}}

\numberwithin{equation}{section}

\newcommand{\Sp}{\mathbb S}

\newcommand{\N}{\mathbb{N}}
\newcommand{\R}{\mathbb{R}}

\newcommand\extrafootertext[1]{%
    \bgroup
    \renewcommand\thefootnote{\fnsymbol{footnote}}%
    \renewcommand\thempfootnote{\fnsymbol{mpfootnote}}%
    \footnotetext[0]{#1}%
    \egroup
}

\begin{document}
\author{Jan-Henrik Metsch
\thanks{Corresponding author: J.-H. Metsch, Department of Mathematics, University of Freiburg, Germany ({\tt\small jan.metsch@math.uni-freiburg.de})}
}
\title{Axially Symmetric Willmore Minimizers with Prescribed Isoperimetric Ratio}
\maketitle

\begin{abstract}
    We establish the existence and smoothness of minimizers of the Willmore energy among axially symmetric surfaces of spherical type with prescribed isoperimetric ratio. Afterwards, we study the behavior of these minimizers as the isoperimetric ratio tends to zero.
\end{abstract}

\extrafootertext{2020 Mathematics Subject Classification: Primary: 49Q10, 35B07, Secondary: 35A15, 35B65}
\extrafootertext{Keywords:  Willmore Functional, Isoperimetric Constraint, Direct Method, Singular Limit}

\section{Introduction}
The isoperimetric ratio $\mathcal I(f)$ of a smooth embedding $f:\Sp^2\rightarrow\R^3$ is defined as
$$\mathcal I[f]:=6\sqrt\pi\frac{V[f]}{A[f]^{\frac32}},$$
where $A[f]$ and $V[f]$ denote the surface area of $f$ and the volume enclosed by $f$, respectively. By the classical isoperimetric inequality 
$$V[f]\leq\frac1{6\sqrt\pi} A[f]^{\frac32}$$
we have $\mathcal I(f)\in(0,1]$ with $\mathcal I(f)=1$ occurring only for round spheres. For $\sigma\in(0,1]$, Schygulla \cite{schygulla} introduces the problem of minimizing the Willmore energy 
$$\mathcal W[f]:=\frac14\int_{\Sp^2} H[f]^2 d\mu_f$$
under all embeddings $f\in C^\infty(\Sp^2,\R^3)$ with $\mathcal I[f]=\sigma$. This problem is motivated by the models for the shape of blood cells by  Canham \cite{canham} and Helfrich \cite{helfrich}, who independently proposed to model the bending energy of a blood cell by a quadratic expression in its curvature. Helfrich's article \cite{helfrich} also motivates the isoperimetric constraint: He prescribes the area and the volume. However, as the Willmore energy is invariant under scaling, these two constraints reduce to prescribing the isoperimetric ratio. By Hopf's theorem, the constraint of prescribed isoperimetric ratio $\mathcal I[f]=\sigma$ is nondegenerate for $\sigma<1$ so that by Lagrange's theorem, the minimizers satisfy the Euler-Lagrange equation 
\begin{equation}\label{IntroPDE}
\Delta_g H+\frac12 H(H^2-4K)=\Lambda\sigma\left(\frac1{V[f]}-\frac{3 H}{2A[f]}\right),
\end{equation}
where $H$ and $K$ denote the mean and the Gauß curvature of $f$ respectively.  In this article, we investigate the axially symmetric case. For $\sigma\in(0,1]$, we introduce 
$$\mathcal S_\sigma:=\set{f\in C^\infty(\Sp^2,\R^3)\textrm{ embedding }|\ \mathcal I[f]=\sigma\textrm{ and }f\circ T=f\textrm{ for all }T\in\operatorname{SO}(2)}.$$
Additionally, we let $\beta(\sigma):=\inf_{f\in\mathcal S_\sigma}\mathcal W[f]$. A convenient characterization of the surfaces in $\mathcal S_\sigma$ is achieved using their \emph{profile curves}. Indeed, after translation, every $f\in\mathcal S_\sigma$ is of the form 
\begin{equation}\label{introrevolutioneq}
f_\gamma:[0,1]\times \Sp^1\rightarrow\R^3,\ f_\gamma(s,\omega):=\begin{bmatrix}
\gamma_1(s)\omega\\
\gamma_2(s)
\end{bmatrix}
\end{equation}
where $\gamma\in\mathcal F_\sigma:=\set{\gamma\in\mathcal P\ |\ \mathcal I[f_\gamma]=\sigma}$ and $\mathcal P$ denotes the set of admissible profile curves
$$\mathcal P:=\left\{
\gamma\in C^\infty([0,1],\R^2)\ \big|\ 
\begin{array}{l}
     \gamma_1(0)=\gamma_1(1)=0\textrm{ and } \gamma_1(t)>0 \textrm{ for all }t\in(0,1)
\end{array}
\right\}.$$
Conversely, using Equation \eqref{introrevolutioneq}, any $\gamma\in\mathcal P$ can be assigned a surface of spherical type in $\R^3$ so that we can identify $f_\gamma$ and $\gamma$. 

\begin{theorem}\label{theorem1}
For all $\sigma\in(0,1]$, there exists $f\in\mathcal S_\sigma$ such that $\mathcal W[f]= \beta(\sigma)$. Moreover, $\beta(\sigma)< 8\pi$ for all $\sigma\in(0,1]$ and $\beta(\sigma)\rightarrow 8\pi$ as $\sigma\rightarrow 0^+$. 
\end{theorem}

The proof of \ref{theorem1} is split into three steps. First, the existence of a weak minimizing profile $\gamma\in W^{1,2}((0,1))\cap W^{2,2}_{\operatorname{loc}}((0,1))$ is established using the direct method of the calculus of variations. The arguments used here follow those executed by Choski and Venernoni in \cite{choskiveneroni}. Next, the smoothness of $\gamma$ on $(0,1)$ -- away from the axis of rotation --  is established using elliptic regularity arguments. 
To study regularity at the axis of rotation, a careful analysis of the local graph representation 
$$\set{(r\omega, u(r))\ |\ r\in[0, r_0),\ \omega\in \mathbb S^1}$$
of $f_\gamma$ near the axis of rotation is required. For the unconstrained Willmore equation, such analysis has been carried out in \cite{chenODE} by Chen and Li. We adapt their analysis to prove
\begin{equation}\label{introsingularsolution}
u(r)=\lambda r^2\ln(r)+C^2([0, r_0)).
\end{equation}
The parameter $\lambda$ is a residue, as has been studied by Kuwert and Schätzle in \cite{removability} and cannot be eliminated by ODE arguments alone, as the Euler-Lagrange equation \eqref{IntroPDE} allows for singular solutions of the form in Equation \eqref{introsingularsolution}.\\

After proving the existence of smooth minimizers, we investigate their behavior in the limit $\sigma\rightarrow0^+$. This has already been done by Schygulla \cite{schygulla}. He proved that as $\sigma\rightarrow0^+$, his minimizers converge to a double sphere in the sense of measures. In \cite{kuwertli}, this result was later refined by Kuwert and Li, who prove that asymptotically, Schygulla's minimizers look like two spheres joined by a catenoidal neck. We prove that our axially symmetric minimizers display the same behavior. Restricting to the class 
$$\mathcal F^{0+}_\sigma:=\left\{\gamma\in\mathcal F_\sigma,\ \bigg|\ \gamma(0)=(0,0),\ A[f_\gamma]=1\textrm{ and }\int_0^1 \gamma_2(t)dt\geq 0\right\},$$
we eliminate the invariance of the problem under vertical translation, scaling and reflection -- thereby out ruling the trivial obstructions to convergence. Characterizing the solutions to the Euler-Lagrange equation in the limit $\sigma\rightarrow0^+$ we then establish:

\begin{theorem}[Convergence to a Double Sphere]\label{doublesphereconvergence}
    Let $(\sigma_k)\subset(0,1]$ be a sequence of isoperimetric ratios $\sigma_k\rightarrow 0^+$ and $(\gamma^{(k)})_k$ be a sequence of profile curves $\gamma^{(k)}\in\mathcal F^{0+}_{\sigma_k}$ satisfying $\beta(\sigma_k)=\mathcal W[f_{\gamma^{(k)}}]$. Then as $k\rightarrow\infty$,
    $$\gamma^{(k)}(t)\rightarrow\kappa(t):=(0,R)+R(|\sin(2\pi t)|, -\cos(2\pi t))\hspace{.5cm}\textrm{with }R:=\frac{1}{\sqrt{8\pi}}$$
   in $C^\infty([0,1]\backslash \{\frac12\})$ and in $W^{1,2}((0,1))$.
\end{theorem}

Since $\kappa$ is not smooth at $t=\frac12$, it is clear that the curves $\gamma^{(k)}$ have to develop some form of singularity near $t=\frac12$. We prove that near $t=\frac12$, the profile curves solve approximately the equation $H=0$. Using a suitable blow-up argument, we then establish the following theorem:   

\begin{theorem}[Formation of a Catenoidal Neck]\label{CatenoidNeckTheorem} Let $\sigma_k\rightarrow0^+$ and $\gamma^{(k)}\in\mathcal F^{0+}_{\sigma_k}$ satisfy $\mathcal{W}[f_{\gamma^{(k)}}]=\beta(\sigma_k)$. There exist $k_0\in\N$ and $\rho>0$ such that the following statements hold:
\begin{enumerate}[(1)]
    \item For $k\geq k_0$ there exists a unique $\tau_k\in[\frac18,\frac78]$ such that $\gamma_1^{(k)}(\tau_k)=\epsilon_k:=\inf_{[\frac18,\frac78]}\gamma_1^{(k)}$.
    \item $\dot\gamma_1^{(k)}<0$ on $[\tau_k-\rho,\tau_k)$ and $\dot\gamma_1^{(k)}>0$ on $(\tau_k,\tau_k+\rho]$. 
    \item The parameters from $(1)$ satisfy $\epsilon_k\rightarrow 0^+$ and $\tau_k\rightarrow\frac12$.
    \item After potentially putting
$\hat\gamma^{(k)}(t):=\left(\gamma_1^{(k)}(1-t),\ \gamma_2^{(k)}(1-t)-\gamma_2^{(k)}(1)\right)$,
it can be assumed that $\dot\gamma_2^{(k)}(\tau_k)\leq 0$ for all $k\geq k_0$ and the following convergence is in $C^\infty_{\operatorname{loc}}(\R)$:
    $$\frac1{\epsilon_k}\left(\gamma^{(k)}(\tau_k+\epsilon_k t)-(0,\gamma_2^{(k)}(\tau_k))\right)\rightarrow \left(\sqrt{1+L_*^2 t^2}, -\operatorname{asinh}(L_*t)\right)\hspace{.5cm}\textrm{with }L_*:=\sqrt{\frac\pi2}$$
\end{enumerate} 
\end{theorem}

The results from Theorems \ref{doublesphereconvergence} and \ref{CatenoidNeckTheorem} are in agreement with the numerically computed minimizers in \cite{seifert} (see Figure 9).  Using numerical methods, several authors have computed candidates for the minimizers. For an overview, we refer to \cite{seifert}. However, we mention that Helfrich has already investigated axially symmetric minimizers in the aforementioned \cite{helfrich} and in a joint work with Deuling in \cite{helfrichDeuling}. \\

The first rigorous treatment of Willmore minimizers with isoperimetric constraint was perhaps carried out in \cite{Nagasawa}, where a one-parameter family of critical points bifurcating from the sphere is constructed. After the aforemetnioned article \cite{schygulla} from Schygulla, the isoperimetric problem for surfaces of higher genus was studied in \cite{KellerIso,kusnerISO,MondinoScharrerISO}.\\

Regarding the axisymmetric case, Choski and Veneroni \cite{choskiveneroni} study systems of axisymmetric surfaces $S=(\Sigma_1,...,\Sigma_N)$ where each $\Sigma_i$ is either an embedded sphere or an embedded torus that minimize the Helfrich energy under the generalized isoperimetric constraint 
$$6\sqrt\pi\frac{V[S]}{A[S]^{\frac32}}=\sigma\in(0,1]\hspace{.5cm}\textrm{where }A[S]=\sum_{i=1}^N A[\Sigma_i]\textrm{ and }V[S]=\sum_{i=1}^N V[\Sigma_i].$$
After the initial submission of this mansucript, the author was informed that Rupp \cite{Ruppiso} recently estblished the existence of critical Willmore spheres with prescribed isoperimetric ratio by analyzing the \emph{`Willmore flow with prescribed isoperimetric ratio'}. Restricting to axially symmetric initial datums, his results establish the existence of axially symmetric solutions to the Euler-Lagrange Equation \eqref{IntroPDE}.\\

\cite{scharrer2022embedded} studies axisymmetric tori minimizing the Willmore energy with an isoperimetric constraint. \\

This article is structured as follows. Section \ref{PreliminariesSection} collects the notation and formulas we require throughout this article. In Section \ref{diectmethod}, the existence of a weak minimizer is proven, and in Section \ref{regularityawayfromaxis}, its smoothness is established. Sections \ref{doublespheresection} and \ref{catenoidNeckSection} contain the proofs of Theorems \ref{doublesphereconvergence} and \ref{CatenoidNeckTheorem}. Technical arguments are moved to the \hyperref[Appendix1]{appendix}.

\section{Preliminaries}\label{PreliminariesSection}

\paragraph{General Notation}\ \\
For $k\in\N_0\cup\set\infty$ and $U\subset\R^n$ open, we denote the set of $C^k$-functions with compact support in $U$ by $C^k_0(U)$.
The right half plane in $\R^2$ is denoted by
$$\mathcal H^2:=\set{(x_1,x_2)\in\R^2\ |\ x_1\geq 0}.$$
If $Q$ is some functional, we put 
$$\delta Q[x]v:=\frac d{dt}\bigg|_{t=0} Q[x+tv].$$

\paragraph{Immersions}\ \\
Let $f:\Sp^2\rightarrow\R^3$ be a sufficiently regular immersion and $n$ denote a regular unit normal vector field along $f$. The first and second fundamental forms are respectively denoted by
$$g_{ij}:=\langle\partial_i f,\partial_j f\rangle 
\hspace{.5cm}\textrm{and}\hspace{.5cm}
h_{ij}:=\langle\partial_{ij} f, n\rangle.$$
The surfaces measure on $\Sp^2$ induced by $f$ is denoted by $\mu_f$. In local coordinates $d\mu_f=\sqrt{\det g}dx$. Let $k_1$ and $k_2$ denote the principal curvatures of $f$ -- that is, the eigenvalues of the Weingarten operator $g^{-1}h$. The mean curvature $H$ and the Gauß curvature $K$ are given by
$$H=k_1+k_2
\hspace{.5cm}\textrm{and}\hspace{.5cm}
K=k_1k_2.$$
For example, considering the inclusion $\imath:\Sp^2\hookrightarrow\R^3$ with the interior normal field $n=-\imath$, we get $H[f]=2$ and $K[f]=1$.\\

\paragraph{Functionals}\ \\
Let $f:\Sp^2\rightarrow\R^3$ be a sufficiently regular immersion. The Willmore energy of $f$ is defined as 
\begin{equation}\label{willmoreenergydef}
\mathcal W[f]:=\frac14\int_{\Sp^2} H[f]^2 d\mu_f.
\end{equation}
For any of the two choices of continuous normal vector field $n$ along $f$, the surface area, the volume and the isoperimetric ratio of $f$ are respectively defined as
\begin{equation}\label{areavoldef}
A[f]:=\int_{\Sp^2} d\mu_f,
\hspace{.5cm}
V[f]:=\left|\int_{\Sp^2}f_2 \langle n, e_2\rangle d\mu_f\right|
\hspace{.5cm}\textrm{and}\hspace{.5cm}
\mathcal I[f]:=6\sqrt\pi\frac{V[f]}{A[f]^{\frac32}}.
\end{equation}
To justify the definition of volume, we have the following lemma:
\begin{lemma}\label{generalvolumelemma}
Let $f\in C^1(\Sp^2,\R^3)$ be an embedding. Then there exists a bounded $C^1$-domain $\Omega\subset\R^3$ with $\partial\Omega=f(\Sp^2)$ and which satisfies $|\Omega|=V[f]$.
\end{lemma}
\begin{proof}
The existence of $\Omega$ is due to the Jordan-Brouwer separation theorem (see e.g. \cite{MayerBook}, Corollary 5.24). Let $n_{\operatorname{int}}$ denote the interior unit normal along $\partial\Omega$. Since $\Omega\in C^1$, we can use Gauß's theorem to compute
$$|\Omega|
=-\int_{\partial\Omega}x_2\langle n_{\operatorname{int}}, e_2\rangle dS
=-\int_{\Sp^2}f_2\langle n_{\operatorname{int}}, e_2\rangle d\mu_f.$$
This establishes Equation \eqref{areavoldef} when $n$ in the interior normal and thereby also for the exterior normal $n_{\operatorname{ext}}=-n_{\operatorname{int}}$.
\end{proof}

\paragraph{First Variation of the Willmore Energy}\ \\
Let $f:\Sigma\rightarrow\R^3$ be a smoothly immersed surface and $\Phi:(-\epsilon_0,\epsilon_0)\rightarrow \R^3$ be a smooth variation of $f$. Denote by $n$ the unit normal field used to define the scalar mean curvature $H=\langle\vec H, n\rangle$. The following formula is, for example, derived in \cite{AK}, Theorem 1. 
\begin{equation}\label{classicalformula}
\frac{d}{d\epsilon}\bigg|_{\epsilon=0}\mathcal W[\Phi(\epsilon)]=\int_{\Sigma}W[f]\langle \Phi'(0), n\rangle d\mu_f+\int_{\partial\Sigma}\omega(\eta) dS_f.
\end{equation}
In this formula, $\eta$ is the interior unit normal along $\partial\Sigma$ and splitting the variational vector field $\Phi'(0)$ into a normal and a tangential part by writing $\Phi'(0)=\varphi n + Df X$, the Willmore operator $W[f]$ and the one-form appearing in Equation \eqref{classicalformula} are given by
\begin{align}
    W[f]&=\frac12\left(\Delta_g H+\frac12H(H^2-4K)\right),\label{WillmoreOperator}\\
\omega(\eta)&=\frac12\left(\varphi\frac{\partial H}{\partial\eta}-H\frac{\partial\varphi}{\partial\eta}-\frac12 H^2 \langle Df\eta, Df X\rangle\right).\label{WillmoreBoundaryterm}
\end{align}

\subsection{Graphical Surfaces}
Let $r_0>0$ and $u:[0,r_0)\rightarrow\R$ be sufficiently regular. We consider the axisymmetric surface $\Sigma_u$ generated by rotating the graph of $u$ around the $x_3$-axis:
$$f_u:[0,r_0)\times[0,2\pi)\rightarrow\R^3,\ f_u(r,\theta):=(r\cos\theta,r\sin\theta, u(r))$$
We collect the geometric data of $\Sigma_u$. The first fundamental form is given by
\begin{equation}\label{graphmetric}
g(r,\theta)=\begin{bmatrix}
1+u'(r)^2 & \\
 & r^2
\end{bmatrix}
\hspace{.5cm}\textrm{and}\hspace{.5cm}\sqrt{\det g}=r\sqrt{1+u'(r)^2}.
\end{equation}
Next, we introduce a normal field $n$, which we choose to coincide with the interior unit normal when $u(r)=\sqrt{1-r^2}$ -- that is when $\Sigma_u$ is a sphere.
$$n(r,\theta):=
\frac{1}{\sqrt{1+u'(r)^2}}
\begin{bmatrix}
u'(r)\cos\theta\\
u'(r)\sin\theta\\
-1
\end{bmatrix}$$
Using this normal, the scalar second fundamental form is given by 
\begin{equation}\label{graphsecondff}
h=-\frac1{\sqrt{1+u'(r)^2}}\begin{bmatrix}
u''(r) & \\
 & ru'(r)
\end{bmatrix}.
\end{equation}
Combining Equations \eqref{graphmetric} and \eqref{graphsecondff}, we obtain the following formulas for the two principal curvatures:
\begin{equation}\label{graphicalPrincipalCurvatures}
k_1=-\frac{u''(r)}{\sqrt{1+u'(r)^2}^3}
\hspace{.5cm}\textrm{and}\hspace{.5cm}
k_2=-\frac{u'(r)}{r\sqrt{1+u'(r)^2}}
\end{equation}
Using these, it is readily checked that the mean curvature $H$ of $\Sigma_u$ is given by
\begin{equation}\label{graphH}
H=k_1+k_2=-\frac1r\frac{d}{dr}\left[r\frac{u'(r)}{\sqrt{1+u'(r)^2}}\right].
\end{equation}

\subsection{Profile Curves}
Let $\gamma:[0,1]\rightarrow\mathcal H^2$ be a sufficiently regular curve. We begin by recalling that the arc length of $\gamma$ is defined as 
\begin{equation}
    L[\gamma]:=\int_0^1|\dot\gamma(t)|dt.
\end{equation}
We consider the surface $\Sigma_\gamma$ generated by rotating the graph of $\gamma$ around the $x_2$-axis. Concretely, it is given by the parameterization
\begin{equation}\label{fgammadef}
f_\gamma:[0,1]\times[0,2\pi)\rightarrow\R^3,\ f_\gamma(t,\theta):=\begin{bmatrix}
\gamma_1(t)\cos\theta\\
\gamma_1(t)\sin\theta\\
\gamma_2(t)
\end{bmatrix}.
\end{equation}
In the following, we collect the geometric data of $\Sigma_\gamma$. The first fundamental form is given by
\begin{equation}\label{metricprofilecurve}
g(t,\theta)=\begin{bmatrix}
|\dot\gamma(t)|^2 & \\
 & \gamma_1(t)^2
\end{bmatrix}
\hspace{.5cm}\textrm{and}\hspace{.5cm}\sqrt{\det g}=\gamma_1(t)|\dot\gamma(t)|.
\end{equation}
We denote the surface element on $\Sigma_\gamma$ by $d\mu_{\Sigma_\gamma}:=\gamma_1|\dot\gamma_1|dtd\theta$. Next, we define normal fields $n$ along $\Sigma_\gamma$ and $\nu$ along $\gamma$. 
\begin{equation}\label{normalsprofilecurve}
n(t,\theta)=\frac1{|\dot\gamma(t)|}\begin{bmatrix}
-\dot\gamma_2(t)\cos\theta\\
-\dot\gamma_2(t)\sin\theta\\
\dot\gamma_1(t)
\end{bmatrix}
\hspace{.5cm}\textrm{and}\hspace{.5cm}
\nu(t)=\frac1{|\dot\gamma(t)|}\begin{bmatrix}
-\dot\gamma_2(t)\\
\dot\gamma_1(t)
\end{bmatrix}
\end{equation}
This normal is chosen, such that it agrees with the interior normal along $\Sigma_\gamma$ when $\gamma(t)=(\cos(\pi t), \sin(\pi t))$ -- that is when $\gamma$ parameterizes a circle run through counterclockwise. Using this normal, we get the following formula for the scalar second fundamental form of $\Sigma_\gamma$:
\begin{equation}\label{curvesecondff}
h(t,\theta):=\frac1{|\dot\gamma(t)|}\begin{bmatrix}
    \ddot\gamma_2\dot\gamma_1-\ddot\gamma_1\dot\gamma_2 & \\
     & \gamma_1\dot\gamma_2
\end{bmatrix}
\end{equation}
Combining Equations \eqref{metricprofilecurve} and \eqref{curvesecondff} gives the following formulas for the two principal curvatures:
\begin{equation}\label{principalcurvatures} 
k_1=\frac{\ddot\gamma_2\dot\gamma_1-\ddot\gamma_1\dot\gamma_2}{|\dot\gamma|^3}
\hspace{.5cm}\textrm{and}\hspace{.5cm}
k_2=\frac{\dot\gamma_2}{\gamma_1|\dot\gamma|}
\end{equation}

\paragraph{Functionals}\ \\
Using Equations \eqref{willmoreenergydef}, \eqref{areavoldef}, \eqref{metricprofilecurve} and \eqref{normalsprofilecurve}, we get the following formula for the Willmore energy, the area and the volume of $\Sigma_\gamma$:
\begin{align}
    \mathcal W[\gamma]&:=\mathcal W[\Sigma_\gamma]=\frac\pi2\int_0^1 (k_1(t)+k_2(t))^2 \gamma_1(t)|\dot\gamma(t)| dt\label{willmoredef}\\
    A[\gamma]&:=A[\Sigma_\gamma]=2\pi\int_0^1 \gamma_1(t)|\dot\gamma(t)|dt\label{areadef}\\
    V[\gamma]&:=V[\Sigma_\gamma]=\pi\left|\int_0^1 \gamma_1(t)^2\dot\gamma_2(t)dt\right|\label{volumedef}
\end{align}
We also put $\mathcal I[\gamma]:=\mathcal I[\Sigma_\gamma]$. In view of Lemma \ref{generalvolumelemma}, we have the following result:
\begin{lemma}\label{volumegauss}
    Let $\gamma\in C^1([0,1],\mathcal H^2)$ be injective, $\dot\gamma(t)\neq 0$ for all $t\in[0,1]$ and $\gamma_1(0)=\gamma_1(1)=0$ as well as $\dot\gamma_2(0)=\dot\gamma_2(1)=0$. Then $\Sigma_\gamma$ is the boundary of a bounded $C^1$ domain $\Omega$ and $|\Omega|=V[\gamma]$. 
\end{lemma}

\paragraph{Arc Length Parameterization}\ \\
A regular curve $\gamma:[0,1]\rightarrow\mathcal H^2$ is \emph{parameterized proportional to arc length} if $|\dot\gamma(t)|=L[\gamma]=:L$ for all $t\in[0,1]$. These curves enjoy several beneficial properties. First 
\begin{equation}\label{velocity_acc_Perp_Eq}
    0=\frac12\frac{d}{dt}|\dot\gamma(t)|^2=\langle\dot\gamma(t),\ddot\gamma(t)\rangle=\dot\gamma_1(t)\ddot\gamma_1(t)+\dot\gamma_2(t)\ddot\gamma_2(t).
\end{equation}
Considering Equation \eqref{principalcurvatures}, it is then straightforward to check that 
\begin{equation}\label{ArcLength_H_Identity}
\dot\gamma_1 H= \frac{\ddot\gamma_2}L+\frac{\dot\gamma_1\dot\gamma_2}{\gamma_1 L}
\hspace{.5cm}\textrm{and}\hspace{.5cm}
\dot\gamma_2 H= -\frac{\ddot\gamma_1}L+\frac{\dot\gamma_2^2}{\gamma_1 L}.
\end{equation}
Additionally, it is readily checked that
\begin{equation}\label{ArcLength_K_Identity}
    k_1^2=\frac{|\ddot\gamma|^2}{L^4}
    \hspace{.5cm}\textrm{and}\hspace{.5cm}
    K=k_1k_2=-\frac{\ddot\gamma_1}{\gamma_1 L^2}.
\end{equation}

\section{Direct Method of the Calculus of Variations}\label{diectmethod}
In this section, we employ the direct method of the calculus of variations to obtain an axially symmetric minimizer of the Willmore energy with prescribed isoperimetric ratio $\sigma$. The analysis follows with only minor modifications, the arguments in \cite{choskiveneroni}.

\subsection{Functional Space}\label{functionalspacesection}
We first define a suitable space of profile curves.
\begin{definition}[Admissible Profile Curves]\label{classPdef}
A curve $\gamma=(\gamma_1,\gamma_2):[0,1]\rightarrow \mathcal H^2$ is called an \emph{admissible profile curve} if:
\begin{enumerate}[(1)]
    \item $\gamma\in C^1([0,1])\cap W^ {2,2}_{\operatorname{loc}}((0,1))$.
    \item $\gamma_1(0)=\gamma_1(1)=0$ and $\gamma_1(t)>0$ for all $t\in(0,1)$.
    \item $\gamma$ is parameterized proportional to arc length.
    \item $\Sigma_\gamma$ has curvature bounded in $L^2(\mu_{\Sigma_\gamma})$. That is 
    $$\int_{\Sp^2}k_1^2+k_2^2 d\mu_{\Sigma_\gamma}=\int_0^1 (k_1^2+k_2^2)2\pi|\dot\gamma|\gamma_1 dt<\infty.$$
\end{enumerate}
We denote the set of all admissible profile curves by $\mathcal P$.
\end{definition}

We wish to deduce the existence of a curve $\gamma\in\mathcal P$ that minimizes the Willmore energy under all curves in $\mathcal P$ with prescribed isoperimetric ratio. Given $\sigma\in(0,1)$ we put
$$
\mathcal F_\sigma:=
\left\{
\gamma\in\mathcal P\ \bigg|\ A[\gamma]=1\hspace{.2cm}\textrm{and}\hspace{.2cm}V[\gamma]=\frac{\sigma}{6\sqrt\pi}
\right\}
\hspace{.5cm}\textrm{and}\hspace{.5cm}
\beta(\sigma):=\inf\set{\mathcal W[\gamma]\ |\  \gamma\in\mathcal F_\sigma}.$$
Note that in the definition of $\mathcal F_\sigma$, we prescribe the area and the volume of $\gamma$. However, as both the Willmore energy $\mathcal W$ and the isoperimetric ratio $\mathcal I$ are invariant under scaling, we have 
\begin{equation}\label{betaDefinition}
\beta(\sigma)
=\inf\set{\mathcal W[\gamma]\ |\  \gamma\in\mathcal F_\sigma}
=\inf\set{\mathcal W[\gamma]\ |\  \gamma\in\mathcal P\textrm{ with }\mathcal I[\gamma]=\sigma}.
\end{equation}

Unfortunately, the space $\mathcal P$ does not have sufficient compactness properties and we require a weakened notation of admissible profile curves. 
\begin{definition}[Weak Profile Curves]\label{PWeakDefinition}
We say that $\gamma:[0,1]\rightarrow\mathcal H^2$ is a \emph{weak profile curve} if:
\begin{enumerate}[(1)]
    \item $\gamma\in W^{1,2}((0,1))$ (so in particular $\gamma\in C^0([0,1])$). Additionally  $\gamma_1(0)=\gamma_1(1)=0$ and $\gamma_1(t)>0$ for almost all $t\in[0,1]$.
    \item $|\dot\gamma|=L[\gamma]$ almost everywhere.
    \item $\gamma\in W^{2,2}_{\operatorname{loc}}(U)$ for all $U\subset[0,1]$ satisfying $\gamma_1>0$ on $U$.
    \end{enumerate}
    By (1) and (3), $\ddot\gamma$, and thus also $k_1$, is defined almost everywhere even though $\dot\gamma$ is not weakly differentiable on $(0,1)$. This allows the formulation of:
    \begin{enumerate}[(1)]\setcounter{enumi}{3}
    \item $\Sigma_\gamma$ has curvature bounded in $L^2(\mu_{\Sigma_\gamma})$. That is 
    $$\int_0^1 (k_1^2+k_2^2)2\pi|\dot\gamma|\gamma_1 dt<\infty.$$
\end{enumerate}
The set of weak profile curves is denoted by $\mathcal P^w$. 
\end{definition}

Clearly $\mathcal P\subsetneq\mathcal P^w$. However, the next theorem establishes that weak profile curves correspond exactly to chaining several admissible profile curves from $\mathcal P$ together. For a proof, we refer to Lemma 6 in \cite{choskiveneroni}.
\begin{theorem}[Regularity of Weak Profile Curves]\label{weakandstrongPconnection}
Let $\gamma\in \mathcal P^w$, put $M:=\int_0^1 (k_1^2+k_2^2)2\pi|\dot\gamma|\gamma_1 dt$ and let $n_0:=\lfloor\frac{M}{8\pi}\rfloor$. Then there exists $m\leq n_0$ and points $0=\tau_0<\tau_1...<\tau_m=1$ such that (up to linear reparamterization) $\gamma|_{[\tau_i,\tau_{i+1}]}$ belongs to $\mathcal P$. Additionally, when $t\rightarrow\tau_i^ \pm$ we have
$$\lim_{t\rightarrow\tau_i^ \pm}\dot\gamma_2(t)=0
\hspace{.5cm}\textrm{and}\hspace{.5cm}
\lim_{t\rightarrow\tau_i^ \pm}\dot\gamma_1(t)=\pm L[\gamma].$$
\end{theorem}

As an immediate consequence, we obtain the following corollary:
\begin{korollar}
Let $\gamma\in\mathcal P$. Then 
\begin{equation}\label{velocityatends}
    \dot\gamma_1(0)=L[\gamma],\ \dot\gamma_1(1)=-L[\gamma]\hspace{.5cm}\textrm{and}\hspace{.5cm}\dot\gamma_2(0)=\dot\gamma_2(1)=0.
\end{equation}
Consequently $\Sigma_\gamma$ is a $C^1$ surface.
\end{korollar}

\paragraph{Properties of Profile Curves}\ \\
We collect several properties of (weak) profile curves. 
The following theorem is Lemma 2 from \cite{choskiveneroni}.
\begin{theorem}[Choski and Veneroni]\label{choskiveneroniresult}
Let $\gamma\in\mathcal P$. Then 
\begin{equation}\label{lengthbound}
\frac{A[\gamma]}{2\pi\operatorname{diam}(\Sigma_\gamma)}\leq L[\gamma]\leq \frac{\sqrt{A[\gamma]}}{2\pi}\left(
\left(\int_{\Sp^2} k_1^2 d\mu_{\Sigma_\gamma}\right)^{\frac12}+
\left(\int_{\Sp^2}  k_2^2 d\mu_{\Sigma_\gamma}\right)^{\frac12}
\right)
.
\end{equation}
\end{theorem}
As an immediate consequence of this theorem, we can also bound the inverse arc length. Indeed, using $\gamma_1(0)=0$, we can bound the diameter of $\Sigma_\gamma$ by estimating
$$|f_\gamma(t,\theta)-f_\gamma(t',\theta')|
\leq |\gamma(t)-\gamma(t')|+|\gamma_1(t')|\ \left|\begin{pmatrix}
\cos\theta-\cos\theta'\\
\sin\theta-\sin\theta'
\end{pmatrix}
\right|
\leq 3L[\gamma].$$
Consequently $\operatorname{diam}(\Sigma_\gamma)\leq 3L[\gamma]$ and thus
\begin{equation}\label{inverseLbound}
\frac1{L[\gamma]^2}\leq\frac{2\pi\operatorname{diam}(\Sigma_\gamma)}{A[\gamma]L[\gamma]}\leq \frac{6\pi}{A[\gamma]}.
\end{equation}

Additionally, many facts about smooth surfaces remain true for surfaces $\Sigma_\gamma$ with $\gamma\in\mathcal P$. We collect these below and refer to Appendix \ref{proofpropertiesofP} for their proofs. 

\begin{lemma}[Li-Yau Inequality]\label{WillmoreLowerBound}\ \\
Let $\gamma\in\mathcal P$. Then $\mathcal W[\Sigma_\gamma]\geq 4\pi$ and equality is only true if $\Sigma_\gamma$ is a sphere and $f_\gamma:\Sp^2\rightarrow \Sigma_\gamma$ is a diffeomorphism. 
\end{lemma}

A direct consequence from Lemma \ref{WillmoreLowerBound} is $\beta(\sigma)\geq 4\pi$ for all $\sigma\in(0,1]$. Additionally, since $\mathcal W[\Sp^2]=4\pi$ we deduce $\beta(1)=4\pi$ and the minimizers are precisely rounds spheres.

\begin{lemma}[Gauß-Bonnet Theorem]\label{gausbonnetlemma}\ \\
Let $\gamma\in\mathcal P$. Then $\Sigma_\gamma$ satisfies the Gauß-Bonnet theorem. That is
\begin{equation}\label{gaussbonnettheorem}
\int_{\Sp^2}k_1k_2d\mu_{\Sigma_\gamma}=4\pi.
\end{equation}
\end{lemma}

\begin{lemma}[Hopf Theorem]\label{hopftypetheorem}\ \\
Let $\gamma\in \mathcal P$ and assume that $\Sigma_\gamma$ has constant mean curvature $H_0$. Then $\Sigma_\gamma$ is a sphere.
\end{lemma}

\begin{lemma}\label{gamma1ddotgammainL1}
    Let $\gamma\in\mathcal P^w$, $L:=L[\gamma]$ and $U\subset[0,1]$ be measurable such that $\inf_U \gamma_1>0$. Then $\gamma_1|\ddot\gamma|^2\in L^1((0,1))$, $|\ddot\gamma|^2\in L^1(U)$ and
    $$\int_0^1 \gamma_1(t)|\ddot\gamma(t)|^2dt\leq \frac{L^3}{2\pi}\int_0^1 k_1^22\pi L\gamma_1(t) dt
    \hspace{.5cm}\textrm{and}\hspace{.5cm}
    \int_U|\ddot\gamma(t)|^2dt\leq \frac{2L^3}{\pi\inf_U\gamma_1}\mathcal W[\gamma].$$
\end{lemma}
\begin{proof}
    In view of Theorem \ref{weakandstrongPconnection}, we can assume that $\gamma\in\mathcal P$ to estblish the first esimtate. By Definition \ref{classPdef} we have $\gamma\in W^{2,2}(\epsilon,1-\epsilon)$ for all small $\epsilon>0$. We use the definition of $\mathcal P$ and Equation \eqref{ArcLength_K_Identity} to deduce that for some $C$ independent of $\epsilon$
    $$C\geq \int_\epsilon^{1-\epsilon}k_1^2 2\pi|\dot\gamma|\gamma_1 dt=\frac{2\pi}{L^3}\int_\epsilon^{1-\epsilon}|\ddot\gamma|^2\gamma_1 dt.$$
    The first part of the lemma follows by letting $\epsilon\rightarrow 0^+$ and using the theorem of Beppo-Levi. Let $\gamma\in\mathcal P^w$. To establish the second estimate, it suffices to bound $\int_0^1 k_1^2 2\pi L\gamma_1dt$. By Theorem \ref{weakandstrongPconnection}, there exists $m\in\N$ such that $\gamma$ is built from chaining together $m$ many $\mathcal P$-curves. We use the Gauß-Bonnet theorem (Lemma \ref{gaussbonnettheorem}) to obtain $\int_0^1 k_1k_22\pi|\dot\gamma|\gamma_1 dt=4n_0\pi\geq 0$ and estiamte 
    $$\int_0^1 k_1^2 2\pi |\dot\gamma|\gamma_1 dt\leq 4\mathcal W[\gamma]-2\int_0^1 k_1k_22\pi |\dot\gamma|\gamma_1 dt\leq 4\mathcal W[\gamma]. $$
\end{proof}

\subsection{Topology}
We now introduce a notion of convergence in $\mathcal P^w$. The following is essentially the same definition as Definition 6 in \cite{choskiveneroni}. Their definition, however, allows for systems of curves, which we exclude here. Additionally, we have eliminated the concept of measure-function couples. 
\begin{definition}\label{convergencedefinition}
We say that a sequence $(\gamma^ {(n)})\subset \mathcal P^w$ \emph{converges to $\gamma^*\in \mathcal P^w$ in $\mathcal P^w$} if:
\begin{enumerate}[(1)]
    \item For $i=1,2$ we have $\gamma^ {(n)}_i\rightarrow\gamma^*_i$ in $C^ 0([0,1])$.
    \item For $i=1,2$ we have $\dot\gamma^ {(n)}_i\rightarrow\dot\gamma^*_i$ in $L^ 2((0,1))$.
    \item For $i=1,2$ we have $\ddot\gamma^ {(n)}_i\gamma_1^{(n)}\rightarrow\ddot\gamma^*_i\gamma_1^*$ weakly in $L^ 1((0,1))$.
\end{enumerate}
\end{definition}
We note that the third condition is well-defined by Lemma \ref{gamma1ddotgammainL1}. 

\paragraph{Lower Semicontinuity}\ \\
We show that the notation of convergence in Definition \ref{convergencedefinition} is strong enough to imply lower semicontinuity of our problem. Let $\gamma^ {(n)}\rightarrow \gamma^*$ in $\mathcal P^w$. Then $\gamma^ {(n)}\rightarrow\gamma^*$ in $C^ 0([0,1])$ and $\dot\gamma^ {(n)}\rightarrow\dot\gamma^*$ in $L^2((0,1))$. Hence 
\begin{align}
    A[\gamma^{(n)}]&=2\pi\int_0^1|\dot\gamma^ {(n)}(t)|\gamma_1^ {(n)}(t) dt\rightarrow A[\gamma^*],\label{areaconv}\\
    V[\gamma^{(n)}]&=\pi\left|\int_0^1\dot\gamma^ {(n)}_2(t)\left(\gamma_1^ {(n)}(t)\right)^2 dt\right|\rightarrow V[\gamma^*].\label{volconv}
\end{align}
In particular $\mathcal I[\gamma^{(n)}]\rightarrow\mathcal I[\gamma^*]$. So, the area, volume and isoperimetric ratio are actually continuous with respect to the convergence from Definition \ref{convergencedefinition}. Moreover, in Lemmas \ref{lowersemiconLemma} and \ref{lemmaprinciplacuvaturesconv} in the appendix we prove
\begin{align}
&\mathcal W[\gamma^*]\leq\liminf_{n\rightarrow\infty}\mathcal W[\gamma^{(n)}],\label{Willmorelowersemicont}\\
&\int_0^1\left((k_1^* )^2+(k_2^* )^2\right)2\pi|\dot\gamma^*|\gamma_1^*dt
\leq\liminf_{n\rightarrow\infty}\int_0^1 \left((k_1^ {(n)})^2+(k_2^ {(n)})^2\right)2\pi|\dot\gamma^{(n)}|\gamma_1 ^{(n)}dt.\label{curvsqlowersemicont}
\end{align}

\subsection{Compactness}
The following compactness theorem is essentially due to Choski's and Veneroni's paper \cite{choskiveneroni} (see Subsection 3.4). However, as we have slightly changed the definition of convergence in $\mathcal P^w$, their proof must also be slightly modified. For the reader's convenience, we provide the proof in Appendix \ref{compactnessappendix}.
\begin{theorem}[Compactness Theorem]\label{compactnesstheorem}
Let $(\gamma^{(n)})\subset\mathcal P^w$ such that:
\begin{enumerate}[(1)]
    \item There exists $R>0$ such that $\gamma^ {(n)}(t)\in B_R(0)$ for all $t\in[0,1]$ and $n\in\N$.
    \item There exists $C<\infty$ such that 
    $$\sup_n \int_0^1\left[ \left(k_1^{(n)}\right)^2+\left(k_2^{(n)}\right)^2\right]2\pi\gamma_1^ {(n)}|\dot\gamma^ {(n)}|dt\leq C.$$
    \item There exist $\theta>0$ and $C<\infty$ such that  $0<\theta\leq A[\gamma^ {(n)}]\leq C$ for all $n\in\N$.
\end{enumerate} 
Then there exists $\gamma^*\in \mathcal P^w$  and a subsequence $\gamma^ {(n_k)}$ such that $\gamma^ {(n_k)}\rightarrow\gamma^*$ in $\mathcal P^w$.
\end{theorem}

\subsection{Existence of a Weak Minimizer}
To prove the existence of a minimizer $\gamma\in\mathcal F_\sigma$ satisfying $\mathcal W[\gamma]=\beta(\sigma)$, we require the following theorem, that is essentially due to Schygulla \cite{schygulla}.
\begin{theorem}[Schygulla]\label{schygulla}\ \\
 Let $\beta$ be as in Equation \eqref{betaDefinition}. Then $\beta(\sigma)<8\pi$ for all $\sigma\in(0,1]$.
\end{theorem}
Schygulla proves this for surfaces without requiring axial symmetry. Therefore Theorem \ref{schygulla} cannot be cited directly from his paper. However, inspecting his proof, it is straightforward to see that it can easily be adapted for our purposes. We comment on the necessary changes in Appendix \ref{schygullapp}.\\

Using Theorem \ref{schygulla}, we can now prove the existence of a weak minimizer.
\begin{theorem}[Existence of Weak Minimizer]\label{weakminimizerexistence}\ \\
For all $\sigma\in(0,1]$ there exists $\gamma\in\mathcal F_\sigma$ such that $\beta(\sigma)=\mathcal W[\gamma]$.
\end{theorem}
\begin{proof}
For $\sigma=1$, any round sphere is a minimizer. For $\sigma\in(0,1)$ let $\gamma^{(n)}\subset \mathcal P$ such that $\mathcal W[\gamma^{(n)}]\rightarrow\beta(\sigma)$. As $\beta(\sigma)<8\pi$, we can assume without loss of generality that $\mathcal W[\gamma^{(n)}]\leq 8\pi-\rho_0$ for all $n\in\N$ and some $\rho_0>0$. The Willmore energy $\mathcal W$ and the isoperimetric ratio $\mathcal I$ are invariant under translations. So, by defining
$$\tilde\gamma^{(n)}(t):=\gamma^{(n)}(t)-(0,\gamma^{(n)}_2(0)),$$
we can further assume that $\gamma^{(n)}(0)=0$ for all $n\in\N$. Let $L_n:=L[\gamma^{(n)}]$. For every individual $\gamma^{(n)}$, the Gauß-Bonnet theorem (see Lemma \ref{gausbonnetlemma}) holds and hence
\begin{align}
&\int_0^1\left(\left(k_1^{(n)}\right)^2+\left(k_2^{(n)}\right)^2\right)2\pi L_n \gamma_1^{(n)}(t)dt\nonumber\\
=&4\mathcal W[\gamma^{(n)}]-2\int_{0}^1k_1^{(n)}k_2^{(n)}2\pi L_n \gamma_1^{(n)}(t)dt\nonumber\\
\leq & 4\cdot(8\pi-\rho_0)-8\pi.\label{compactnesssecondassumpion}
\end{align}
In view of Theorem \ref{choskiveneroniresult}, we get $L_n\leq C$. For any $t\in[0,1]$ and $n\in\N$ we can now estimate 
$$|\gamma^{(n)}(t)|=|\gamma^{(n)}(t)-\gamma^{(n)}(0)|\leq \int_0^1 |\dot\gamma^{(n)}(t)|dt\leq C.$$
Finally $A[\gamma^ {(n)}]\equiv 1$. Thus, all assumptions from the compactness theorem (Theorem \ref{compactnesstheorem}) are satisfied and we can extract a subsequence, again denoted by $\gamma^{(n)}$, that converges to a limit curve $\gamma^*\in\mathcal P^w$. Let $L_*:=L[\gamma^*]$. We will now prove that $\gamma^*\in\mathcal P$. Using the lower semicontinuity results from Equations \eqref{Willmorelowersemicont} and \eqref{curvsqlowersemicont} and Estimate \eqref{compactnesssecondassumpion}, we get 
\begin{align} 
\mathcal W[\gamma^*]\leq\beta(\sigma)
\hspace{.5cm}\textrm{and}\hspace{.5cm}
\int_0^1\left(\left(k_1^{*}\right)^2+\left(k_2^{*}\right)^2\right)2\pi L_* \gamma_1^{*}(t)dt&<3\cdot 8\pi.\label{leq38pi}
\end{align}
So, by Theorem \ref{weakandstrongPconnection}, there is at most one $\tau\in(0,1)$ such that $\gamma^*_1(\tau)=0$. Suppose that such a $\tau$ exists. Then, by Theorem \ref{weakandstrongPconnection}, $\gamma^*$ restricted to $[0,\tau]$ and $[\tau,1]$, up to reparameterization, belong to $\mathcal P$. Applying the Li-Yau-inequality (Lemma \ref{WillmoreLowerBound}) produces the contradiction
\begin{equation}\label{reallyinP}
8\pi-\rho_0\geq \mathcal W[\gamma^*]
=
\mathcal W[\gamma^*|_{[0,\tau]}]+\mathcal W[\gamma^*|_{[\tau,1]}]
\geq 2\cdot 4\pi.
\end{equation}
This shows $\gamma_{1}^*(t)>0$ for all $t\in(0,1)$ and hence $\gamma^*\in\mathcal P$. Using the continuity of the area and volume functional from Equations \eqref{areaconv} and \eqref{volconv}, we deduce $\gamma^*\in\mathcal F_\sigma$ and therefore $\mathcal W[\gamma^*]\geq \beta(\sigma)$. In view of Estimate \eqref{leq38pi} we get $\mathcal W[\gamma^*]=\beta(\sigma)$ and hence $\gamma^*$ is a minimizer.
\end{proof}
\section{Regularity of Minimizers}\label{regularityawayfromaxis}
In this section, we prove that the weak minimizers from Theorem \ref{weakminimizerexistence} are smooth.
\begin{theorem}\label{regularitytheorem1}
Let $\sigma\in(0,1]$ and $\gamma\in\mathcal F_\sigma$ such that $\beta(\sigma)=\mathcal W[\gamma]$. Then $\Sigma_\gamma$ is a smooth surface. 
\end{theorem}
For $\sigma=1$, we have already shown that the minimizers in $\mathcal P$ are round spheres. So it suffices to consider the case $\sigma\in(0,1)$. The proof is split into two parts. In the first step, we prove that the profile curve $\gamma$ is smooth on $(0,1)$. This shows that $\Sigma_\gamma$ is smooth away from the axis of rotation. In the second step, we consider the two intersection points $\gamma_2(0)e_3$ and $\gamma_2(1)e_3$ of $\Sigma_\gamma$ with the axis of rotation and show that here locally $\Sigma_\gamma$ is a smooth graph. 

\subsection{Euler-Lagrange Equation}
Our main tool for improving the regularity of $\Sigma_\gamma$ is a suitable Euler-Lagrange equation. However, deriving this is a little involved when we restrict ourselves to the class $\mathcal P$ due to the many restrictions imposed on curves $\gamma\in\mathcal P$. 
\begin{definition}[Generalized Profile Curves]
We say that $\gamma:[0,1]\rightarrow \mathcal H^2$ is a \emph{generalised profile curve} if
\begin{enumerate}[(1)]
    \item $\gamma\in C^1([0,1])\cap W^ {2,2}_{\operatorname{loc}}((0,1))$,
    \item $\gamma_1(0)=\gamma_1(1)=0$ and $\gamma_1(t)>0$ for all $t\in(0,1)$,
    \item $\Sigma_\gamma$ has curvature bounded in $L^2(\Sigma_\gamma)$: $\int_{\Sp^2} k_1^2+k_2^2d\mu_{\Sigma_\gamma}<\infty$. 
\end{enumerate}
We denote the set of all generalized profile curves by $\mathcal P'$. 
\end{definition}
Clearly $\mathcal P\subset\mathcal P'$. 
We consider the isoperimetric ratio functional
$$\mathcal I:\mathcal P'\rightarrow\R_0^+,\ \mathcal I[\gamma]:=6\sqrt\pi\frac{V[\gamma]}{A[\gamma]^{\frac32}}
\hspace{.5cm}\textrm{ and put }\hspace{.5cm}\
\mathcal F'_\sigma:=\set{\gamma\in\mathcal P'\ |\ \mathcal I[\gamma]=\sigma}.$$
As both the Willmore energy $\mathcal W$ and the isoperimetric ratio $\mathcal I$ are invariant under scaling and reparameterization, we have 
$$
\beta(\sigma)
=
\inf\left\{\mathcal W[\gamma]\ |\ \gamma\in\mathcal F_\sigma\right\} 
=
\inf\left\{\mathcal W[\gamma]\ |\ \gamma\in\mathcal F'_\sigma\right\} .
$$
Consequently, given a minimizer $\gamma\in\mathcal F_\sigma$ and a variation $\Phi:(-\epsilon_0,\epsilon_0)\rightarrow\mathcal F'_\sigma$ satisfying $\Phi(0)=\gamma$ we have 
\begin{equation}\label{eqeq0willmorevar}
0=\frac d{d\epsilon}\bigg|_{\epsilon=0}\mathcal W[\Phi(\epsilon)]=\delta\mathcal W[\gamma]\Phi'(0).
\end{equation}
\paragraph{Variation Formulas}\ \\
In the following, we require the variation of the isoperimetric ratio functional. Let $\gamma\in\mathcal P'$ with $V[\gamma]>0$ and $\phi\in C^\infty([0,1])$, not necessarily vanishing at $t=0,1$. Using the definition of $\mathcal I$ as well as Equations \eqref{areadef} and \eqref{volumedef} it is easy to see that
\begin{align}
    \delta\mathcal I[\gamma]\phi&=\frac{6\sqrt\pi}{A[\gamma]^{\frac32}}
    \left[
    \delta V[\gamma]\phi-\frac32\frac{V[\gamma]}{A[\gamma]}\delta A[\gamma]\phi
    \right],\label{var1}\\
    \delta V[\gamma]\phi&=\pm_\gamma\pi\int_0^1 \gamma_1^2(t)\dot\phi_2+2\gamma_1(t)\dot\gamma_2(t)\phi_1(t)dt,\label{var2}\\
\delta A[\gamma]\phi&=2\pi\int_0^1|\dot\gamma(t)|\phi_1(t)+\frac{\gamma_1(t)\langle\dot\phi(t),\dot\gamma(t)\rangle}{|\dot\gamma(t)|}dt.\label{var3}
\end{align}
The $\pm_\gamma$ in Equation \eqref{var2} is the sign of the integral in Equation \eqref{volumedef}. Now suppose that $\phi\in C^\infty_0((0,1))$. Then, as $\gamma\in W^{2,2}_{\operatorname{loc}}((0,1))$, we can integrate the integrals in Equations \eqref{var2} and \eqref{var3} by parts. Using $\nu$ from Equation \eqref{normalsprofilecurve} we get
\begin{align}
    \delta V[\gamma]\phi&=\mp_\gamma 2\pi\int_0^1 \langle \nu,\phi\rangle|\dot\gamma|\gamma_1 dt,\label{var4}\\
    \delta A[\gamma]\phi&=-2\pi\int_0^1 H[\gamma]\langle \nu,\phi\rangle |\dot\gamma|\gamma_1 dt.\label{var5}
\end{align}
The following lemma proves that the constraint of prescribed isoperimetric ration $\sigma\neq 1$ is non-degenerate.
\begin{lemma}\label{psi0existencelemma}
Let $\gamma\in\mathcal P$. Either $\Sigma_\gamma$ is a sphere or there exists $\psi_0\in C_0^\infty((0,1))$ such that 
$$\frac d{d\epsilon}\bigg|_{\epsilon=0}\mathcal I[\gamma+\epsilon\psi_0]=1.$$
\end{lemma}
\begin{proof}
If no $\phi\in C_0 ^ \infty((0,1))$ satisfying $\delta \mathcal I[\gamma]\phi\neq 0$ exists, then we can use Equations \eqref{var1}, \eqref{var4} and \eqref{var5} to get
$$\pm_\gamma2\pi\int_0^1 \langle \nu,\phi\rangle |\dot\gamma|\gamma_1 dt
=
2\pi\frac{3 V[\gamma]}{2A[\gamma]}\int_0^1 H[\gamma]\langle \nu,\phi\rangle |\dot\gamma|\gamma_1 dt \hspace{.5cm}\textrm{for all }\phi\in C^ \infty_0((0,1)).$$
By the fundamental lemma of the calculus of variations analysis $H\equiv \pm_\gamma\frac{2A[\gamma]}{2V[\gamma]}$ and by Lemma \ref{hopftypetheorem} -- the Hopf theorem -- $\Sigma_\gamma$ is a sphere.
\end{proof}

\subsection{Regularity Away from the Axis of Rotation}
Let $\sigma\in(0,1)$ and $\gamma\in\mathcal F_\sigma$ such that $\mathcal W[\gamma]=\beta(\sigma)$. We derive a suitable Euler-Lagrange equation for $\gamma$ on $(0,1)$. Using Lemma \ref{psi0existencelemma}, we choose $\psi_0\in C_0^\infty((0,1))$ such that $\delta\mathcal I[\gamma]\psi_0=1$. For any smooth $\eta\in C_0^\infty((0,1))$ we consider the vector field 
$$\tilde\eta:=\eta-\alpha(\eta)\psi_0\hspace{.5cm}\textrm{with}\hspace{.5cm}\alpha(\eta):=\delta\mathcal I[\gamma]\eta.$$
In Lemma \ref{variationexistence} we prove that for any such $\eta$ there exists a variation $\Phi_\eta:(-\epsilon_0,\epsilon_0)\rightarrow\mathcal F_\sigma'$ such that $\Phi_\eta'(0)=\tilde\eta$. So, by Equation \eqref{eqeq0willmorevar} we get 
\begin{equation}\label{weakeulerlagrangefirstappe}
0=\frac d{d\epsilon}\bigg|_{\epsilon=0}\mathcal W[\Phi_\eta(\epsilon)]=\delta\mathcal W[\gamma]\tilde\eta
\hspace{.5cm}\textrm{and hence}\hspace{.5cm}\delta\mathcal W[\gamma]\eta=\left(\delta\mathcal W[\gamma]\psi_0\right)\delta\mathcal I[\gamma]\eta.
\end{equation}

Using this Euler-Lagrange equation, we can now derive regularity of $\gamma$ as long as we \emph{`stay away from the axis of rotation'}. To do so, we require a general regularity result, which we have moved to Appendix \ref{GeneralRegularityAppendix}.

\begin{korollar}\label{gammasmooththeorem1}
    Let $\sigma\in(0,1)$, $\gamma\in\mathcal F_\sigma$ such that $\mathcal W[\gamma]=\beta(\sigma)$, $0<\tau_1<\tau_2<1$ and $U:=(\tau_1,\tau_2)$. Assume that $\kappa_1(\gamma;U)=\inf_U\gamma_1>0$
    and that there exists $\psi_0\in C^2_0(U)$ such that $\delta\mathcal I[\gamma]\psi_0=1$.
    Then $\gamma\in C^\infty(U)$ and for all $m\in\N$
    $$\|\gamma-\gamma(0)\|_{C^m((\tau_1,\tau_2))}\leq C(m,\kappa(\gamma;U),\|\psi_0\|_{C^2(U)}).$$
    The constant is decreasing in the second slot. 
\end{korollar}
\begin{proof}
    Without loss of generality, we assume $\gamma(0)=0$. We wish to apply Theorem \ref{GeneralRegularityTheorem}. To do so, we first establish an appropriate differential equation. Using Equation \eqref{weakeulerlagrangefirstappe}, we get
    $$\delta\mathcal W[\gamma]=\left(\delta\mathcal W[\gamma]\psi_0\right)\delta\mathcal I[\gamma](\phi)=I(\phi)\hspace{.5cm}\textrm{for all }\phi\in C^\infty_0(U,\R^2).$$
    Here $I(\phi)$ is an in Equation \eqref{IFUnctionalPDE} with parameters $\alpha$ and $\beta$ that, in view of Equations \eqref{var1}, \eqref{var2} and \eqref{var3}, are given as
    $$\alpha=\pm_\gamma\frac{6\sqrt\pi^3}{A[\gamma]^{\frac32}}\left(\delta\mathcal W[\gamma]\psi_0\right)
    \hspace{.5cm}\textrm{and}\hspace{.5cm}
    \beta=-\frac{3\pi}{A[\gamma]}\frac{6\sqrt\pi V[\gamma]}{A[\gamma]^{\frac32}}\left(\delta\mathcal W[\gamma]\psi_0\right).$$
    Next, we collect the appropriate bounds to satisfy assumption \eqref{uniformestimate}.  We have $\mathcal W[\gamma]=\beta(\sigma)\leq 8\pi$. Using $A[\gamma]=1$, Theorem \ref{choskiveneroniresult} and Estimate \eqref{inverseLbound}, we get $L[\gamma]+L[\gamma]^{-1}\leq C$. Combining the bound for $L[\gamma]$ with Lemma \ref{gamma1ddotgammainL1}, we have $\|\ddot\gamma\|_{L^2(U)}\leq C(\kappa(\gamma;U),U)$.
     Since $|\dot\gamma|=L[\gamma]\leq C$ and we assume $\gamma(0)=0$, we obtain $\|\gamma\|_{C^0}\leq C$.\\

    To estimate $\alpha$ and $\beta$ we recall $A[\gamma]=1$ and $V[\gamma]=\frac\sigma{6\sqrt\pi}$. It remains to estimate  $|\delta\mathcal W[\gamma]\psi_0|$. Considering Equations \eqref{principalcurvatures} and \eqref{willmoredef} it is easy to see that 
    $$|\delta\mathcal W[\gamma]\psi_0|\leq C(\|\gamma\|_{W^{1,\infty}(U)},\|\ddot\gamma\|_{L^2(U)},\kappa(\gamma;U))\|\psi_0\|_{C^2(U)}.$$
    In fact, we essentially establish this estimate in Step 2 in the proof of Lemma \ref{regularitystep1}.\\

    Summarising the above, we see that we can satisfy Estimate \eqref{uniformestimate} with a parameter $M$ that depends only on $U$, $\kappa(\gamma;U)$ and $\|\psi_0\|_{C^2(U)}$. So, Theorem \ref{GeneralRegularityTheorem} implies the corollary. 
\end{proof}

Combining Lemma \ref{psi0existencelemma}, $\gamma_1(t)>0$ for all $t\in(0,1)$ and Corollary \ref{gammasmooththeorem1}, we immediately get:

\begin{korollar}\label{gammasmooththeorem}
    Let $\sigma\in(0,1)$ and $\gamma\in\mathcal F_\sigma$ such that $\mathcal W[\gamma]=\beta(\sigma)$. Then $\gamma\in C^\infty((0,1))$. 
\end{korollar}

\subsection{Regularity At the Axis of Rotation}\label{RegularityAtAxisSubsection}
We continue to fix $\sigma\in(0,1)$ and $\gamma\in\mathcal F_\sigma$ such that $\mathcal W[\gamma]=\beta(\sigma)$. By Definition \ref{classPdef} and Corollary \ref{gammasmooththeorem} we have $\gamma\in C^\infty((0,1))\cap C^1([0,1])$. The main ingredient in the discussion of the regularity of $\Sigma_\gamma$ near the axis of rotation is the study of an appropriate Euler-Lagrange equation:

\begin{lemma}[Euler-Lagrange Equation]\label{eulerlagrangeequation}
There exists $\Lambda\in\R$ such that $\gamma$ satisfies the following Euler-Lagrange equation on $(0,1)$:
$$\Delta_g H+ \frac12H(H^2-4K)=\Lambda(\pm_\gamma4\sqrt\pi-\sigma H)$$
\end{lemma}
\begin{proof}
Let $\phi\in C^\infty_0((0,1))$ and $\psi_0$ be as in Lemma \ref{psi0existencelemma}. Then $\delta\mathcal W[\gamma](\phi-\alpha(\phi)\psi_0)=0$ by Equation \eqref{weakeulerlagrangefirstappe} . As $\phi$ is supported in $(0,1)$ where $\gamma$ is smooth, we can use Equation \eqref{variationformulastandard} for the variation of Willmore energy from the appendix to get 
$$\delta \mathcal W[\gamma]\phi=2\pi\int_0^1 W[\gamma]\langle\phi,\nu\rangle\gamma_1 |\dot\gamma|dt.$$
We put $\Lambda:=-3(\delta\mathcal W[\gamma]\psi_0)$ and recall from Equations \eqref{var1}, \eqref{var4} and \eqref{var5} that 
$$\alpha(\phi)=\delta\mathcal I[\gamma]\phi=-2\pi\sigma \int_0^1 \left(\pm_\gamma\frac1{V[\gamma]}-\frac32\frac{H}{A[\gamma]}\right)\langle\phi,\nu\rangle \gamma_1|\dot\gamma|ds.$$
Using $A[f]=1$, $V[f]=\frac\sigma{6\sqrt\pi}$, the explicit formula for the Willmore operator as well as the fundamental lemma of the calculus of variations, we get 
$$\frac12\left(\Delta_g H+\frac12H(H^2-4K)\right)=W[\gamma]=\frac13\Lambda \sigma \left(\pm_\gamma\frac{6\sqrt\pi}\sigma-\frac32 H\right).$$

\end{proof}
 To continue the investigation of the regularity of $\Sigma_\gamma$, we will now introduce a more suitable description of the surface $\Sigma_\gamma$ close to $(0,0,\gamma_2(0))$. 

\paragraph{Graph Representation near $(0,0,\gamma_2(0))$}\ \\
Near $(0,0,\gamma_2(0))$ we can write the surface $\Sigma_\gamma$ as a graph 
$$(r\cos\theta,r\sin\theta, u(r))\hspace{.5cm}\textrm{where}\hspace{.5cm} r\in[0, r_0),\ \theta\in[0,2\pi).$$
Indeed, putting $r(t):=\gamma_1(t)$ we have $r\in C^1([0,1])$ and $r'(0)=L[\gamma]\neq 0$. By the inverse function theorem, there exists $t_0>0$ such that $\gamma_1:[0, t_0]\rightarrow\gamma_1([0,t_0])$ is a $C^1$-diffeomorphism and we may define
$$u:[0, r_0)\rightarrow\R,\ u(r):=\gamma_2(\gamma_1^ {-1}(r))\hspace{.5cm}\textrm{where }r_0=\gamma_1(t_0).$$
Since $\gamma\in C^\infty((0, t_0])$, we get $u\in C^1([0, r_0])\cap C^ \infty((0, r_0])$. Additionally, $\dot\gamma_2(0)=0$ implies $u'(0)=0$.\\

Throughout the rest of this subsection, we fix $u$ as above.\\

We now wish to write the Euler-Lagrange equation from Lemma \ref{eulerlagrangeequation} in the graph parameterization. Using a result from \cite{dziuk2006error}, the authors of \cite{chenODE} derive an elegant formula for the Willmore operator. Letting $w:=u'(r)$ and $v:=\sqrt{1+w^2}$ we deduce that $w$ solves the following ODE on $(0, r_0)$:
\begin{equation}\label{ODE}
\frac1r\frac d{dr}\left[ r\frac1{v^ 5}\left(w''+\left(\frac{w}{r}\right)'-\varphi(w)\right)
\right]=W[\gamma]=
a- b H
\end{equation}
In view of Lemma \ref{eulerlagrangeequation} $a$ and $b$ are explicitly given by $a=\pm_\gamma4\sqrt\pi\Lambda$ and $b=\sigma\Lambda$. Furthermore, $\varphi(w)$ is computed in Lemma 2.2 in \cite{chenODE} and given by
\begin{equation}\label{varphiofwdef}
\varphi(w):=\frac{5w}{2(1+w^2)}(w')^2+\frac{w^ 3}{2r^2}(3+w^2).
\end{equation}
Note from Equation \eqref{graphH}, that $H$ is also of the form $r^{-1}\partial_r(...)$. This allows us to reduce the order of the ODE in Equation \eqref{ODE}. Indeed, multiplying Equation \eqref{ODE} by $r$, inserting Equation \eqref{graphH} for $H$, integrating and finally multiplying by $\frac{v^ 5}r$ shows that there exists $\lambda\in\R$ such that
\begin{equation}\label{ODEamRand}
\left\{
\begin{aligned}
&w''+\left(\frac{w}{r}\right)'-\varphi(w)
=
\frac12 v^ 5 a r+bv^ 4w+\frac\lambda rv^ 5,\\
&w(0)=0.
\end{aligned}
\right.
\end{equation}

\paragraph{Regularity of $u$}\ \\
$u$ satisfies Equation \eqref{ODEamRand} on $(0,r_0)$. We note that Equation \eqref{ODEamRand} allows for nonsmooth solutions. For example, the first remark in Section 3 in \cite{removability} discusses that the inverted catenoid solves \eqref{ODEamRand} on $(0, r_0)$ with $a=b=0$. It is also shown that the graph function $u_{ic}$ describing the inverted catenoid satisfies $u_{ic}\in C^ {1,\alpha}([0,r_0])$ for all $\alpha\in(0,1)$ and $u_{ic}''\in L^p((0, r_0))$ for all $p\in[1,\infty)$ but $u_{ic}\not\in C^{1,1}((0,r_0))$. In the next lemma, we prove that our solution $u$ possesses at least the same regularity.\\

The following lemma is essentially the same as Lemma 3.3 in \cite{chenODE}. Here the same analysis is carried out for the unconstrained Willmore equation. We show that the same argument can be applied when including the Lagrange multiplier in Equation \eqref{ODEamRand}.\\

\begin{lemma}\label{regularityatend01}
There exists $\xi\in C^1([0, r_0])$ such that $u'(r)=\frac12\lambda r\ln(r)+\xi(r)$. In particular, $u'\in C^{0,\alpha}([0,r_0])$ and $u''\in L^ p((0, r_0))$ for all $\alpha\in (0,1)$ and $p\in[1,\infty)$.
\end{lemma}
\begin{proof}
The proof is separated into six steps.\\

\noindent
\textbf{Step 1}\ \\
Note  $w\varphi(w)\geq 0$ by the definition of $\varphi(w)$ in Equation \eqref{varphiofwdef}.  We multiply Equation \eqref{ODEamRand} by $w$, choose any $\rho\in (0, r_0)$  and integrate:
$$\int_r ^ \rho w''w+\left(\frac wt\right)' w dt
\geq \frac a2 \int_ r^ \rho v^ 5 t w dt+
b\int_r^ \rho v^ 4 w^2 dt+\lambda \int_r ^ \rho \frac{v^ 5}{t} w dt$$
We already know that $u$ is $C^1([0, r_0))$ and satisfies $u'(0)=0$. Hence $v$ and $w$ are bounded and we get the estimate  
$$\int_r ^ \rho w''w+\left(\frac wt\right)' w dt
\geq 
-C\left(1+\int_r^ \rho\frac{|w|}tdt\right).$$
$C$ depends on $\sigma$, $\lambda$ and on the Lagrange multiplier. We will not make this dependence explicit in the notation. Integrating by parts gives 
$$-w'(r)w(r)-\frac{w(r)^2}r-\int_r^ \rho (w'(t))^2+\frac{ ww'}{t}dt
\geq 
-C(\rho)\left(1+\int_r^ \rho\frac{|w|}tdt\right).$$
We have included the boundary terms at $t=\rho$ in the constant $C(\rho)$. Note that $C(\rho)<\infty$ for all $\rho\in(0, r_0)$ as $w\in C^\infty((0, r_0))$. Next, we integrate by parts again and get
$$-w'(r)w(r)-\frac{w(r)^2}r+\frac{w(r)^2}{2r}-\int_r^ \rho (w'(t))^2+\frac{ w(t)^2}{2t^2}dt
\geq 
-C(\rho)\left(1+\int_r^ \rho\frac{|w|}tdt\right).$$
Rearranging this estimate gives 
\begin{align}
\int_r^ \rho (w'(t))^2+\frac{ w(t)^2}{2t^2}dt
&\leq C(\rho)\left(1+\int_r^ \rho\frac{|w|}tdt\right)+\frac{w(r)^2}{2r}-w(r)\frac{w(r) +rw'(r)}r\nonumber\\
&= C(\rho)\left(1+\int_r^ \rho\frac{|w|}tdt\right)+\frac{w(r)^2}{2r}-w(r)\frac{(r w(r))'}r.\label{step1result}
\end{align}

\noindent
\textbf{Step 2}\ \\
Let again $\rho\in(0,r_0)$. This time we integrate Equation \eqref{ODEamRand} directly and use $w'+\frac wr=\frac1r(wr)'$ to get 
\begin{equation}\label{step2begin}
\frac {(w t)'}t\bigg|_{t=\rho}-\frac {(w r)'}r=\int_ r^ \rho \varphi(w)+\lambda\frac{ v^ 5-1}t+\frac12 a v^ 5 t + b v^ 4wdt-\lambda\ln\frac r\rho.
\end{equation}
Let $\epsilon>0$. Since $w=u'\in C^0([0, r_0))$ and $w(0)=0$, we can choose $\rho(\epsilon)$ small enough such that $|w(r)|\leq \epsilon$ for all $r\in[0,\rho(\epsilon)]$. For $\rho\leq\rho(\epsilon)$ we get 
\begin{align*}
    \left|\frac{(r w(r))'}{r}\right| 
    &\leq C(\epsilon)+\left|\lambda\ln\frac r\rho\right|+\int_r^ \rho
    \frac{5 |w(t)| w'(t)^2}{2(1+w(t)^2)}+w(t)^2\frac{3|w(t)|+|w(t)|^3}{2 t^2}+|\lambda|\left|\frac{v^ 5-1}{t}\right|dt\\
    &\hspace{.5cm}+C\int_r^ \rho t v^ 5+ v^ 4|w| dt.
\end{align*}
The constant $C(\epsilon)$ is due to the boundary terms at $\rho(\epsilon)$. The second constant $C$ does not depend on $\epsilon$. Using $u'(0)=w(0)=0$, we can shrink $\rho(\epsilon)$ to achieve
\begin{align}
    \left|\frac{(r w(r))'}{r}\right| 
    &\leq C(\epsilon)\ln\frac1r+\epsilon \int_r^ \rho w'(t)^2 +\frac{w(t)^2}{2t^2} dt\hspace{.5cm}\textrm{for all }r\in(0, \rho(\epsilon)).\label{step2result}
\end{align}

\noindent
\textbf{Step 3}\ \\
Let $\epsilon>0$ be arbitrary. By taking $\rho=\rho(\epsilon)$ small enough in Estimate \eqref{step1result} we may insert Estimate \eqref{step2result} to get 
\begin{align*} 
\int_ r^ {\rho}w'(t)^2+\frac{w(t)^2}{2t^2} dt&\leq C(\epsilon)\left(1+\int_r^ \rho\frac{|w|}{t}dt\right)+\frac{w(r)^2}{2r}\\
&\hspace{1cm}+|w(r)|\left(C(\epsilon)\ln\frac1r+\epsilon \int_r^ \rho w'(t)^2 +\frac{w(t)^2}{2t^2} dt\right).
\end{align*}
By potentially shrinking $\rho$ we can assume without loss of generality that $|w|\leq \epsilon$ on $[0,\rho]$ and so, in particular $|w(r)|\leq \epsilon$. Choosing $\epsilon$ small enough, we can absorb the second integral on the right to the left and obtain that for some small enough $\rho$ and $r\in(0,\rho)$
\begin{equation}\label{step3result}
\int_r^ {\rho}w'(t)^2+\frac{w(t)^2}{2t^2} dt
\leq
C(\rho)\left(\frac{w(r)^2}{2r}+|w(r)|\ln\frac1r+\int_r^ \rho\frac{|w(t)|}tdt+1\right).
\end{equation}

\noindent
\textbf{Step 4}\ \\
Let $\rho$ be small enough so that Estimate \eqref{step3result} holds. Inserting Estimate \eqref{step3result} into Estimate \eqref{step2result} with $\epsilon=1$ and using that $|w|$ is bounded, we get
\begin{align*} 
\left|\frac{(r w(r))'}{r}\right| 
    \leq &
    C\ln\frac1r+C\left(\frac{w(r)^2}{2r}+|w(r)|\ln\frac1r+\int_r^ \rho\frac{|w(t)|}tdt+1\right) \\
    \leq &
    C\ln\frac1r+C\left(\frac{w(r)^2}{2r}+\int_r^ \rho\frac{1}tdt\right).
\end{align*}
We multiply by $r$. Clearly $|\ln(\rho)-\ln(r)|\leq C\ln(\frac1r)$ as $r\leq \rho$. Let $\mu\in(0,1)$. By potentially shrinking $\rho$ we can assume that $|w|\leq C^ {-1}\mu$ on $[0,\rho]$. Hence 
$$|(r w)'|\leq C(w(r)^2+r\ln\frac1r)\leq \mu |w|+Cr\ln\frac1r.$$ 
In general, for any weakly differentiable $g$ we have $|g|'\leq |g'|$. As $|rw|=r|w|$ we get 
$$r|w|'+|w|=(r|w|)'=|rw|'\leq |(rw)'|\leq \mu|w|+Cr\ln\frac1r \hspace{.5cm}\textrm{for all $r\in(0,\rho]$}.$$
Multiplying with $r^ {1-\mu}$ we get 
$$(r^ {1-\mu} |w|)'=r^ {1-\mu}|w|'+(1-\mu) r^ {-\mu}|w|\leq C r^ {1-\mu} \ln\frac1r.$$
We integrate this inequality from $r_1<r$ to $r$ and subsequently send $r_1\rightarrow 0^+$. Since $w\in C^0([0,r_0))$, we get
$$r^ {1-\mu}|w(r)|\leq C\int_0^ r t^ {1-\mu}\ln\frac1 t dt\leq C(\mu) r^{2-\mu}\ln\frac1r.$$
Taking e.g. $\mu=\frac12$ shows that there exists $\rho_0>0$ such that
\begin{equation}\label{step4result}
|w(r)|\leq Cr\ln\frac1r\hspace{.5cm}\textrm{for all }0<r<\rho_0.
\end{equation}

\noindent
\textbf{Step 5}\ \\
 Using Estimate \eqref{step4result} we deduce the following three estimates for $r\in(0,\rho_0)$:
\begin{align}
    \frac{w(r)^2}r&\leq C\frac{r^ 2\ln^2\frac1r}{r}\leq Cr\ln^2\frac1r\label{step5estimate1}\\
    |w(r)|\ln\frac1r&\leq Cr\ln^2\frac1r\label{step5estimate2}\\
    \int_0^ r\frac{|w(t)|}tdt&\leq C\int_0^  r \ln\frac1t dt\leq C\label{step5estimate3}
\end{align}
After potentially shrinking $\rho_0$, we may use Estimate \eqref{step3result}. Inserting Estimates \eqref{step5estimate1}-\eqref{step5estimate3} into Estimate \eqref{step3result}, we get
\begin{equation}\label{step5ausgangspunkt}
\int_r^ {\rho_0}w'(t)^2+\frac{w(t)^2}{2t^2} dt
\leq C.
\end{equation}
Potentially shrinking $\rho_0$ even further, we may assume $|w|\leq\frac12$ and, recalling that $v=\sqrt{1+w^2}$, we get
\begin{equation}\label{step5helpestimates}
\left|\frac{5w}{2(1+w^2)}(w')^2\right|
\leq C|w'|^2,
\hspace{.5cm}
\left|\frac{3w^3+w^ 5}{2t^2}\right|
\leq 
C\frac{w^2}{t^2}
\hspace{.5cm}\textrm{and}\hspace{.5cm}
\left|\frac{v^5-1}t\right|\leq C\frac{w^2}t.
\end{equation}
Combining Estimates \eqref{step5ausgangspunkt} and \eqref{step5helpestimates}, we obtain
\begin{equation}\label{step5result}
\int_0^ {\rho_0}\left|
\frac{5w}{2(1+w^2)}(w')^2
+
\frac{3w^3+w^ 5}{2t^2}
+\lambda\frac{v^ 5-1}t
\right|dt<\infty.
\end{equation}

\noindent
\textbf{Step 6}\ \\
We choose $\rho$ small enough, such that the results from all previous steps are valid. Considering the integral appearing in Equation \eqref{step2begin} from Step 2 we introduce
$$\psi(r):=-\int_ r^ \rho \varphi(w)+\lambda\frac{ v^ 5-1}t+\frac12 a v^ 5 t + b v^ 4wdt.$$
Recalling the definition of $\varphi(w)$ from Equation \eqref{varphiofwdef} and using Estimate \eqref{step5result} as well as $v\in C^ 0([0,\rho])$ we deduce $\psi\in C^0([0,\rho])$. Rewriting Equation \eqref{step2begin} we have shown that there exists a constant $\tilde c_0\in\R$ such that
$$(w(r) r)'=r\psi(r)+\tilde c_0+\lambda r\ln(r).$$
This implies that there exist constants $c_0, c_1\in\R$ such that
$$w(r)r=\int_0 ^r t\psi(t)dt +c_0 r+c_1 +\frac12 \lambda r^2 \ln(r)-\frac14\lambda r^2.$$
Letting $r\rightarrow 0^+$ we get that $c_1=0$ so that
$$w(r)-\frac12 \lambda r\ln(r)-c_0+\frac14\lambda r=\frac1r\int_0 ^r t\psi(t)dt=:h(r).$$
It remains to prove that $h\in C^1([0,\rho])$. First note, that $\psi\in C^0([0,\rho])$ is bounded. Hence $h$ is continuous at $r=0$ with $h(0)=0$. Indeed
$$|h(r)|\leq \frac Cr\int_0 ^r tdt\leq Cr\rightarrow 0,\hspace{.5cm}\textrm{as $r\rightarrow 0^+$}.$$
On $(0,\rho]$ we can differentiate the definition of $h$ to derive the equation $\frac{(r h(r))'}{r}=\psi(r)$. 
As $\psi\in C^0([0,\rho])$ we may use Lemma \ref{oderegularitylemma} from the appendix and get $h\in C^1([0,\rho])$. Finally, as $u$ is smooth on $[\rho, r_0]$, the lemma follows. 
\end{proof}

The next step is to derive an expansion of the mean curvature $H$.

\begin{lemma}\label{HExpansionLemma}
There exists a function $H_0(r)\in C^ {1,\alpha}([0, r_0])$ for all $\alpha\in(0,1)$ such that $H(r)=-\lambda\ln(r)+H_0(r)$.
\end{lemma}
\begin{proof}
Writing out the Willmore operator in Equation \eqref{ODE} in terms of $H$, we get 
$$\Delta_g H+ \frac12H(H^2-4K)=a-bH.$$
We put $f(r):=a-bH-\frac12 H(H^2-4K)$. Using Equations \eqref{graphmetric} and \eqref{graphH} we may write
$$\frac1{r \sqrt{1+u'(r)^2}}\frac{d}{dr}\left[\frac{r}{\sqrt{1+u'(r)^2}}H'(r)\right]
=\Delta_g H(r)
=f(r).$$
Lemma \ref{regularityatend01} gives $u'(r)=\frac12\lambda r\ln(r)+\xi$ with $\xi\in C^1([0, r_0])$. As $u'(0)=0$ we have $\xi(0)=0$. Using Equations \eqref{graphmetric} and \eqref{graphicalPrincipalCurvatures} allows us to estimate both $|H(r)|\leq C\ln\frac1r$ and $|K(r)|\leq C\ln^2\frac1r$. Therefore $|f(r)|\leq C\ln^3\frac1r$ and for all $p\in[1,\infty)$
\begin{equation}\label{docheinname}
\frac{d}{dr}\left[\frac r{\sqrt{1+u'(r)^2}}H'(r)\right]=r \sqrt{1+u'(r)^2} f(r)=:z(r)\in L^p((0, r_0)).
\end{equation}
Since $u\in C^ \infty((0,r_0])\cap C^1([0,r_0])$, $|H|\leq C\ln\frac1r$ and $|K|\leq C\ln^2\frac1r$ we get $z\in C^ \infty((0,r_0])\cap C^ 0([0,r_0])$. Integrating Equation \eqref{docheinname} implies that there exists $c_0\in\R$ such that  
$$H'(r)=\frac{c_0}{r}\sqrt{1+u'(r)^2}+\frac{\sqrt{1+u'(r)^2}}r\int_0^r z(s) ds.$$
We put $H_0(r):=H(r)-c_0\ln(r)$. Clearly $H_0\in C^ \infty((0,r_0])$ and we can compute 
\begin{align*}
    H_0''(r)=&\frac d{dr}\left[c_0\frac{\sqrt{1+u'(r)^2}-1}r+\frac{\sqrt{1+u'(r)^2}}r\int_0 ^r z(s)ds\right]\\
    =&-c_0\frac{\sqrt{1+u'(r)^2}-1}{r^2}+c_0\frac{u'(r) u''(r)}{r\sqrt{1+u'(r)^2}}
    -\frac{\sqrt{1+u'(r)^2}}{r^2}\int_0 ^r
    z(s)ds
    \\
    &+\frac{u'(r)u''(r)}{r\sqrt{1+u'(r)^2}}\int_0 ^r z(s)ds
    +\frac{\sqrt{1+u'(r)^2}}rz(r).
\end{align*}
We have $|u'(r)|\leq Cr\ln\frac1r$ and $|u''(r)|\leq C\ln\frac1r$. Also $|z(r)|\leq Cr\ln^3\frac1r$. This gives 
$$
|H_0''(r)|\leq C\left[\ln^2\frac1r+\frac 1{r^2}\int_0^r s\ln^3\frac1sds+\ln^2\frac1r\int_0^r s\ln^3\frac1sds+\ln^3\frac1r\right]\leq C\ln^3\frac1r.
$$
Hence $|H_0''(r)|\leq C\ln^3\frac1r\in L^p((0, r_0))$ for all $p\in[1,\infty)$. Using Lemma \ref{derivativeinL1} we get $H_0\in C^1([0,r_0])$ and thus $H_0\in W^ {2,p}((0,r_0))$ for all $p<\infty$. By the Sobolev embedding theorem we deduce $H_0\in C^{1,\alpha}([0,r_0])$ for all $\alpha\in(0,1)$. Recalling the definition of $H_0$, this gives
\begin{equation}\label{Hexpansionversion0}
H(r)=c_0\ln(r)+H_0(r)\hspace{.5cm}\textrm{with $H_0\in C^ {1,\alpha}([0, r_0])$}.
\end{equation}
It remains to check that $c_0=-\lambda$. We recall $u'(r)=\frac\lambda2 r\ln(r)+\xi(r)$ where $\xi\in C^1([0,r_0])$. In particular $u'\in C^0([0,r_0])$. Using $u'(0)=0$ and Equation \eqref{graphicalPrincipalCurvatures} we compute 
$$ \frac{H(r)}{\ln(r)}=-\frac{u''(r)}{\ln(r)}\frac1{\sqrt{1+u'(r)^2}^3}-\frac{u'(r)}{r\ln(r)}\frac1{\sqrt{1+u'(r)^2}}
\rightarrow-\lambda\hspace{.5cm}\textrm{as }r\rightarrow0^+.$$
In view of Equation \eqref{Hexpansionversion0}, we obtain $c_0=-\lambda$.
\end{proof}

\paragraph{Higher Regularity}\ \\
In the next lemma, we prove that the parameter $\lambda$ from Equation \eqref{ODEamRand} vanishes. In view of Lemma \ref{regularityatend01} we then get $u\in C^2([0,r_0])$. The key ingredient that allows this further improvement of regularity is the fact that the surface $\Sigma_\gamma$ minimizes the Willmore energy inside the class $\mathcal F_\sigma$.\\ 

Recall that $\gamma_1$ is a $C^1$-diffeomorphism from $[0,t_0]$ onto $[0,r_0]$ and that $u:[0,r_0]\rightarrow\R,\ u(r):=\gamma_2(\gamma_1^{-1}(r))$. We put $r_1:=\frac{r_0}2$, $t_1:=\gamma_1^{-1}(r_1)\in(0, t_0)$ and choose $\chi\in C^\infty([0,1])$ such that $\chi\equiv 1$ on $[0, t_1]$ and $\chi\equiv 0$ on $[\frac{t_0+t_1}2,1]$. For $\varphi\in C^\infty_0([0,r_1))$ satisfying $\varphi'(0)=0$ we consider the variation
$$
\Phi:(-\epsilon_0,\epsilon_0)\times[0,1]\rightarrow\mathcal H^2,\ \Phi(\epsilon, t):=\gamma(t)+\epsilon\chi(t)\varphi(\gamma_1(t))e_2\hspace{.5cm}\textrm{and put}\hspace{.5cm}
\phi(t):=\frac{\partial\Phi}{\partial\epsilon}\bigg|_{\epsilon=0}.$$
We note that $\Phi(\epsilon, t)\equiv \gamma(t)$ for $t\geq t_1$ and $\Phi(\epsilon,\cdot)\in\mathcal P'$ for all small $\epsilon$. However, the variation $\Phi$ does not preserve the isoperimetric ratio. Employing the function $\psi_0\in C^\infty_0((0,1))$ from Lemma \ref{psi0existencelemma} we may define an admissible variational vector field by putting 
\begin{equation}\label{Vdef}
\eta:=\phi-\alpha(\phi)\psi_0
\hspace{.5cm}\textrm{with}\hspace{.5cm}
\alpha(\phi)=\delta\mathcal I[\gamma]\phi.
\end{equation}

We require the following lemma, 
\begin{lemma}\label{eulerlagangeataxis}
There exists a variation $\tilde\Phi:(-\epsilon_0,\epsilon_0)\rightarrow\mathcal F_\sigma'$ such that $\tilde\Phi'(0)=\eta$. In particular $\delta \mathcal W[\gamma]\eta=0$ by Equation \eqref{weakeulerlagrangefirstappe}. Additionally 
$$0=\delta \mathcal I[\gamma] \eta=-3\pi\lim_{\tau\rightarrow0^+}\int_\tau^1\left( \pm_\gamma 4\sqrt\pi-\sigma H[\gamma]\right)\langle\nu,\eta\rangle L\gamma_1 dt.$$
\end{lemma}
\begin{proof}
    The existence of $\tilde\Phi$ is established in Appendix \ref{variationexistenceappendix}. Here we only establish the formula. Note that $\eta=\chi\cdot(\varphi\circ\gamma_1)e_2-\alpha(\phi)\psi_0\in C^1([0,1])$ since $\gamma\in C^1([0,1])$. Using Equation \eqref{var3}, $\gamma_1(0)=0$ and the fact that $\eta$ vanishes at $1$, we compute 
    \begin{align}
    \delta A[\gamma]\eta=&2\pi\lim_{\tau\rightarrow0^+}\int_\tau^1L\eta_1+\frac{\gamma_1(t)\langle\dot \eta(t),\dot\gamma(t)\rangle}{L}dt\nonumber\\
                        =&-2\pi\lim_{\tau\rightarrow0^+}\left[\frac{\gamma_1(\tau)\langle \eta(\tau), \dot\gamma(\tau)\rangle}L
                            +\int_\tau^1 H[\gamma]\langle n,\eta\rangle L\gamma_1(t)dt\right]\nonumber\\
                        =&-2\pi\lim_{\tau\rightarrow0^+}\int_\tau^1 H[\gamma]\langle \nu,\eta\rangle L\gamma_1(t)dt.\label{uberlegung1}
\end{align}
Using Equation \eqref{var2}, $\gamma_1(0)=0$, $\eta\equiv 0$ near $1$ and the regularity of $\gamma$ we write 
\begin{align}
    \delta V[\gamma]\eta=&\mp_\gamma2\pi\int_0^1 \langle\nu,\eta\rangle L\gamma_1(t) dt\mp_\gamma\pi\gamma_1^2(0)\eta_2(0)\nonumber\\
                        =&\mp_\gamma2\pi\lim_{\tau\rightarrow 0^+}\int_\tau^1 \langle\nu,\eta\rangle L\gamma_1(t) dt.\label{uberlegung2}
\end{align}
Inserting Equations \eqref{uberlegung1} and \eqref{uberlegung2} into Equation \eqref{var1} and using $A[\gamma]=1$ and $V[\gamma]=\frac{\sigma}{6\sqrt\pi}$, we get
\begin{align}
    \delta\mathcal I[\gamma]\eta=&\frac{6\sqrt\pi}{A[\gamma]^{\frac32}}\left[\delta V[\gamma]\eta-\frac{3 V[\gamma]}{2 A[\gamma]}\delta A[\gamma]\eta\right]\nonumber\\
        =&-6\sqrt\pi\lim_{\tau\rightarrow0^+}2\pi\int_\tau^1\left( \pm_\gamma 1-\frac{\sigma}{4\sqrt\pi}H[\gamma]\right)\langle\nu,\eta\rangle L\gamma_1 dt.\label{isoperimetricepsiloin}
\end{align}
\end{proof}

\newpage
\begin{lemma}\label{uisc2}
$\lambda=0$ and consequently by Lemmas \ref{regularityatend01} and \ref{HExpansionLemma}, $u\in C^2([0,r_0])$ and $H\in C^1([0,r_0])$.
\end{lemma}

\begin{proof}
Throughout the proof, we use both the profile curve and the graph function parameterization. To shorten the notation, we denote the mean curvature by $\tilde H$ when working in the graph function parameterization and by $ H$ when using the profile curve parameterization. That is: $\tilde H(r)=\tilde H(\gamma_1(s))=H(s)$.\\

Let $\varphi\in C^\infty_0([0, r_1))$ such that $\varphi'(0)=0$ and define $\eta$ as in Equation \eqref{Vdef}. By Lemma \ref{eulerlagangeataxis} there exists a variation $\tilde\Phi:(-\epsilon_0,\epsilon_0)\rightarrow\mathcal F_\sigma'$ such that $\frac{\partial\tilde\Phi(\epsilon)}{\partial\epsilon}\big|_{\epsilon=0}=\eta$ and we get the Euler-Lagrange equation 
\begin{equation}\label{weakeulerlagrangelambda}
0=\delta\mathcal W[\gamma]\eta=\delta \mathcal W[\gamma]\phi-\alpha(\phi)\delta\mathcal W[\gamma]\psi_0.
\end{equation}

\paragraph{The first term}\ \\
$\phi$ is supported on $[0, t_1]$. Here $\chi\equiv 1$ and hence
\begin{align*}
    \delta\mathcal W[\gamma]\phi&=\frac d{d\epsilon}\bigg|_{\epsilon=0}\mathcal W[\gamma+\epsilon\phi]=\frac{2\pi}4 \frac d{d\epsilon}\bigg|_{\epsilon=0}\int_0^{r_1} \tilde H[\gamma+\epsilon\phi]^2 \sqrt{1+(u'+\epsilon\varphi')^2}rdr\\
   =& \frac{2\pi}4 \int_0^{r_1}
   \left(
   2 \tilde H(r)\delta \tilde  H[\gamma]\phi\sqrt{1+u'(r)^2}+ \tilde  H(r)^2 \frac{u'(r)\varphi'(r)}{\sqrt{1+u'(r)^2}} 
   \right)rdr.
\end{align*}
Let us consider $\delta  \tilde H[\gamma]\phi$. By Lemma \ref{regularityatend01} there exists $\xi\in C^1([0,r_1])$ such that $u'(r)=\frac\lambda 2r\ln(r)+\xi(r)$. Since $u'(0)=0$ we have $\xi(0)=0$. Also $u''(r)=\frac\lambda2\ln(r)+\tilde\xi(r)$ where $\tilde\xi=\xi'+\frac\lambda2$ is continuous on $[0,r_0]$. Using Equation \eqref{graphicalPrincipalCurvatures} shows that for all $p\in[1,\infty)$
$$\delta \tilde H[\gamma]\phi=
-\frac{\varphi''(r)}{\sqrt{1+u'(r)^2}^3}
+3\frac{u''(r)u'(r)\varphi'(r)}{\sqrt{1+u'(r)^2}^5}
-\frac{\varphi'(r)}{r\sqrt{1+u'(r)^2}}
+\frac{u'(r)^2\varphi'(r)}{r\sqrt{1+u'(r)^2}^3}
\in L^p((0,r_1)).$$
Additionally, $ \tilde H\in L^p((0, r_0))$ for all $p\in[1,\infty)$ by Lemma \ref{HExpansionLemma}. Let $\rho>0$ be a small parameter. We put $\gamma_\rho:=\gamma|_{[\gamma_1^{-1}(\rho), \gamma_1^{-1}(r_1)]}$. Using Lebesgue's theorem, we get
\begin{align}
    \delta\mathcal W[\gamma]\phi
   &= \lim_{\rho\rightarrow0^+}\frac{2\pi}4 \int_\rho ^{r_1}  \left(2 \tilde H(r)\delta \tilde H[\gamma]\phi
   \sqrt{1+u'(r)^2}+  \tilde H(r)^2 \frac{u'(r)\varphi'(r)}{\sqrt{1+u'(r)^2}} \right) rdr\nonumber\\
   &= \lim_{\rho\rightarrow0^+}\delta\mathcal W[\gamma_\rho]\phi.\label{limitfirstterm}
\end{align}
Since $\gamma\in C^\infty((0,1))$ we have  $\gamma_\rho\in C^\infty([\gamma_1^{-1}(\rho),\gamma_1^{-1}(r_1)])$ for all small $\rho>0$. Using Equation \eqref{variationformulastandard} from the appendix and using that $\phi$ is compactly supported in $[0,\gamma_1^{-1}(r_1))$, we get
\begin{align}
\delta\mathcal W[\gamma_\rho]\phi=2\pi\int_{\gamma_1^{-1}(\rho)}^{\gamma_1^{-1}(r_1)} W[\gamma]\langle \phi, \nu\rangle \gamma_1 |\dot\gamma|dt
+B(\rho)\nonumber\\
=2\pi\int_{\gamma_1^{-1}(\rho)}^{1} W[\gamma]\langle \phi, \nu\rangle \gamma_1 |\dot\gamma|dt
+B(\rho).
\label{firsttermresult}
\end{align}
Here we have introduced the boundary term 
\begin{equation}\label{Bdefinition}
B(\rho)=
\pi\left[
\frac{\varphi(\gamma_1(s))   H'(s)-(\varphi\circ\gamma_1)'(s) H(s)
-\frac12H^2(s)\langle \dot\gamma(s),\phi(s)\rangle}
{|\dot\gamma(s)|}
\right]\gamma_1(s)\bigg|_{s=\gamma_1^{-1}(\rho)}.
\end{equation}

\paragraph{The second term}\ \\
As $\psi_0\in C^\infty_0((0,1))$, its support lies in an interval $(a,b)$ with $0<a<b<1$ where $\gamma\in C^\infty([a,b])$. We use Equation \eqref{variationformulastandard} from the appendix. As $\psi_0$ has compact support in $(a,b)$ we get
\begin{equation}\label{secondtermresult}
\delta\mathcal W[\gamma]\psi_0=2\pi\int_a^bW[\gamma]\langle \psi_0, \nu\rangle \gamma_1|\dot\gamma|ds
=2\pi\lim_{\rho\rightarrow 0^+}\int_{\gamma_1^{-1}(\rho)}^1W[\gamma]\langle \psi_0, \nu\rangle \gamma_1|\dot\gamma|ds.
\end{equation}

\paragraph{Combining the results}\ \\
By Lemma \ref{eulerlagrangeequation}, $\gamma$ satisfies the Euler-Lagrange equation $W[\gamma]=\Lambda (\pm_\gamma4\sqrt\pi-\sigma H[\gamma])$ on $(0,1)$. Using Lemma \ref{eulerlagangeataxis}, we get
 \begin{align}
     &\lim_{\rho\rightarrow0^+}\left[2\pi\int_{\gamma_1^{-1}(\rho)}^1W[\gamma]\langle \eta, \nu\rangle \gamma_1|\dot\gamma|ds\right]\nonumber\\
    =&
    \lim_{\rho\rightarrow0^+}\left[2\pi\Lambda \int_{\gamma_1^{-1}(\rho)}^1(\pm_\gamma4\sqrt\pi-\sigma H[\gamma])\langle \eta, \nu\rangle \gamma_1|\dot\gamma|ds\right]\nonumber\\
   =&
    -\frac{2\Lambda}3\delta\mathcal I[\gamma]\eta
    =0.\label{surfaceintegralscombined}
 \end{align}
Combining Equations \eqref{weakeulerlagrangelambda}, \eqref{limitfirstterm}, \eqref{firsttermresult},  \eqref{secondtermresult} and  \eqref{surfaceintegralscombined} we deduce 
$
\lim_{\rho\rightarrow0^+}B(\rho)=~0.
$
We now rewrite the boundary term in terms of the graph function $u$. Note 
$$
r=\gamma_1(s)
\hspace{.5cm}\textrm{and consequently}\hspace{.5cm}
\frac\partial{\partial s}=\frac{\partial r}{\partial s}\frac\partial{\partial r}=\dot\gamma_1(s)\frac{\partial}{\partial r}.$$
Hence $\dot\gamma_2(s)=\dot\gamma_1(s)(\gamma_2\circ\gamma_1^{-1})'(r)=\dot\gamma_1(s) u'(r)$ and therefore $|\dot\gamma(s)|=\dot\gamma_1(s)\sqrt{1+u'(r)^2}$. Inserting into Equation \eqref{Bdefinition}, we get
$$B(\rho)=\pi r\frac{\varphi(r) \tilde H'(r)-\varphi'(r) \tilde H(r)-\frac12 \tilde H^2(r)u'(r)\varphi(r)}{\sqrt{1+u'(r)^2}}\bigg|_{r=\rho}.$$
Using the formulas $u'(r)=\frac\lambda2 r\ln(r)+\xi(r)$ where $\xi\in C^1([0, r_0])$ satisfies $\xi(0)=0$ from Lemma \ref{regularityatend01}, $ \tilde H(r)=-\lambda\ln(r)+ \tilde H_0(r)$ with $ \tilde H_0\in C^{1,\alpha}([0,r_0))$ from Lemma \ref{HExpansionLemma}, the smoothness of $\varphi$ and $\varphi'(0)=0$, we get 
$$0=\lim_{\rho\rightarrow 0^+} B(\rho)=\lim_{\rho\rightarrow0^+}\left[\pi\rho\frac{\varphi(\rho)}{\sqrt{1+u'(\rho)^2}}\left(-\frac\lambda\rho +\tilde H_0'(\rho)\right)\right]=-\pi\lambda\varphi(0).$$
As $\varphi(0)$ is arbitrary we deduce $\lambda=0$. 
\end{proof}

Having improved the regularity of $u$ to $C^2$, we can now easily deduce that $u$ is, in fact, smooth.
\begin{lemma}\label{uissmooth}
$u\in C^\infty([0,r_0))$. 
\end{lemma}
\begin{proof}
The proof is inductive. Let $k\geq 0$ and suppose that we know $H\in C^k([0,r_0))$. By Equation \eqref{graphH} 
$$-\frac1r\frac{d}{dr}\left[r\frac{u'(r)}{\sqrt{1+u'(r)^2}}\right]=H\in C^k([0,r_0)).$$
Lemma \ref{oderegularitylemma} implies $u'\in C^{k+1}([0,r_0))$ and thus $u\in C^{k+2}([0,r_0))$. In view of Equation \eqref{graphicalPrincipalCurvatures} we get $k_1\in C^k([0, r_0))$ and since $u'(0)=0$ 
$$k_2(r)=-\frac{u'(r)}{r\sqrt{1+u'(r)^2}}=-\frac1{\sqrt{1+u'(r)^2}}\int_0^1 u''(sr) ds\hspace{.5cm}\textrm{is also of class }C^k([0, r_0)).$$
Thus $K=k_1k_2\in C^ k([0,r_0))$. Writing out the Euler-Lagrange equation from Lemma \ref{eulerlagrangeequation}, we get
$$\frac1{r\sqrt{1+u'(r)^2}}\frac d{dr}\left[\frac r{\sqrt{1+u'(r)^2}}\frac {dH}{dr}\right]
=
\Delta_g H
=-\frac12 H(H^2-4K)+a-bH\in C^k([0, r_0)).
$$
Using $u'\in C^{k+1}([0, r_0))$ and putting $\alpha(r):=\frac 1{\sqrt{1+u'(r)^2}}\frac {dH}{dr}$, we get $\frac{(r\alpha)'}r\in C^ k([0,r_0))$. Lemma \ref{oderegularitylemma} implies $\alpha\in C^{k+1}([0, r_0))$. 
By definition of $\alpha$ and using $u'\in C^{k+1}([0, r_0))$ we deduce $H'(r)\in C^{k+1}([0, r_0))$ and consequently $H\in C^{k+2}([0, r_0))$.\\

We have shown the following two implications: $H\in C^k([0,r_0) \Rightarrow H\in C^{k+2}([0,r_0))$ and $H\in C^k([0,r_0)) \Rightarrow u\in C^{k+2}([0,r_0))$. Since we have shown that $H\in C^1([0, r_0))$ in Lemma \ref{uisc2}, we deduce $u\in C^\infty([0,r_0))$.
\end{proof}

In Lemma \ref{uissmooth}, we have shown that the graph function $u$ describing $\gamma$ near $t=0$ is smooth. By the same arguments, the graph function $\tilde u$ describing $\gamma$ near $t=1$ is also smooth.
As an easy corollary, we obtain that $\gamma$ is smooth up to $t=0$ and $t=1$.

\begin{korollar}\label{gammasmoothuptoBoundaryCorollary}
    $\gamma\in C^\infty([0,1])$
\end{korollar}
\begin{proof}
    In view of Corollary \ref{gammasmooththeorem} we already have $\gamma\in C^\infty((0,1))$. Recall that we chose $t_0>0$ and $r_0:=\gamma_1(t_0)$ such that $\gamma_1:[0, t_0)\rightarrow [0, r_0)$ is a $C^1$-diffeomorphism. In particular, as $\dot\gamma_1(0)=L:=L[\gamma]>0$, we have $\dot\gamma_1(t)>0$ for all $t\in[0, t_0)$. From Lemma \ref{uissmooth}, we know that  $u:[0,r_0)\rightarrow\R,\ u(r):=\gamma_2(\gamma_1^{-1}(r))$
    is in $C^\infty([0, r_0))$. Clearly $\gamma_2(t)=u(\gamma_1(t))$. Differentiating and using $|\dot\gamma(t)|^2=L^2$ yields 
    $$L^2-\dot\gamma_1(t)^2=\dot\gamma_2(t)^2=u'(\gamma_1(t))^2\dot\gamma_1(t)^2\hspace{.5cm}\textrm{and so}\hspace{.5cm}\dot\gamma_1(t)^2=\frac{L^2}{1+u'(\gamma_1(t))^2}.$$
    Since $\dot\gamma_1(t)\geq 0$ for all $t\in[0, t_0)$ we can take the square root and obtain 
    $$\dot\gamma_1(t)=\frac{L}{\sqrt{1+u'(\gamma_1(t))^2}}.$$
    As $\gamma_1\in C^0([0, t_0))$ and $u'\in C^\infty([0, r_0))$, we deduce that $\gamma_1\in C^\infty([0, t_0))$ and consequently $\gamma_2=u\circ\gamma_1\in C^\infty([0, t_0))$. This establishes the smoothness of $\gamma$ near $t=0$. The smoothness near $t=1$ can be established by the same method. 
\end{proof}

We now prove Theorem \ref{regularitytheorem1} by demonstrating that $u^{(2k+1)}(0)=0$ for all integers $k\geq 0$. Indeed the same proof shows the analog result for the graph function $\tilde u$ describing $\gamma$ near $t=1$. 

\begin{lemma}[Proof of Theorem \ref{theorem1} (apart from $\beta(\sigma)\rightarrow 8\pi$ for $\sigma\rightarrow 0^+$)]\label{Sigmaissmooth}\ \\
$u^{(2k+1)}(0)=0$ for all $k\in\N_0$ and consequently  $\Sigma_\gamma$ is smooth.
\end{lemma}
\begin{proof}
The proof is inductive. By Equation \eqref{ODEamRand} we have $w(0)=u'(0)=0$. We show that if there exists $n\geq 0$ such that $u^{(2k+1)}(0)=0$  for all $0\leq k\leq n$, then $u^{(2n+3)}(0)=0$. By Lemma \ref{uisc2} and Equation \eqref{ODEamRand} we have 
\begin{equation}\label{SmoothnessFinalStepODE}
u'''+\left(\frac{u'}{r}\right)'
=
\frac12 \sqrt{1+u'(r)^2}^ 5 a r+b(1+u'(r)^2)^2u'(r)+\varphi(u').
\end{equation}
Since $u\in C^\infty([0, r_0)$ and $u^{(2k+1)}(0)=0$ for $0\leq k\leq n$, the function
$$U:(-r_0,r_0)\rightarrow\R,\ U(r):=\left\{
\begin{aligned}
u(r) &\hspace{.5cm} \textrm{ if }r\geq 0,\\
u(-r) &\hspace{.5cm}\textrm{ if } r<0
\end{aligned}
\right.$$
is smooth on $[0,r_0)$, $(-r_0, 0]$ and $(-r_0,r_0)\backslash\set 0$. However, the derivatives at $0$ of order$\geq 2n+3$ might not match up so that only $U\in C^{2n+2}((-r_0,r_0))$ is guaranteed. Consequently, the function 
$$F(r):=\frac12 \sqrt{1+U'(r)^2}^ 5 a r+b(1+U'(r)^2)^2U'(r)+\varphi(U')$$
is of class $C^{2n}((-r_0, r_0))$. Indeed, considering the definition of $\varphi(U')$ from Equation \eqref{varphiofwdef}, we see that only the term $\frac{(U'(r))^3}{2r^2}(3+(U'(r))^2)$ is of concern. However, using that $u'(0)=0$, we get 
$$\frac{U'(r)^2}{r^2}=\left(\int_0^1 U''(sr)ds\right)^2,\hspace{.5cm}\textrm{which is of class }C^{2n}(-r_0, r_0).$$
Additionally, it is readily checked that $F(-r)=-F(r)$. Thus $F^{(2n)}(0)=0$. Differentiating Equation \eqref{SmoothnessFinalStepODE} $2n$ times, we get 
\begin{equation}\label{differentiatedODE}
u^{(2n+3)}+\left(\frac{u'}r\right)^{(2n+1)}=F^{(2n)}(r)\hspace{.5cm}\textrm{on }[0, r_0).
\end{equation}
Since $u\in C^\infty([0, r_0))$ we may write
$$\left(\frac{u'}r\right)^{(2n+1)}
=
\left(\int_0^1 u''(rs)ds\right)^{(2n+1)}
=
\int_0^1 u^{(2n+3)}(rs)s^{2n+1} ds.$$
Now we consider Equation \eqref{differentiatedODE} the limit $r\rightarrow 0^+$. Since $F^{(2n)}(0)=0$ we get
\begin{equation}\label{nextderivative0}
0=\lim_{r\rightarrow 0^+}\left(u^{(2n+3)}(r)+\int_0^ 1  u^{(2n+3)}(rs)s^{2n+1} ds\right)=
\left(1+\frac1{2n+2}\right)u^{(2n+3)}(0)
\end{equation}
and hence $u^{(2n+3)}(0)=0$.
\end{proof}

\section{Convergence to a Double Sphere}\label{doublespheresection}
Let $\sigma_k\rightarrow0^+$ and  $\gamma^{(k)}\in\mathcal F_{\sigma_k}$ satisfy $\mathcal W[\gamma^{(k)}]=\beta(\sigma_k)$. Since vertical translations $(\gamma_1,\gamma_2)\mapsto (\gamma_1,\gamma_2+c)$ and the reflection $(\gamma_1,\gamma_2)\mapsto (\gamma_1,-\gamma_2)$ leave the Willmore energy and the isoperimetric ratio invariant, we cannot expect the profile curves $\gamma^{(k)}$ to converge. To rule out both obstructions, we introduce 
$$\mathcal F^ {0+}_\sigma:=\set{\gamma\in\mathcal F_\sigma\ |\ \gamma(0)=0\textrm{ and }\int_0^1 \gamma_2(t) dt\geq 0},$$
and stress that $\beta(\sigma)=\inf\set{\mathcal W[\gamma]\ |\ \gamma\in\mathcal F_\sigma}
=\inf\set{\mathcal W[\gamma]\ |\ \gamma\in\mathcal F_\sigma^{0+}}$. Now, considering a sequence $\gamma^{(k)}\in\mathcal F_{\sigma_k}^{0+}$ convergence to a limit is possible. Indeed, letting $R:=\frac1{\sqrt{8\pi}}$ and 
$$\kappa:[0,1]\rightarrow \R^2,\ \kappa(t):=(0, R)+R(|\sin(2\pi t)|, -\cos(2\pi t)),$$
we prove the following:

\begin{theorem}[Convergence to a Double Sphere]\label{asymptoticsthm}
Let $(\sigma_k)\subset (0,1)$ such that $\sigma_k\rightarrow 0^+$ and let $\gamma^ {(k)}\in\mathcal F_{\sigma_k}^ {0+}$ such that $\mathcal W[\gamma^{(k)}]=\beta(\sigma_k)$.
Then $\beta(\sigma_k)\rightarrow 8\pi$ and $\gamma^{(k)}\rightarrow\kappa$ in $W^ {1,2}((0,1))$ and in $C^\infty_{\operatorname{loc}}([0,1]\backslash\{\frac12\})$.
\end{theorem}

In particular, we note that Theorem \ref{asymptoticsthm} implies  Theorem \ref{doublesphereconvergence} and concludes the proof of Theorem \ref{theorem1}.\\
\ \\
For the entirety of this section we fix a sequence $\sigma_k\rightarrow 0^+$ as well as $\gamma^{(k)}\in \mathcal F^{0+}_{\sigma_k}$ such that $\beta(\sigma_k)=\mathcal W[\gamma^{(k)}]$. We will prove that there exists a subsequence $\gamma^{(k_l)}$ that converges to $\kappa$ in $W^{1,2}((0,1))$ and in $C^\infty_{\operatorname{loc}}([0,1]\backslash\{\frac12\})$. Once we have shown this, Theorem \ref{asymptoticsthm} follows by the usual subsequence argument. 

\begin{lemma}\label{convergentsubsequence}
There exists a subsequence $\gamma^{(k_l)}$ and a curve $\gamma^*\in\mathcal P^w$ such that $\gamma^{(k_l)}\rightarrow\gamma^*$ in the sense of Definition \ref{convergencedefinition}. Moreover, either $\gamma^*\in \mathcal P$ or there exists $\tau\in(0,1)$ such that, after reparameterization, $\gamma^*|_{[0,\tau]}$ and  $\gamma^*|_{[\tau,1]}$ both belong to $\mathcal P$.
\end{lemma}
\begin{proof}
The proof for the existence of $\gamma^*$ and the convergence is essentially the same as the one from Theorem \ref{weakminimizerexistence}. However, instead of Estimate \eqref{leq38pi} we only get 
$$\mathcal W[\gamma^*]\leq 8\pi\hspace{.5cm}\textrm{and}\hspace{.5cm}\int_0^1 \left((k_1^*)^2+(k_2^*)^2 \right)2\pi L_*\gamma_1^*(t)\leq 3\cdot 8\pi.$$
A priori Theorem \ref{weakandstrongPconnection} implies that there are at most 2 points $\tau\neq\tau'\in (0,1)$ such that $\gamma_1^*(\tau)=\gamma_1^*(\tau')=0$. However, following the arguments after Equation \eqref{leq38pi} implies that there can be at most one such point.
\end{proof}

Let $\delta>0$ be small. Depending on whether the limit $\gamma^*$ from Lemma \ref{convergentsubsequence} satisfies $\gamma^*_1(\tau)=0$ for some $\tau\in(0,1)$ or not, we put 
$$
I_\delta:=
\left\{
\begin{aligned}
(\delta,1-\delta) &\ \textrm{if $\gamma^*\in\mathcal P$},\\
(\delta,\tau-\delta)\cup (\tau+\delta,1-\delta)&\ \textrm{if $\gamma^*\not\in\mathcal P$}.
\end{aligned}
\right.
$$

We can now prove that along a subsequence $\gamma^{(k)}$ converges to $\gamma^*$ not only in $\mathcal P^w$ but in $C^m(I_\delta)$ for all $\delta>0$ and $m\in\N$. 

\begin{lemma}\label{smoothconvergence}
Let  $\gamma^ {(k_l)}$ and $\gamma^*$ be as in Lemma \ref{convergentsubsequence}. Then  $\gamma^*\in C^\infty(I_\delta)$ for all $\delta>0$ and $\gamma^ {(k_l)}\rightarrow\gamma^*$ in $C^m(I_\delta)$ for all  $m\in\N$ and small $\delta>0$.
\end{lemma}

\begin{proof}
We prove that for arbitrary $m\in\N$ the sequence $\gamma^ {(k_l)}$ is bounded in $W^ {m+1,\infty}(I_\delta)$. The Sobolev embedding $W^ {m+1,\infty}(I_\delta)\hookrightarrow C^ {m,\alpha}(\bar I_\delta)$ then shows that $\gamma^ {(k_l)}$ is bounded in $C^ {m,\alpha}(I_\delta)$. Once this is shown, we can use $\gamma^ {(k_l)}\rightarrow \gamma^*$ in $\mathcal P^w$ to deduce the lemma. From now on we write $\gamma^{(k)}$ instead of $\gamma^{(k_l)}$\\

\noindent
\textbf{Uniform Bounds in $W^ {m,\infty}$}\ \\
By Definition \ref{convergencedefinition} and Lemma \ref{convergentsubsequence} we get $\gamma^ {(k)}\rightarrow \gamma^*$ in $C^0([0,1])$. So we can assume that
\begin{equation}\label{kappauniform}
\inf_{t\in I_\delta}\gamma_1^{(k)}(t)
\geq 
\frac12\inf_{t\in I_\delta}\gamma_1^*(t)
=\frac12\kappa(\gamma^*;I_\delta)>0.
\end{equation}
Either $I_\delta=(\delta,1-\delta)$ or $I_\delta=(\delta,\tau-\delta)\cup(\tau+\delta,1-\delta)$. Let $(\tau_1,\tau_2)$ denote any of these intervals. We claim that there is $\psi^*\in C^ \infty_0((\tau_1,\tau_2))$ such that 
\begin{equation}\label{psistardef}
1\overset!=6\sqrt\pi \delta V[\gamma^*]\psi^*
=\pm_{\gamma^*}6\sqrt\pi^\frac32 \int_{\tau_1}^ {\tau_2} (\gamma_1^*)^2\dot\psi^*_2+2\gamma_1^*\dot\gamma_2^*\psi^*_1 dt.
\end{equation}
Indeed, otherwise, we could use partial integration to get for all $\psi^*\in C^ \infty_0((\tau_1,\tau_2))$
$$
0=\int_{\tau_1}^ {\tau_2} (\gamma_1^*)^2\dot\psi^*_2+2\gamma_1^*\dot\gamma_2^*\psi^*_1 dt
=2\int_{\tau_1}^ {\tau_2} \gamma_1^*\langle\dot\gamma^*, (-\psi_2^*,\psi_1^*)\rangle dt.
$$
Using $\gamma^*\in W^{1,2}((0,1))$ and $\gamma_1^*>0$ almost everywhere we conclude $\dot\gamma^*= 0$ almost everywhere in $(\tau_1,\tau_2)$. This is a contradiction as $\gamma^*\in\mathcal P^w$ and hence $|\dot\gamma^*|=L[\gamma^*]$ almost everywhere. \\

Now let $\psi^*$ be as in Equation \eqref{psistardef}. Using $\gamma^{(k)}\rightarrow\gamma^*$ in $C^0([0,1])$ and $W^{1,2}((0,1))$ and $\sigma_k\rightarrow 0$, we can use Equations \eqref{var1}-\eqref{var3} to get $1=\lim_{k\rightarrow\infty}\delta\mathcal I[\gamma^{(k)}]\psi^*$. This proves that, after potentially ignoring the first couple of terms in the sequence, there exists a sequence $\zeta_k\subset(0,\infty)$ converging to $1$ such that 
$$1=\zeta_k\left(\delta\mathcal I[\gamma^{(k)}]\psi^*\right)=\delta\mathcal I[\gamma^{(k)}](\zeta_k\psi^*).$$
This allows us to use Corollary \ref{gammasmooththeorem1} with $U=(\tau_1,\tau_2)$ and $\psi_0^ {(k)}=\zeta_k\psi^*$ to prove the lemma. Indeed, Estimate \eqref{kappauniform} provides us with a uniform lower bound for $\kappa(\gamma^{(k)};U)$ and, as $\zeta_k\rightarrow 1$, the $C^2$-norms of $\psi^ {(k)}_0$ are bounded uniformly.
\end{proof}

We now prove a first part of Theorem \ref{asymptoticsthm}.
\begin{lemma}\label{Lemma5point4}
    Let $(\sigma_k)\subset (0,1)$ such that $\sigma_k\rightarrow 0^+$ and let $\gamma^ {(k)}\in\mathcal F_{\sigma_k}^ {0+}$ such that $\mathcal W[\gamma^{(k)}]=\beta(\sigma_k)$.
Then $\beta(\sigma_k)\rightarrow 8\pi$ and $\gamma^{(k)}\rightarrow\kappa$ in $W^ {1,2}((0,1))$ and in $C^\infty_{\operatorname{loc}}((0,1)\backslash\{\frac12\})$.
\end{lemma}
\begin{proof}
Combining Lemmas \ref{convergentsubsequence} and \ref{smoothconvergence} we have $\gamma^{(k)}\rightarrow\gamma^*$ in $\mathcal P^w$ and in $C^m(I_\delta)$ for all $m\in\N$ and $\delta>0$ along a subsequence again denoted by $\gamma^{(k)}$. Additionally, either $\gamma^*\in\mathcal P$ or there exists $\tau\in(0,1)$ such that $\gamma_1^*(\tau)=0$ and $\gamma^*|_{[0,\tau]}$,  $\gamma^*|_{[\tau,1]}$ both belong to $\mathcal P$. In either case, we can use the weak lower semicontinuity from Equation \eqref{Willmorelowersemicont} and deduce
\begin{equation}\label{limitWillmoreupperbound}
\mathcal W[\gamma^*]\leq \liminf_{k\rightarrow\infty}\mathcal W[\gamma^{(k)}]\leq\limsup_{k\rightarrow\infty}\mathcal W[\gamma^{(k)}]\leq 8\pi.
\end{equation}
Using $\mathcal I[\gamma^ {(k)}]=\sigma_k\rightarrow 0^ +$, $A[\gamma^ {(k)}]\equiv 1$ and the continuity of the area and volume with respect to the $\mathcal P^w$ convergence (see Equations \eqref{areaconv} and \eqref{volconv}), we get 
\begin{equation}\label{areaANDvolume}
A[\gamma^*]=\lim_{k\rightarrow\infty}A[\gamma^ {(k)}]\equiv 1
\hspace{.5cm}\textrm{and}\hspace{.5cm}
V[\gamma^*]=\lim_{k\rightarrow\infty}V[\gamma^ {(k)}]=0.
\end{equation}

\paragraph{The first case: $\gamma^*\in\mathcal P$}\ \\
We claim that there exist $\tau<\tau'\in[0,1]$ such that $\gamma(\tau)=\gamma(\tau')=:(p_1,p_2)$. Indeed, if $\gamma^*$ was injective, then $V[\gamma^*]>0$ by Lemma \ref{volumegauss} which contradicts Equation \eqref{areaANDvolume}. By Lemma \ref{monotonicityformula} we conclude $\mathcal W[\gamma^*]\geq 8\pi$. Combining this with Estimate \eqref{limitWillmoreupperbound}, we deduce $\mathcal W[\gamma^*]=8\pi$. Putting $p:=(p_1,0,p_2)$, we are in the equality case of the monotonicity formula from Lemma \ref{monotonicityformula} and get 
$$\frac14\vec H[f_{\gamma^*}]=\frac{(f_{\gamma^*}-p)^\perp}{|f_{\gamma^*}-p|^2}.$$
Using Lemma \ref{SphereOrInvertedCat} and $\mathcal W[\Sigma_{\gamma^*}]=8\pi$, we deduce that $\Sigma_{\gamma^*}$ is the inversion of a scaled and vertically translated catenoid and thus encloses a positive volume, which contradicts Equation \eqref{areaANDvolume}.

\paragraph{The Second Case: $\gamma^*\not\in\mathcal P$}\ \\
Let $\alpha^{(1)}:=\gamma^*|_{[0,\tau]}$ and $\alpha^{(2)}:=\gamma^*|_{[\tau,1]}$. Lemma \ref{WillmoreLowerBound} gives $\mathcal W[\alpha^{(i)}]\geq 4\pi$ for $i=1,2$. This shows
\begin{equation}\label{twospheresequation}
8\pi\leq \mathcal W[\alpha^{(1)}]+ \mathcal W[\alpha^{(2)}]= \mathcal W[\gamma^*]\leq 8\pi.
\end{equation}
Hence $\mathcal W[\gamma^*]=8\pi$ and in view of Estimate \eqref{limitWillmoreupperbound} we get $\beta(\sigma_k)=\mathcal W[\gamma^{(k)}]\rightarrow8\pi$. Additionally, Estimate \eqref{twospheresequation} gives $\mathcal W[\alpha^{(i)}]= 4\pi$ for $i=1,2$ so that by the Li-Yau inequality (Lemma \ref{WillmoreLowerBound}), both curves are the profile curve of a sphere. Since $\alpha^{(i)}\in\mathcal P$, are parameterized proportional to arc length, there exist $z_i\in\R$, $\epsilon_i\in\set{\pm1}$ and $\omega_i, R_i\in\R_0^+$, such that
\begin{align*}
    &\alpha^{(1)}:[0,\tau]\rightarrow \mathcal H^2,\ \alpha^{(1)}(t)=(0, z_1)+ R_1(\sin(\omega_1 t), \epsilon_1\cos(\omega_1 t)),\\
   &\alpha^{(2)}:[\tau,1]\rightarrow \mathcal H^2,\ \alpha^{(2)}(t)=(0, z_2)+ R_2(\sin(\omega_2(t-\tau)), \epsilon_1\cos(\omega_2(t-\tau))).
\end{align*}
Using Equation \eqref{areaANDvolume} we get
\begin{align*}
    1&=A[\gamma^*]=A[\gamma^ {(1)}]+A[\gamma^ {(2)}]=4\pi(R_1^2+R_2^2),\\
    0&=V[\gamma^*]=V[\gamma^ {(1)}]+V[\gamma^ {(2)}]=\frac43\pi|\epsilon_1R_1^3+\epsilon_2 R_2^3|.
\end{align*}
These imply $\epsilon_2=-\epsilon_1=:-\epsilon$ and $R_1=R_2=(8\pi)^{-\frac12}=:R$. Since $\gamma^*$ is parameterized proportional to arc length, we have $\omega_1 R=|\dot\alpha^{(1)}|=|\dot\gamma^*|=|\dot\alpha^{(2)}|=\omega_2 R$
and hence $\omega_1=\omega_2=:\omega$. Note $\alpha^{(1)}_1(t)>0$ for $t\in(0,\tau)$ and $\alpha^{(1)}_1(\tau)=0$. This proves $\omega\tau=\pi$. Similarly  $\alpha^{(2)}_1(t)>0$ for $t\in(\tau,1)$ and $\alpha^{(2)}_1(1)=0$ give $\omega(1-\tau)=\pi$. Therefore $\omega=2\pi$ and $\tau=\frac12$. By continuity of $\gamma^*$ at $\tau$ 
$$z_1+\epsilon R\cos(\omega\tau)=z_2-\epsilon R\cos(\omega(\tau-\tau)).$$
Using $\omega=2\pi$ and $\tau=\frac12$ we get $z_1=z_2=:z$. As $\gamma^*(0)=(0,0)$, we get $\alpha_2^{(1)}(0)=0$ and hence $z=-\epsilon R$.  Finally, since $\gamma^{(k)}\in\mathcal F^{0+}_{\sigma_k}$ we get
$$0\leq \lim_{k\rightarrow\infty}\int_0^1\gamma_2^{(k)}(t)dt
=\int_0^1\gamma_2^*(t)dt
=\frac{z_1+z_2}2
=z
=-\epsilon R.$$
So $\epsilon=-1$ and $\gamma^*=\kappa$.
\end{proof} 

It remains to study the asymptotic behaviour of $\gamma^ {(k)}$ near $t=0$ and $t=1$ when $k\rightarrow\infty$. 
\begin{korollar}\label{Lagrangemultipliergoto0Lemma}
We denote by $\Lambda_k$ the Lagrange multiplier from Lemma \ref{eulerlagrangeequation}. Then $\Lambda_k\rightarrow 0$ as $k\rightarrow\infty$
\end{korollar}
\begin{proof}
By Lemma \ref{eulerlagrangeequation}
\begin{equation}\label{LambdakGoestoZeroEquation}
|\Lambda_k|=\left|\frac{2W[\gamma^{(k)}]}{\pm_{\gamma^{(k)}} 4\sqrt\pi-\sigma_k H[\gamma^{(k)}]}\right|.
\end{equation}
Let $I:=(\frac18,\frac38)$. By Lemma \ref{Lemma5point4} we have $\gamma^{(k)}\rightarrow\kappa$ in $C^r(I)$ for all $r\in\N$ and hence  
$$W[\gamma^{(k)}]\bigg|_{I}\rightarrow W[\kappa]\bigg|_{I}\equiv 0
\hspace{.5cm}\textrm{and}\hspace{.5cm}
H[\gamma^{(k)}]\bigg|_{I}\rightarrow H[\kappa]\bigg|_{I}\equiv 2.$$
In particular, $\sigma_k H[\gamma^{(k)}]\rightarrow0$ on $I$ and consequently $\Lambda_k\rightarrow 0$ by Equation \eqref{LambdakGoestoZeroEquation}.
\end{proof}

For all $k\in\N$ we define 
$$\tau_k:=\sup\{t\geq 0\ |\ \dot\gamma_1^ {(k)}(s)>0\hspace{.2cm}\textrm{for all $0\leq s\leq t$}\}.$$
We prove that $\tau_k$ is bounded from below. In fact, we even prove the following fact:

\begin{lemma}\label{lowerboundeddotgamma1Lemma}
   Let $\epsilon>0$. There exist $k_0(\epsilon)\in\N$ and $t_0(\epsilon)>0$ such that 
   $$\dot\gamma_1^ {(k)}(t)\geq L[\gamma^{(k)}](1-\epsilon)\hspace{.5cm}\textrm{for all $k\geq k_0(\epsilon)$ and }t\in[0, t_0(\epsilon)].$$
\end{lemma}
\begin{proof}
    The proof is by contradiction. Assuming that the lemma is false, there exists $\epsilon>0$ such that after passing to a subsequence, we have $\dot\gamma_1^{(k)}(t_k)<L_k(1-\epsilon)$ for a sequence $t_k\rightarrow 0^+$ and where $L_k:=L[\gamma^{(k)}]$. By Lemma \ref{Lemma5point4} we have $\gamma^ {(k)}\rightarrow\kappa$ in $W^ {1,2}((0,1))$ and in $C^\infty_{\operatorname{loc}}((0,1)\backslash\{\frac12\})$ and therefore $L_k\rightarrow L_*:=L[\kappa]$.  For small $\delta>0$, we recall $I_\delta:=(\delta,\frac12-\delta)\cup(\frac12+\delta,1-\delta)$.
    Denoting the principal curvatures of $\kappa$ by $k_i^*$ we have
    \begin{equation}\label{Limitdeltarightarrow0eq}
   \lim_{\delta\rightarrow0^+}\frac{2\pi}4\int_{I_\delta}(k_1^*+k_2^*)^2 \kappa_1|\dot\kappa| dt
   = \mathcal W[\kappa]
   = 8\pi. 
    \end{equation}    
    For each $\delta>0$ there exists $k_0(\delta)>0$ such that $t_k<\delta$ for all $k\geq k_0(\delta)$. Using that $\mathcal W[\gamma^{(k)}]=\beta(\sigma_k)\leq8\pi$  we estimate 
    \begin{align}
    8\pi\geq \mathcal W[\gamma^{(k)}]=&\frac{2\pi}4\int_0^1 (k_1^{(k)}+k_2^{(k)})^2 \gamma_1^{(k)}|\dot\gamma^{(k)}|dt\nonumber\\
    \geq &\frac{2\pi}4\int_{I_\delta} (k_1^{(k)}+k_2^{(k)})^2 \gamma_1^{(k)}|\dot\gamma^{(k)}|dt+\frac{2\pi}4\int_{0}^{t_k} (k_1^{(k)}+k_2^{(k)})^2 \gamma_1^{(k)}|\dot\gamma^{(k)}|dt.\label{8piContradictionsetupspellmistake}
    \end{align}
     We derive a suitable estimate for the second integral. By Corollary \ref{gammasmoothuptoBoundaryCorollary} we know $\gamma^ {(k)}\in C^ \infty([0, t_k])$ and using Equation \eqref{ArcLength_K_Identity}, we get
     $$k_1^{(k)}k_2^{(k)}\gamma_1^ {(k)}|\dot\gamma^ {(k)}|=-\frac1{L_k}\ddot\gamma_1^ {(k)}\hspace{.5cm}\textrm{for $t\in[0,t_k]$}.$$
     Using this identity, $\dot\gamma_1^ {(k)}(0)=L_k$ and $\dot\gamma_1^ {(k)}(t_k)\leq L_k(1-\epsilon)$ we may compute
    \begin{align*}
    &\int_0^{t_k}(k_1^{(k)}+k_2^{(k)})^2 \gamma_1^{(k)}|\dot\gamma^{(k)}|dt
    \geq 
    2\int_0 ^{t_k}k_1^{(k)}k_2^{(k)} |\dot\gamma^{(k)}|\gamma_1^{(k)}dt
    =
-\frac2{L_k}\int_0^{t_k}\ddot\gamma_1^{(k)}(t) dt\\
=&-\frac2{L_k}\left(\dot\gamma_1^ {(k)}(t_k)-\dot\gamma_1^ {(k)}(0)\right)\geq  2L_k-2(L_k-\epsilon)
    =2\epsilon.
    \end{align*}
    Inserting into Estimate \eqref{8piContradictionsetupspellmistake}, we deduce that for all $\delta>0$ and $k\geq k_0(\delta)$ we have  
    $$
    8\pi\geq \frac{2\pi}4\int_{I_\delta} (k_1^{(k)}+k_2^{(k)})^2 \gamma_1^{(k)}|\dot\gamma^{(k)}|dt+\pi\epsilon.
   $$
   Since $\gamma^{(k)}\rightarrow \kappa$ in $C^r(I_\delta)$ for all $r\in\N$, we can send $k\rightarrow\infty$ and deduce that for all $\delta>0$
   $$
    8\pi\geq \frac{2\pi}4\int_{I_\delta} (k_1^*+k_2^*)^2 \kappa_1|\dot\kappa|dt+\pi\epsilon.
   $$
   In view of Equation \eqref{Limitdeltarightarrow0eq}, we arrive at a contradiction by letting $\delta\rightarrow0^+$. 
\end{proof}

Putting $L_*:=L[\kappa]$ we have $L_k\rightarrow L_*$. So, by Lemma \ref{lowerboundeddotgamma1Lemma}, there exists $k_0>0$ and $t_0>0$ such that for all $k\geq k_0$ and $t\in[0, t_0]$ we have $\dot\gamma_1^ {(k)}(t)\geq\frac14L_*>0$.
Consequently, each $\gamma_1^ {(k)}|_{[0, t_0]}$ is a smooth diffeomorphism onto its image. We also note that 
$$\gamma_1^ {(k)}(t_0)=\int_0^{t_0}\dot\gamma_1^ {(k)}(t)dt\geq \frac{t_0 L_*}4=:r_0.$$
This implies that for all $k\geq k_0$
\begin{equation}\label{definitionofuk}
u_k(r):[0, r_0]\rightarrow\R,\ u_k(r):=\gamma_2^{(k)}\circ \left(\gamma_1^{(k)}\right)^{-1}(r)
\end{equation}
is well-defined and smooth. Since $\gamma^{(k)}\rightarrow\kappa$ in $C^\infty_{\operatorname{loc}}((0,1)\backslash\{0\})$, we deduce that 
\begin{equation}\label{LimitofGraphfunction01}
u_k(r)\rightarrow u^*(r):= R-\sqrt{R^2-r^2}\hspace{.5cm}\textrm{ in $C^\infty_{\operatorname{loc}}((0, r_0]))$}\hspace{.2cm}\textrm{ where $R=\frac1{\sqrt{8\pi}}$.}
\end{equation}
We wish to prove that the convergence is even in $C^\infty([0, r_0]))$. 

\begin{korollar}\label{Asymptotics_Derivativesmall}
   For all $\epsilon>0$ there exist $k_0(\epsilon)\in\N$ and $r_1(\epsilon)\in(0, r_0]$ such that $|u_k'(r)|\leq\epsilon$ for all $k\geq k_0(\epsilon)$ and $r\in[0, r_1(\epsilon)]$.
\end{korollar}
\begin{proof}
For small $\epsilon>0$ we take $k_0(\epsilon)$ and $t_0(\epsilon)$ as in Lemma \ref{lowerboundeddotgamma1Lemma} and restrict ourselves to $k\geq k_0(\epsilon)$. We put $\rho_k(\epsilon):=\gamma_1^ {(k)}(t_0(\epsilon))$. Note that for $r\in[0,\rho_k(\epsilon)]$ we have $(\gamma^ {(k)}_1)^ {-1}(r)\in[0, t_0(\epsilon)]$. Hence $\dot\gamma_1^ {(k)}((\gamma^ {(k)}_1)^ {-1}(r))\geq L_k(1-\epsilon)$ and therefore 
$$\left|\dot\gamma_2^{(k)}\left(\left(\gamma^ {(k)}_1\right)^ {-1}(r)\right)\right|=\left(L_k^2-\left(\dot\gamma_1^{(k)}\left(\left(\gamma^ {(k)}_1\right)^ {-1}(r)\right)\right)^2\right)^{\frac12}\leq L_k\sqrt{1-(1-\epsilon)^2}.$$
So, for all $k\geq k_0(\epsilon)$ and  $r\leq \rho_k(\epsilon)$ we get 
    $$|u_k'(r)|=\left|\dot\gamma_2^ {(k)}\left(\left(\gamma_1^{(k)}\right)^{-1}(r)\right)\cdot\frac1{\dot\gamma_1^ {(k)}\left(\left(\gamma_1^{(k)}\right)^{-1}(r)\right)}\right|
    \leq \frac{L_k\sqrt{2\epsilon-\epsilon^2}}{L_k(1-\epsilon)}=\frac{\sqrt{2-\epsilon}}{1-\epsilon}\sqrt\epsilon.$$
 This implies that the lemma is proven as soon as a uniform lower bound for $\rho_k(\epsilon)$ is established. To do so, we use the defining property of $t_0(\epsilon)$ from Lemma \ref{lowerboundeddotgamma1Lemma} and estimate  
$$\rho_k(\epsilon)=\gamma_1^{(k)}(t_0(\epsilon))=\int_0^{t_0(\epsilon)}\dot\gamma_1^{(k)}(t) dt\geq t_0(\epsilon) L_k(1-\epsilon).$$
Since $L_k\rightarrow L_*>0$ we have $\inf_k L_k>0$ and so $\rho_k(\epsilon)\geq r_1(\epsilon)>0$ independent of $k$. 
\end{proof}

Corollary \ref{Asymptotics_Derivativesmall} allows us to essentially repeat the proof of Lemma \ref{regularityatend01} to establish a uniform bound for $u_k''$.

\begin{lemma}\label{AsymptoticsSecondDerivativeBound}
There exists $C>0$ and $r_0>0$ such that $|u_k''(r)|\leq C$ for all $k\in\N$ and $r\in[0, r_0]$.
\end{lemma}
\begin{proof}
As each individual $u_k\in C^\infty([0, r_0))$, we only need to prove a uniform bound $|u_k''|\leq C$ for all $k\geq k_1$ with some arbitrary $k_1$. To do so, we follow the proof of Lemma \ref{regularityatend01} and put $w_k:=u_k'$ and $v_k:=\sqrt{1+w_k^2}$. In view of Equation \eqref{ODEamRand}, we have 
\begin{equation}\label{ODEamRand_Neu}
w_k''+\left(\frac{w_k}{r}\right)'-\varphi(w_k)
=
\frac12 v_k^ 5 a_k r+b_kv_k^ 4w_k.
\end{equation}
Here we used that $\lambda_k=0$ by Lemma \ref{uisc2}. Also we note that $a_k=\pm_{\gamma^{(k)}}4\sqrt\pi \Lambda_k$ and $b_k=\sigma_k\Lambda_k$. Corollary \ref{Lagrangemultipliergoto0Lemma} gives $\Lambda_k\rightarrow0$ as $k\rightarrow \infty$ so that $a_k$ and $b_k$ are bounded. Additionally, Corollary \ref{Asymptotics_Derivativesmall} and $u_k\rightarrow u^*$ in $C^\infty((0, r_0])$ together imply that $w_k=u_k'$ is bounded.\\

\noindent
\textbf{Step 1}\ \\
Note  $w_k\varphi(w_k)\geq 0$ by the definition of $\varphi(w)$ in Equation \eqref{varphiofwdef}. We multiply Equation \eqref{ODEamRand_Neu} by $w_k$, choose any $\rho\in (0, r_0)$  and integrate to get
$$\int_r ^ \rho w_k''w_k+\left(\frac {w_k}t\right)' w_k dt
\geq \frac {a_k}2 \int_ r^ \rho v_k^ 5 t w_k dt+
b_k\int_r^ \rho v_k^ 4 w_k^2 dt\geq -C.$$
$C$ is uniform over $k$ as $a_k$, $b_k$ and $w_k$ are bounded. Performing the same calculations that lead to Estimate \eqref{step1result}, we derive
\begin{align}
\int_r^ \rho (w_k'(t))^2+\frac{ w_k(t)^2}{2t^2}dt
&\leq C(\rho)+\frac{w_k(r)^2}{2r}-w_k(r)\frac{w_k(r) +rw_k'(r)}r\nonumber\\
&= C(\rho)+\frac{w_k(r)^2}{2r}-w_k(r)\frac{(r w_k(r))'}r.\label{step1result_Neu}
\end{align}
The constant $C(\rho)$ includes the values $w_k(\rho)$ and $w_k'(\rho)$ but is uniform over $k$. Indeed, as $u_k\rightarrow u^*$ in $C^\infty_{\operatorname{loc}}((0, r_0])$, we have $w_k(\rho)\rightarrow (u^*)'(\rho)$ and $w_k'(\rho)\rightarrow  (u^*)''(\rho)$.\\

\noindent
\textbf{Step 2}\ \\
Let again $\rho\in(0,r_0)$. This time we integrate Equation \eqref{ODEamRand_Neu} directly and use $w_k'+\frac {w_k}r=\frac1r(w_kr)'$ to get 
\begin{equation}\label{step2begin_Neu}
\frac {(w_k t)'}t\bigg|_{t=\rho}-\frac {(w_k r)'}r=\int_ r^ \rho \varphi(w_k)(t)+\frac12 a_k v_k^ 5 t + b_k v_k^ 4w_kdt.
\end{equation}
Given $\epsilon>0$, we use Corollary \ref{Asymptotics_Derivativesmall} and choose $k_0(\epsilon)\in\N$ and $\rho(\epsilon):=r_1(\epsilon)\in(0, r_0]$ such that $|w_k(r)|\leq \epsilon$ for all $k\geq k_0$ and $r\in(0, \rho(\epsilon))$. This allows us to perform the same manipulations that lead to Estimate \eqref{step2result} and get
\begin{align}
    \left|\frac{(r w_k(r))'}{r}\right| 
    &\leq C(\epsilon)+\epsilon \int_r^ \rho w_k'(t)^2 +\frac{w_k(t)^2}{2t^2} dt\hspace{.5cm}\textrm{for all }r\in(0, \rho(\epsilon)).\label{step2result_Neu}
\end{align}
Again $C(\epsilon)$ contains the boundary values $w_k(\rho(\epsilon))$ and $w_k'(\rho(\epsilon))$ which are, however, controlled uniformly over $k$.\\ 

\noindent
\textbf{Step 3}\ \\
Let $\epsilon>0$ be arbitrary. By taking $\rho=\rho(\epsilon)$ small enough in Estimate \eqref{step1result_Neu} we may insert Estimate \eqref{step2result_Neu} to get 
\begin{align*} 
\int_ r^ {\rho}w_k'(t)^2+\frac{w_k(t)^2}{2t^2} dt&\leq C(\epsilon)+\frac{w_k(r)^2}{2r}
+|w_k(r)|\left(C(\epsilon)+\epsilon \int_r^ \rho w_k'(t)^2 +\frac{w_k(t)^2}{2t^2} dt\right).
\end{align*}
By potentially shrinking $\rho$ we can assume without loss of generality that $|w_k|\leq \epsilon$ on $[0,\rho]$ and so, in particular $|w_k(r)|\leq \epsilon$. Choosing $\epsilon$ small enough, we can absorb the integral on the right to the left and obtain that for some small enough $\rho$ and $r\in(0,\rho)$
\begin{equation}\label{step3result_Neu}
\int_r^ {\rho}w_k'(t)^2+\frac{w_k(t)^2}{2t^2} dt
\leq
C\left(\frac{w_k(r)^2}{2r}+|w_k(r)|+1\right).
\end{equation}

\noindent
\textbf{Step 4}\ \\
Let $\rho$ be small enough so that Estimate \eqref{step3result_Neu} holds. Inserting Estimate \eqref{step3result_Neu} into Estimate \eqref{step2result_Neu} with $\epsilon=1$ and using that $|w_k|$ is bounded uniformly, we get
\begin{align*} 
\left|\frac{(r w_k(r))'}{r}\right| 
    \leq &
    C+C\left(\frac{w_k(r)^2}{2r}+|w_k(r)|+1\right) 
    \leq 
    C\left(\frac{w_k(r)^2}{2r}+1\right).
\end{align*}
We multiply by $r$ and choose any $\mu\in(0,1)$. By using Corollary \ref{Asymptotics_Derivativesmall} and potentially shrinking $\rho$ we can assume that $|w_k|\leq C^ {-1}\mu$ on $[0,\rho]$. Hence 
$$|(r w_k)'|\leq C(w_k(r)^2+r)\leq \mu |w_k|+Cr.$$ 
Using the same justification as in the proof of Lemma \ref{regularityatend01}, we compute
$$r|w_k|'+|w_k|=(r|w_k|)'=|rw_k|'\leq |(rw_k)'|\leq \mu|w_k|+Cr\hspace{.5cm}\textrm{for all $r\in(0,\rho]$},$$
multiply with $r^ {-\mu}$ and get 
$$(r^ {1-\mu} |w_k|)'=r^ {1-\mu}|w_k|'+(1-\mu) r^ {-\mu}|w_k|\leq C r^ {1-\mu} .$$
We integrate this inequality from $0$ to $r$ and get
\begin{equation}\label{step4result_Neu}
r^ {1-\mu}|w_k(r)|\leq C\int_0^ r t^ {1-\mu} dt\leq C(\mu) r^{2-\mu}
\hspace{.5cm}\textrm{and hence}\hspace{.5cm}
|w_k(r)|\leq Cr.
\end{equation}

We now deviate from the proof of Lemma \ref{regularityatend01}.\\

\noindent
\textbf{Step 5}\ \\
Inserting \eqref{step4result_Neu} into \eqref{step3result_Neu}, we get
$$\int_r^ {\rho}w_k'(t)^2+\frac{w_k(t)^2}{2t^2} dt
\leq
C.$$
Inserting this into Estimate \eqref{step2result_Neu} with $\epsilon=1$, we obtain $|(w_k(r) r)'|\leq Cr$. Combining this with Estimate \eqref{step4result_Neu}, we get 
$$r|w_k'(r)|\leq |rw_k'(r)+w_k(r)|+|w_k(r)|\leq \left|(w_k(r) r)'\right|+|w_k(r)|\leq Cr.$$
Hence $w_k'(r)$ is bounded on $[0, \rho]$. The lemma follows by recalling that $u_k\rightarrow u^*$ in $C^\infty([\rho, r_0])$.
\end{proof}

Corollary \ref{Asymptotics_Derivativesmall} and Lemma \ref{AsymptoticsSecondDerivativeBound} prove that $u_k'$ and $u_k''$ are bounded. To deduce $u_k\rightarrow u^*$ in $C^\infty([0, r_0])$, we prove that all higher derivatives are bounded as well.

\begin{lemma}\label{ukfinalconvergencecorollary}
$u_k\rightarrow u^*$ in $C^\infty([0, r_0])$.
\end{lemma}
\begin{proof}
    In view of Equation \eqref{LimitofGraphfunction01} it suffices to prove that $(u_k)\subset C^m((0, r_0))$ is bounded for all $m\geq 0$. We have $\gamma^{(k)}\rightarrow\kappa$ in $C^0([0,1])$ and hence $\gamma_2^{(k)}$ is a bounded sequence of  functions. By the definition of $u_k$ in Equation \eqref{definitionofuk} we deduce that $(u_k)\subset C^0([0, r_0])$ is bounded. Combining this with
        Corollary \ref{Asymptotics_Derivativesmall} and Lemma \ref{AsymptoticsSecondDerivativeBound}, we obtain that $(u_k)\subset C^2([0, r_0])$ is bounded.\\
        
    We put $w_k:=u_k'$ and $v_k:=\sqrt{1+w_k^2}$. By Equation \eqref{ODEamRand_Neu} 
    $$w_k''+\left(\frac{w_k'}r\right)'=\frac12 v_k^5 a_k r+b_k v_k^4 w_k+\frac{5w_k}{2(1+w_k^2)}(w_k')^2+\frac{w_k^3}{2r^2}(3+w_k^2)=:f_k(r).$$

    To prove that $(w_k)\subset C^m([0, r_0])$ is bounded for all $m\geq 1$ we argue by induction. $m=1$ has already been established. Assuming that $(w_k)\subset C^m([0, r_0])$ is bounded for some $m\geq 1$, we first use $w_k(0)=u_k'(0)=0$ to rewrite
    \begin{equation}\label{Newformulaforf}
    f_k(r)=
    \frac12 v_k^5 a_k r+b_k v_k^4w_k +\frac{5w_k}{2(1+w_k^2)}(w_k')^2+\frac{w_k}{2}(3+w_k^2)\left(\int_0^1 w_k'(s r) ds\right)^2.
    \end{equation}
    As $(w_k)\subset  C^{m}([0, r_0])$ is bounded, Equation \eqref{Newformulaforf} implies that $(f_k)\subset C^{m-1}([0, r_0])$ is bounded. Also $f\in C^\infty((0, r_0])$. So, using Lemma \ref{uniformboundedLemma} we deduce that $(w_k^{(m+2)})\subset C^0([0, r_0])$ is bounded. Noting that $w_k^{(m+1)}(r_0)$ is bounded since $u_k\rightarrow u^*$ in $C^\infty_{\operatorname{loc}}((0, r_0])$, we may use Lemma \ref{uniformboundedLemma} to get that $w_k^{(m+1)}\subset C^0([0,r_0])$ is bounded and the inductive step follows.\\

    Since $u_k'=w_k$ and $(u_k)\subset C^0([0, r_0])$ is bounded, we have shown that $(u_k)\subset C^m([0, r_0])$ is bounded for all $m\in\N_0$. 
\end{proof}

Finally, we prove Theorem \ref{asymptoticsthm} In view of Lemma \ref{Lemma5point4}, we only need to prove that $\gamma^{(k)}$ is bounded in $C^m([0,t_0]\cup[1-t_0,1])$ for some small $t_0>0$ and all $m\in\N_0$.\\

\noindent
\textit{Proof of Theorem \ref{asymptoticsthm}}\ \\
Let $L_*:=L[\kappa]$ and $t_0>0$ such that $\gamma^{(k)}_1|_{[0,t_0]}$ are smooth diffeomorphisms for all $k\in\N$. By Lemma \ref{lowerboundeddotgamma1Lemma}, we may assume $\dot\gamma_1^{(k)}(t)\geq\frac{L_*}2>0$ for all $k\geq k_0$ and $t\in[0, t_0]$ after potentially shrinking $t_0$. We establish that $(\gamma^{(k)})\subset C^m([0, t_0])$ is bounded for all $m\in\N$ by an inductive argument. For $m=0$ this follows from the fact that $\gamma^{(k)}\rightarrow\kappa$ in $C^0([0,1])$. Now for the inductive step. Suppose that for some $m\in\N_0$ we already know that $\gamma^{(k)}\subset C^m([0, t_0])$ is bounded.
Recalling $\gamma_2^{(k)}(t)=u_k(\gamma_1^{(k)}(t))$ and $L_k^2=|\dot\gamma^{(k)}(t)|^2$, we deduce 
$$L_k^2-\left(\dot\gamma_1^{(k)}(t)\right)^2=(u_k'(\gamma_1^{(k)}(t))^2\left(\dot\gamma_1^{(k)}(t)\right)^2
\hspace{.5cm}\textrm{and so}\hspace{.5cm}
\left(\dot\gamma_1^{(k)}(t)\right)^2=\frac{L_k^2}{ 1+ u_k'(\gamma_1^{(k)}(t))^2}.$$
Using that $\dot\gamma_1^{(k)}(t)\geq 0$ for all $t\in[0, t_0]$ we can take the square root and obtain
\begin{equation}\label{gamma1odewithu}
    \dot\gamma_1^{(k)}(t)=\frac{L_k}{\sqrt{1+u_k'(\gamma_1^{(k)}(t))^2}}\hspace{.5cm}\textrm{when $t\in[0, t_0]$}.
\end{equation}
By Lemma \ref{ukfinalconvergencecorollary}
$(u_k)\subset C^l([0, r_0])$ is bounded for all $l\in\N$. Using Equation \eqref{gamma1odewithu} it then follows that $\gamma_1^{(k)}\subset C^{m+1}([0, t_0])$ is bounded. Since $\gamma_2^{(k)}=u_k\circ \gamma_1^{(k)}$, we deduce that $\gamma^{(k)}$ is bounded in $C^{m+1}([0, t_0])$, which completes the inductive step.\\

Finally, we note that a similar discussion proves that $(\gamma^{(k)})$ is bounded $C^m([1-t_1,1])$ for some (potentially small) $t_1>0$ and all $m\in\N$. Theorem \ref{asymptoticsthm} follows.  
\qed

\section{Formation of the Catenoidal Neck}\label{catenoidNeckSection}
Throughout this section let $\sigma_k\rightarrow 0^+$ and $\gamma^{(k)}\in\mathcal F_{\sigma_k}^{0+}$ such that $\mathcal W[\gamma^{(k)}]=\beta(\sigma_k)$. Additionally, let $L_k:=L[\gamma^{(k)}]$ and
$$\epsilon_k:=\inf_{t\in[\frac18,\frac78]}\gamma_1^{(k)}(t).$$
\begin{lemma}\label{catenoidLemma01}
    \begin{enumerate}[(1)]
    \item If $(t_k)\subset [\frac18,\frac78]$ is a sequence of times such that $\gamma_1^{(k)}$ has a local minimum at $t_k$, then $t_k\rightarrow\frac12$ and $\gamma_1^{(k)}(t_k)\rightarrow 0$. Moreover, there exists $k_0\in\N$ such that the minima are strict for all $k\geq k_0$.
    \item  $\epsilon_k\rightarrow 0^+$ as $k\rightarrow{\infty}$. Additionally, $\tau_k\rightarrow\frac12$ for every sequence $(\tau_k)\subset [\frac18,\frac78]$ satisfying $\gamma_1^{(k)}(\tau_k)=\epsilon_k$.
    \end{enumerate} 
\end{lemma}
    \begin{proof}
   Clearly, the second part is a special case of the first part. \\
   
        Now let $(t_k)\subset[\frac18,\frac78]$ be a sequence of local minima for $\gamma_1^{(k)}$.  After passing to a subsequence, we have $t_k\rightarrow t^*\in[\frac18,\frac78]$. We prove that $t^*=\frac12$ by contradiction. Clearly $\dot\gamma_1^{(k)}(t_k)=0$ and $\ddot\gamma_1^{(k)}(t_k)\geq 0$. If $t^*\neq\frac12$, we can use $\gamma^{(k)}\rightarrow\kappa$ in $C^\infty([0,1]\backslash\{\frac12\})$ to obtain $\dot\kappa_1(t_*)=0$ and $\ddot\kappa_1(t^*)\geq 0$. However, since $\kappa_1(t)=|\sin(2\pi t)|$, this is impossible. So $t^*=\frac12$ and 
        $$h_k:=\gamma_1^{(k)}(t_k)\rightarrow\kappa_1(\frac12)=0.$$
        Finally, we prove that for large $k$, the minima at $t_k$ are strict. For each $k$ we have $\dot \gamma_1^{(k)}(t_k)=0$ and $\ddot\gamma_1^{(k)}(t_k)\geq 0$. If $\ddot\gamma_1^{(k)}(t_k)>0$, the local minimum at $t_k$ is strict. If $\ddot\gamma_1^{(k)}(t_k)=0$ we also have $\dddot \gamma_1^{(k)}(t_k)=0$ and $\ddddot\gamma_1^{(k)}(t_k)\geq 0$ as otherwise $t_k$ could not be a local minimum of $\gamma_1^{(k)}$. If $\dot\gamma_1^{(k)}(t_k)=\ddot\gamma_1^{(k)}(t_k)=\dddot\gamma_1^{(k)}(t_k)=0$ and $\ddddot\gamma_1^{(k)}(t_k)>0$ the local minimum at $t_k$ is strict.  We prove by contradiction that 
        \begin{equation}\label{deriavtives0contrdiction}
        \dot\gamma_1^{(k)}(t_k)=\ddot\gamma_1^{(k)}(t_k)=\dddot\gamma_1^{(k)}(t_k)=\ddddot\gamma_1^{(k)}(t_k)=0\end{equation}
         along a subsequence is impossible. So let us assume that we had \eqref{deriavtives0contrdiction} along a subsequence again denoted by $\gamma^{(k)}$. Since $\dot\gamma_1^{(k)}(t_k)=0$ we have $|\dot\gamma_2^{(k)}(t_k)|=L_k\neq 0$ so that differentiating $|\dot\gamma^{(k)}(t)|^2=L_k^2$ implies
         $\ddot\gamma_2^{(k)}(t_k)=\dddot\gamma_2^{(k)}(t_k)=\ddddot\gamma_2^{(k)}(t_k)=0$.     
         Using Equation \eqref{principalcurvatures} we get 
         $$k_1^{(k)}(t_k)=\dot k_1^{(k)}(t_k)=\ddot k_1^{(k)}(t_k)=0=\dot k_2^{(k)}(t_k)=\ddot k_2^{(k)}(t_k)
         \hspace{.5cm}\textrm{and}\hspace{.5cm}|k_2^{(k)}(t_k)|=\frac1{h_k}.$$
         This implies that $|H[\gamma^{(k)}](t_k)|=\frac1{h_k}$, $K[\gamma^{(k)}](t_k)=0$ and $(\Delta_g H)[\gamma^{(k)}](t_k)=0$. Inserting into the Euler-Lagrange equation from Lemma \ref{eulerlagrangeequation}, we get
         \begin{equation}\label{MoonZappa}
         \frac1{2h_k^3}=\frac12 |H[\gamma^{(k)}](t_k)|^3=2|W[\gamma^{(k)}](t_k)|=\left|\Lambda_k\left(\pm_k4\sqrt\pi\pm_k\frac{\sigma_k}{h_k}\right)\right|. 
        \end{equation}
         By Corollary \ref{Lagrangemultipliergoto0Lemma} we have $\Lambda_k\rightarrow 0$. So, Equation \eqref{MoonZappa} implies $\frac12=|\Lambda_k(\pm_k4\sqrt\pi h_k^3\pm_k\sigma_kh_k^2)|\rightarrow 0$ as $k\rightarrow\infty$, which is a contradiction.        
    \end{proof}

By Theorem \ref{asymptoticsthm}, the curves $\gamma^{(k)}$ asymptotically look like two circles. The fact that these are joined by a catenoidal shape is derived by showing that near $t=\frac12$, they solve approximately $H[\gamma^{(k)}]=0$. The precise formulation of this observation is given in the following lemma.

\begin{lemma}\label{MinimalSurfaceLemma}
    Let $a_k\uparrow\frac12$ and $b_k\downarrow\frac12$. Then 
    $$\lim_{k\rightarrow\infty}\mathcal W\left[\gamma^{(k)}\big|_{[a_k, b_k]}\right]
    =
    \frac\pi2\lim_{k\rightarrow\infty}\int_{a_k}^{b_k}H[\gamma_k]^2\gamma_1^{(k)}|\dot\gamma^{(k)}|dt=0.$$
\end{lemma}
\begin{proof}
    The proof is by contradiction. Assume that along a subsequence again denoted by $\gamma^{(k)}$ we had 
    \begin{equation}
        \frac\pi2\int_{a_k}^{b_k}H[\gamma_k]^2\gamma_1^{(k)}|\dot\gamma^{(k)}|dt\geq \epsilon^*>0.
    \end{equation}
    For $\delta>0$ we put $I_\delta:=[0,\frac12-\delta)\cup(\frac12+\delta, 1]$. For each $\delta>0$ there exists $k_0(\delta)$ such that $I_\delta\cap(a_k, b_k)=\emptyset$ for all $k\geq k_0(\delta)$. Note that $\mathcal W[\gamma^{(k)}]=\beta(\sigma_k)\leq8\pi$. Taking any $\delta>0$ and $k\geq k_0(\delta)$, we can therefore estimate
    \begin{align*} 
    8\pi\geq &\mathcal W[\gamma^{(k)}]\\
    =&\frac\pi2\int_0^1 H[\gamma^{(k)}]^2 |\dot\gamma^{(k)}|\gamma_1^{(k)}dt\\
    \geq & \epsilon^*+\frac\pi2\int_{I_\delta} H[\gamma^{(k)}]^2 |\dot\gamma^{(k)}|\gamma_1^{(k)}dt.
    \end{align*}
    By Theorem \ref{asymptoticsthm}, we have $\gamma^{(k)}\rightarrow\kappa$ in $C^\infty(I_\delta)$ as $k\rightarrow\infty$. So, letting $k\rightarrow\infty$, we obtain 
    $$8\pi\geq \epsilon^*+\frac\pi2\int_{I_\delta} H[\kappa]^2 |\dot\kappa|\kappa_1dt=\epsilon^*+\mathcal W\left[\kappa\big|_{I_\delta}\right].$$
    We have $\mathcal W[\kappa]=8\pi$, so, letting $\delta\rightarrow0^+$, we arrive at a contradiction. 
\end{proof}

The remaining arguments are essentially based on the following blow-up lemma.

\begin{lemma}[Blow-Up Lemma]\label{blowuplemma}
    Let $t_k\rightarrow\frac12$ and put $h_k:=\gamma_1^ {(k)}(t_k)$. Assume that there is $\rho_0>0$ independent of $k$ such that $\gamma_1^ {(k)}(t_k+h_kt)\geq \frac12 h_k$ for all $|t|\leq \rho_0$. Then 
    $$\Gamma^{(k)}(t):=\frac1{h_k}\left(\gamma^ {(k)}(t_k+h_k t)-(0,\gamma_2^ {(k)}(t_k))\right)$$
   is bounded in $C^m([-\rho,\rho])$ for all $\rho\in(0,\rho_0)$ and $m\in\N$. Additionally, if along some subsequence $\Gamma^{(k_l)}\rightarrow \Gamma^*$ in $C^2([-\rho,\rho])$, the limit $\Gamma^*$ satisfies $H[\Gamma^*]=0$. 
\end{lemma}
\begin{proof}
By Lemma \ref{eulerlagrangeequation}, the curves $\gamma^{(k)}$ satisfy an Euler Lagrange equation. Using the well-known scaling of the Willmore operator and the mean curvature, we get
\begin{align*}
W[\Gamma^{(k)}](t)=&h_k^3W[\gamma^{(k)}](t_k+h_k t)\\
    =&2\Lambda_k(\pm_{\gamma^{(k)}}4\sqrt\pi h_k^3-\sigma_k h_k^3H[\gamma^{(k)}](t_k+h_k t))\\
    =&2\Lambda_k(\pm_{\gamma^{(k)}} 4\sqrt\pi h_k^3-\sigma_k h_k^2H[\Gamma^{(k)}](t)).
\end{align*}
Using Equation \eqref{variationformulastandard} and  Equations \eqref{var2}, \eqref{var3}, \eqref{var4}, \eqref{var5} we deduce that for all $\phi\in C^\infty_0((-\rho,\rho))$
$$\delta\mathcal W[\Gamma^{(k)}]\phi=2\Lambda_k\sigma_k h_k^2 \delta A[\Gamma^{(k)}]\phi-8\sqrt\pi\Lambda_k h_k^3 \delta V[\Gamma^{(k)}]\phi.$$
 $h_k\rightarrow 0^+$ since $t_k\rightarrow\frac12$ and by Corollary \ref{Lagrangemultipliergoto0Lemma} we also know $\Lambda_k\rightarrow 0$. We have $|\dot\Gamma^{(k)}|=L_k\rightarrow L_*>0$ and hence $L_k+L_k^{-1}$ is bounded. By scaling invariance of the Willmore energy we have $\mathcal W[\Gamma^{(k)}|_{(-\rho_0,\rho_0)}]\leq \mathcal W[\gamma^{(k)}]\leq 8\pi$. Next we exploit the assumption $\gamma_1^{(k)}(t_k+h_k t)\geq\frac12 h_k$. First, by definition of $\Gamma^{(k)}$, it implies $\Gamma_1^{(k)}(t)\geq \frac12$ for all $|t|< \rho_0$. Second, using Lemma \ref{gamma1ddotgammainL1}, we get 
$$\int_{-\rho_0}^{\rho_0}|\ddot\Gamma^{(k)}(t)|^2 dt=h_k \int_{t_k-h_k\rho_0}^{t_k+h_k\rho_0}|\ddot\gamma(s)|^2ds\leq C.$$
Finally, since by construction $\Gamma^{(k)}(0)=(1,0)$ and $|\dot\Gamma^{(k)}|=L_k\leq C$, we obtain $|\Gamma^{(k)}|\leq C$. We have proven that the assumptions of Theorem \ref{GeneralRegularityTheorem} are satisfied with $M$ and $\kappa(\Gamma^{(k)};(-\rho,\rho))\geq \kappa_0>0$ independent of $k$. Therefore, the uniform $C^m$ bounds follow from Theorem \ref{GeneralRegularityTheorem}. \\

Now assume that after passing to a subsequence $\Gamma^{(k)}\rightarrow\Gamma^*$ in $C^2$. We use the invariance of the Willmore energy under scaling and apply Lemma \ref{MinimalSurfaceLemma} to compute 
  \begin{align*}
  \frac\pi2\int_{-\rho}^ {\rho}H[\Gamma^*]^2 |\dot\Gamma^*|\Gamma^*_1dt=&\frac\pi2\lim_{k\rightarrow\infty}
  \int_{-\rho}^ {\rho}H[\Gamma^{(k)}]^2 |\dot\Gamma^{(k)}|\Gamma^{(k)}_1dt\\
  =&\lim_{k\rightarrow\infty}\mathcal W[\Gamma^ {(k)}|_{[-\rho,\rho]}]\\
   =&\lim_{k\rightarrow\infty}\mathcal W[\gamma^ {(k)}|_{[t_k-h_k\rho,t_k+h_k\rho]}]\\
   =&0.
  \end{align*}
Since $\Gamma_1^*\geq \frac12$ on $(-\rho,\rho)$ and $|\dot\Gamma^*|=L_*> 0$ we deduce  $H[\Gamma^*]\equiv0$. 
\end{proof}

Next, we establish that for $k$ large enough, the minimal value $\epsilon_k$ is only attained once.

\begin{lemma}[Proof of Theorem \ref{CatenoidNeckTheorem}, parts (1) and (2)] \label{catenoidLemma02}\ \\
    There exists $k_0\in\N$ such that for all $k\geq k_0$ there exists a unique $\tau_k\in[\frac18,\frac78]$ satisfying $\gamma_1^{(k)}(\tau_k)=\epsilon_k$. Moreover, there exists $\rho_0>0$ such that $\dot\gamma_1(t)<0$ for $t\in(\tau_k-\rho_0,\tau_k)$ and $\dot\gamma_1(t)>0$ for $t\in(\tau_k, \tau_k+\rho_0)$.
\end{lemma}
Note that part (3) of Theorem \ref{CatenoidNeckTheorem} follows by combining Lemmas \ref{catenoidLemma01} and \ref{catenoidLemma02}.
\begin{proof}
    We choose a sequence $(\tau_k)\subset[\frac18,\frac78]$ such that $\gamma_1^{(k)}(\tau_k)=\epsilon_k$ and prove by contradiction that there exists $\rho_0>0$ such that $\dot\gamma_1^{(k)}(t)\neq 0$ for all $t\in(\tau_k-\rho_0,\tau_k+\rho_0)\backslash\set{\tau_k}$. Once this is shown, uniqueness follows since any other sequence $\tau_k'$ would satisfy $|\tau_k-\tau_k'|\rightarrow 0$ by Lemma \ref{catenoidLemma01}.\\
    
    Assuming the above claim was false and potentially passing to a subsequence again denoted by $\gamma^{(k)}$, we obtain a sequence $\tilde\tau_k\neq\tau_k$ satisfying $|\tilde\tau_k-\tau_k|\rightarrow 0$ and $\dot\gamma_1^{(k)}(\tilde\tau_k)=0$. For fixed $k$ either $\ddot\gamma_1^{(k)}(\tilde\tau_k)\leq 0$ or $\ddot\gamma_1^{(k)}(\tilde\tau_k)> 0$. Suppose that the second case is true. Then $\gamma_1^{(k)}$ has a local minimum at $\tilde\tau_k$.  By Lemma \ref{catenoidLemma01}, the minima are strict and hence we deduce that $\gamma_1^{(k)}$ must attain a local maximum at some $\hat\tau_k$ in between $\tau_k$ and $\tilde\tau_k$. Then $\dot\gamma^{(k)}(\hat\tau_k)=(0,\pm L_k)$ and $\ddot\gamma_1^{(k)}(\hat\tau_k)\leq 0$. In the first case,  $\tilde\tau_k$ has the same properties: $\dot\gamma^{(k)}(\tilde\tau_k)=(0,\pm L_k)$ and $\ddot\gamma_1^{(k)}(\tilde \tau_k)\leq 0$.\\
    
    So, in any case we have produced a sequence $\tau_k\neq \hat\tau_k\rightarrow\frac12$ such that after potentially passing to another subsequence we have $\ddot\gamma_1^{(k)}(\hat\tau_k)\leq 0$ and $\dot\gamma_1^{(k)}(\hat\tau_k)=(0,\pm L_k)$. We will assume that $+$ is the correct sign. If `$-$' is correct, the argument needs no modification.\\
    
     Since $\gamma^{(k)}\rightarrow\kappa$ in $C^0([0,1])$ we deduce $h_k:=\gamma_1^{(k)}(\hat\tau_k)\rightarrow \kappa_1(\frac12)=0$.
    For a parameter $\rho_0>0$, which we will choose shortly, we define the curve 
    $$\Gamma^{(k)}:[-\rho_0,\rho_0]\rightarrow\mathcal H^2,\ \Gamma^{(k)}(t):=\frac1{h_k}\left(\gamma^{(k)}(\hat\tau_k+h_k t)-\left(0, \gamma^{(k)}_2(\hat\tau_k)\right)\right).$$
    Using $L_k|=\dot\gamma^{(k)}|\rightarrow L[\kappa]=L_*$ we deduce that $L_k$ is bounded and estimate 
    \begin{equation}\label{gamma1klowerbound}
    \gamma_1^{(k)}(\hat\tau_k+h_k t)=h_k+h_k\int_0^t\dot\gamma_1^{(k)}(\hat\tau_k+h_k s)ds\geq h_k(1-L_k |t|)\geq h_k(1-C\rho_0).
    \end{equation}
 Choosing $\rho_0$ small enough, we obtain $\gamma^ {(k)}_1(\hat\tau_k+h_k t)\geq \frac12h_k$. We fix any $\rho\in (0,\rho_0)$. By Lemma \ref{blowuplemma} we deduce that $\Gamma^ {(k)}$ is bounded in $C^ m([-\rho,\rho])$ for all $m\in\N$. Consequently there exists a subsequence, again denoted by $\Gamma^ {(k)}$, that converges in $C^2([-\rho,\rho])$ to some limit $\Gamma^*$.  Lemma \ref{blowuplemma} gives $H[\Gamma^*]=0$. Since $\Gamma^ {(k)}(0)=(1,0)$, $\dot\Gamma^ {(k)}(0)=(0,L_k)$ and $\ddot\Gamma^ {(k)}_1(0)\leq 0$ we obtain $\Gamma^*(0)=(1,0)$, $\dot\Gamma^*(0)=(0,L_*)$ and $\ddot\Gamma^*_1(0)\leq 0$. Inserting into $H[\Gamma^*]=0$, we produce the contradiction 
 $$0=k_1[\Gamma^*](0)+k_2[\Gamma^*](0)
 =-\frac{\ddot\Gamma_1^*(0)\dot\Gamma^*_2(0)}{L_*^3}+\frac{\dot\Gamma_2^*(0)}{\Gamma_1^*(0)L_*}=-\frac{\ddot\Gamma_1^*(0)}{L_*^2}+1\geq 1.$$
Consequently there exists $\rho_0>0$ such that $\dot\gamma_1^{(k)}(t)\neq 0$ for all $t\in(\tau_k-\rho_0,\tau_k+\rho_0)$. We argue for example that $\dot\gamma_1^{(k)}> 0$ when $t\in(\tau_k,\tau_k+\rho)$. Indeed, if this were not true, then $\dot\gamma_1(t)<0$ would have to be true. But that would imply that for small $t>0$
$$\epsilon_k=\inf_{[\frac18,\frac78]}\gamma_1^{(k)}\leq \gamma_1^{(k)}(\tau_k+t)=\gamma_1^{(k)}(\tau_k)+\int_0^t \dot\gamma_1^{(k)}(\tau_k+s)ds<\epsilon_k.$$
\end{proof}

\begin{lemma}\label{gammakModificationLemma}
    There exists $k_1\in\N$ such that for $k\geq k_1$, the curves $\hat\gamma^{(k)}(s):=(\gamma_1^{(k)}(1-s),\gamma_2^{(k)}(1-s)-\gamma_2^{(k)}(1))$ satisfy $\hat\gamma^{(k)}\in \mathcal F_\sigma^{0+}$ and $\mathcal W[\hat\gamma^{(k)}]=\beta(\sigma_k)$.
\end{lemma}
\begin{proof}
    It is clear that $\hat\gamma^{(k)}\in\mathcal P$, $\mathcal W[\gamma^{(k)}]=\mathcal W[\hat\gamma^{(k)}]$ and that $\mathcal I[\gamma^{(k)}]=\mathcal I[\hat\gamma^{(k)}]$. Additionally $\hat\gamma^{(k)}(0)=0$.
    Let $\epsilon>0$ satisfy 
    $$0<\epsilon<\int_0^1 \kappa_2(t)dt.$$
    Since $\gamma^{(k)}\rightarrow \kappa$ in $C^0([0,1])$ we can choose $k_1(\epsilon)$ such that for all $k\geq k_1$ 
    $$
\int_0^1\hat\gamma_2^{(k)}(t) dt
=
\int_0^1\gamma_2^{(k)}(t) dt-\gamma_2^{(k)}(1)
\geq 
\int_0^1\kappa_2(s) ds-\kappa_2(1)-\epsilon
= 
\int_0^1\kappa_2(s) ds-\epsilon
>0.
$$
\end{proof}
By Lemma \ref{catenoidLemma02} we know that for $k$ large enough, $\gamma_1^{(k)}$ attains its minimum on $[\frac18,\frac78]$ only once, namely at $t=\tau_k$. At $t=\tau_k$ we then have $\dot\gamma^{(k)}(\tau_k)=(0,\pm_k L_k)$. Using Lemma \ref{gammakModificationLemma}, we can now modify the sequence $\gamma^{(k)}$ and ensure that for all $k\geq k_1$
\begin{equation}\label{assumeptionLk}
    \dot\gamma^{(k)}(\tau_k)=(0,- L_k).
\end{equation}

\begin{lemma}[Proof of Theorem \ref{CatenoidNeckTheorem}, part (4)]\ \\
    Let $\gamma^{(k)}$ now additionally satisfy Equation \eqref{assumeptionLk}. For all $R>0$ there exists $k_0(R)\in\N$ such that for all $k\geq k_0(R)$ the function 
    $$\Gamma^ {(k)}:[-R,R]\rightarrow\mathcal H^2,\ \Gamma^ {(k)}(t):=\frac1{\epsilon_k}\left(\gamma^ {(k)}(\tau_k+\epsilon_k t)-(0,\gamma_2^ {(k)}(\tau_k))\right)$$
    is well-defined. Additionally, putting 
    $$\Gamma^*(t):=\left(\sqrt{1+L_*^2t^2},\ -\operatorname{asinh}(L_*t)\right)\hspace{.5cm}\textrm{where }L_*=\sqrt{\frac\pi2},$$
    we have $\Gamma^{(k)}\rightarrow \Gamma^*$ in $C^m([-R, R])$ for all $m\in\N$. 
\end{lemma}
\begin{proof}
    For $R>0$ we choose $k_0(R)$ so large that $|\tau_k+\epsilon_k t|\in(\frac18,\frac78)$ for all $|t|\leq R+1$ and $k\geq k_0(R)$. By definition of $\epsilon_k$ we get $\Gamma^ {(k)}_1(t)\geq 1$ for all $t\in[-R,R]$. Therefore we can use Lemma \ref{blowuplemma} and deduce that $\Gamma^ {(k)}$ is bounded in $C^ m([-R,R])$ for all $m\in\N$. We show that any subsequence contains another subsequence that converges to the claimed limit. Let $\Gamma^ {(k)}$ be some subsequence, again denoted by $\Gamma^ {(k)}$. By the uniform $C^ m$ bound we can assume that $\Gamma^ {(k)}\rightarrow\Gamma^*$ in $C^2([-R,R])$ after potentially passing to yet another subsequence and by Lemma \ref{blowuplemma}, the limit satisfies $H[\Gamma^*]=0$. By assumption we have $\dot\gamma_2 ^{(k)}(\tau_k)=-L_k$ and since $\tau_k$ is a minimum of $\gamma_1^{(k)}$ we have $\dot\gamma_1^{(k)}(\tau_k)=0$ and $\ddot\gamma_1^{(k)}(\tau_k)\geq 0$. By Theorem \ref{asymptoticsthm}, we have $L_k=L[\gamma^{(k)}]\rightarrow L[\kappa]=\sqrt{\frac\pi2}$. Exploiting the $C^2$-convergence, we get
    $$\Gamma^ *(0)=(1,0),\hspace{.5cm} \dot\Gamma^*(0)=(0, -L_*)\hspace{.5cm}\textrm{and}\hspace{.5cm}\ddot\Gamma^*_1(0)\geq 0.$$
    Since $|\dot\Gamma^*(t)|=\lim_{k\rightarrow\infty}|\dot\Gamma^{(k)}(t)|=\lim_{k\rightarrow\infty}L_k=L_*$, the limit $\Gamma^*$ is parameterized proportional to arc length. Using Equation \eqref{ArcLength_H_Identity}, we get
    \begin{equation}\label{LimitH0Equations}
    \ddot\Gamma_1^*=\frac{(\dot\Gamma_2^*)^2}{\Gamma_1^*}
    \hspace{.5cm}\textrm{and}\hspace{.5cm}
    \ddot\Gamma_2^*=-\frac{\dot\Gamma_1^*\dot\Gamma_2^*}{\Gamma_1^*}.
    \end{equation}
    Using the first identity and $|\dot\Gamma^*|^2=L_*^2$, we compute 
   $$
        \frac12\frac{d^2}{dt^2}\Gamma_1^*(t)^2=\ddot\Gamma_1^*\Gamma_1^*+(\dot\Gamma_1^*)^2=(\dot\Gamma_2^*)^2+(\dot\Gamma_1^*)^2=L_*^2.
   $$
So, using  $\Gamma_1^*(0)=1$, $\dot\Gamma_1^*(0)=0$  and $\Gamma_1^*(t)\geq 0$, we deduce  $\Gamma_1^*(t)=\sqrt{1+L_*^2 t^2}$. Now, using the second identity in Equation \eqref{LimitH0Equations}, we compute
    $$\frac d{dt}(\dot\Gamma_2^*\Gamma_1^*)=\ddot\Gamma_2^*\Gamma_1^*+\dot\Gamma_1^*\dot\Gamma_2^*=0
     \hspace{.5cm}\textrm{and hence}\hspace{.5cm}
     \dot\Gamma_2^*(t)\Gamma_1^*(t)=\dot\Gamma_2^*(0)\Gamma_1^*(0).$$
     We have $\dot\Gamma_2^*(0)=-L_*$ and $\Gamma^*(0)=(1,0)$, which implies 
     $$\dot\Gamma_2^*(t)=\frac{-L_*}{\Gamma_1^*(t)}=\frac{-L_*}{\sqrt{1+L_*^2 t^2}}\hspace{.5cm}\textrm{and hence}\hspace{.5cm}
     \Gamma_2^*(t)=-\operatorname{asinh}(L_*t).$$
\end{proof}
\newpage
\appendix
\section{Properties of Profile Curves}\label{Appendix1}
\subsection{Proofs for Subsection \ref{functionalspacesection}}\label{proofpropertiesofP}
We begin by proving Lemma \ref{WillmoreLowerBound}
\begin{proof}
We put $L:=L[\gamma]$ and $d\mu_\gamma:=2\pi|\dot\gamma|\gamma_1 dt$ so that $k_1^2+k_2^2\in L^1(\mu_\gamma)$ by definition of the class $\mathcal P$. Hence $|k_1k_2|\in L^1(\mu_\gamma)$ and we may estimate 
\begin{align}
\mathcal W[\gamma]=&\frac{2\pi}4\int_0^1 (k_1+k_2)^2\gamma_1|\dot\gamma(t)|dt\nonumber\\ &\geq 2\pi\int_0 ^1 k_1 k_2 |\dot\gamma|\gamma_1 dt
=\lim_{\epsilon\rightarrow0^+} 2\pi\int_\epsilon^{1-\epsilon} k_1 k_2 |\dot\gamma|\gamma_1 dt.\label{WillLB01}
\end{align}
By definition of the class $\mathcal P$ we have $\gamma\in W^{2,2}((\epsilon,1-\epsilon))$. So we may use Equation \eqref{ArcLength_K_Identity} and deduce
$$2\pi\int_\epsilon^{1-\epsilon} k_1k_2|\dot\gamma|\gamma_1dt=-\frac{2\pi}L\int_\epsilon^ {1-\epsilon}\ddot\gamma_1(t)dt=\frac{2\pi}L(\dot\gamma_1(\epsilon)-\dot\gamma_1(1-\epsilon)).$$
We insert into Equation \eqref{WillLB01} and wish to let $\epsilon\rightarrow 0^ +$. To do so, note that $\dot\gamma$ is continuous by the definition of the class $\mathcal P$. So, using Equation \eqref{velocityatends}, we find
\begin{equation}\label{gaussbonnet}
\mathcal W[\Sigma_\gamma]\geq 2\pi\int_0^1 k_1k_2|\dot\gamma|\gamma_1 dt=\lim_{\epsilon\rightarrow0^ +}2\pi\int_\epsilon^{1-\epsilon} k_1k_2|\dot\gamma|\gamma_1dt
= \frac{2\pi}L(L-(-L))=4\pi.
\end{equation}
Equality only occurs when $(k_1+k_2)^2=4k_1k_2$ almost everywhere, which is true only if $k_1=k_2$ almost everywhere. This gives 
$$\ddot\gamma_2\dot\gamma_1-\ddot\gamma_1\dot\gamma_2=L^3k_1=L^3 k_2=\frac{\dot\gamma_2}{\gamma_1}L^2.$$
Multiplying this equation by $\dot\gamma_2$ and using $\langle\dot\gamma,\ddot\gamma\rangle=0$ yields 
\begin{equation}\label{globalgamma1ode}
\ddot\gamma_1\gamma_1=\dot\gamma_1^2-L^2.
\end{equation}
As $\gamma_1>0$ on $(0,1)$ and $\gamma\in C^1([0,1])$, we deduce $\gamma_1\in C^2((0,1))\cap C^1([0,1])$. Since $\dot\gamma_1\in C^0([0,1])$, $\dot\gamma_1(0)=L$ and $\dot\gamma_1(1)=-L$ there exists a non-empty, open interval $I\subset(0,1)$ such that $\dot\gamma_1(t)\not\in\set{0,\pm L}$ for all $t\in I$. On $I$, the following computation is valid:
$$-\frac12\frac d{dt}\ln(L^2-\dot\gamma_1^2)=\frac{\dot\gamma_1\ddot\gamma_1}{L^2-\dot\gamma_1^2}\overset{\eqref{globalgamma1ode}}=-\frac{\dot\gamma_1}{\gamma_1}=-\frac d{dt}\ln(\gamma_1(t))$$
From this, we deduce that there exists $K>0$ such that $L^2-\dot\gamma_1^2=K\gamma_1^2(t)$ for all $t\in I$. Differentiating gives $\dot\gamma_1(\ddot\gamma_1+K\gamma_1)=0$. As $\dot\gamma_1(t)\neq 0$ we deduce that for some $a,b\in\R$ we get
\begin{equation}\label{gamma1sol}
\gamma_1(t)=a\cos(\sqrt K t)+b\sin(\sqrt K t)\hspace{.5cm}\textrm{for all $t\in I$}.
\end{equation}
Using $\gamma_1>0$ for all $t\in(0,1)$ and Equation \eqref{globalgamma1ode}, the Picard-Lindeöff theorem implies that Equation \eqref{gamma1sol} is valid for all $t\in(0,1)$. Using $\dot\gamma_2^2=L^2-\dot\gamma_1^2$, it is then readily seen that $\gamma$ traces out a semicircle and hence $\Sigma_\gamma$ is a sphere. 
\end{proof}

Note that in Equation \eqref{gaussbonnet}, we have verified Lemma \ref{gausbonnetlemma} as we have shown that
$$
\int_{\Sp^2}k_1k_2d\mu_{\Sigma_\gamma}
=2\pi\int_0^1 k_1k_2 |\dot\gamma|\gamma_1 dt
=4\pi.
$$

Next, we prove Lemma \ref{hopftypetheorem}
\begin{proof}
First we note $H_0\neq 0$ as $4\pi\leq \mathcal W[\gamma]=\frac14 A[\gamma]H_0^2$ by Lemma \ref{WillmoreLowerBound}. Let $L:=L[\gamma]$. Using the assumption $H=H_0$ and Equation \eqref{ArcLength_H_Identity}, we get
\begin{equation}\label{hopf0}
    H_0\dot\gamma_2=-\frac{\ddot\gamma_1}L+\frac{\dot\gamma_2^2}{L\gamma_1}
    \hspace{.5cm}\textrm{and}\hspace{.5cm}
     H_0\dot\gamma_1=\frac{\ddot\gamma_2}L+\frac{\dot\gamma_1\dot\gamma_2}{L\gamma_1}.
\end{equation}
Using $\gamma\in C^1([0,1])$, $\gamma_1>0$ on $(0,1)$ and Equation \eqref{hopf0}, we get $\gamma\in C^\infty((0,1))$. Using the second identity in Equation \eqref{hopf0}, we compute 
\begin{equation}\label{sphereidentity1}
\frac d{dt}(\dot\gamma_2\gamma_1)=\ddot\gamma_2\gamma_1+\dot\gamma_2\dot\gamma_1=LH_0\gamma_1\dot\gamma_1
\hspace{.5cm}\textrm{and hence}\hspace{.5cm}
\dot\gamma_2\gamma_1=\frac12LH_0\gamma_1^2.
\end{equation}
There is no integration constant since $\gamma_1(0)=0$. Using $\gamma_1(t)>0$ for $t\in (0,1)$, we may deduce $\dot\gamma_2=\frac12 LH_0\gamma_1$. Inserting into the first identity in Equation \eqref{hopf0} and putting $\omega:=\frac{LH_0}2$, we get 
$$\ddot\gamma_1+\omega^2\gamma_1=0\hspace{.5cm}\textrm{and hence}\hspace{.5cm}\gamma_1(t)=a\cos(\omega t)+b\sin(\omega t).$$
Combining $\gamma_1(0)=\gamma_1(1)=0$ and $\gamma_1(t)>0$ for $t\in(0,1)$ we deduce $a=0$, $b>0$ and $\omega=\pi$.  Using $\dot\gamma_2=\omega \gamma_1$ we get $\gamma_2(t)=\gamma_2(0)+b(1-\cos(\pi t))$. Finally, since $|\dot\gamma|=L$, we conclude $b=\frac L\pi$.
\end{proof}
\subsection{First Variation of the Willmore Energy}
Let $I=[s_0,s_1]$ be an interval in $\R$ and $\gamma=(\gamma_1,\gamma_2)\in C^\infty(I,\mathcal H^2)$. We assume that the axisymmetric surface defined by 
$$f:I\times[0,2\pi)\rightarrow\R^3,\ f(s, \theta):=\begin{bmatrix}
\gamma_1(s)\cos\theta\\
\gamma_1(s)\sin\theta\\
\gamma_2(s)
\end{bmatrix}$$
is smooth. Given a variation $\phi:(-\epsilon_0,\epsilon_0)\rightarrow\mathcal H^2$ of $\gamma$, we can define a variation of $f$ by putting 
$$\Phi:(-\epsilon_0,\epsilon_0)\times I\times[0,2\pi)\rightarrow\R^3,\ \Phi(\epsilon,s,\theta):=\begin{bmatrix}
\phi_1(\epsilon,s)\cos\theta\\
\phi_1(\epsilon,s)\sin\theta\\
\phi_2(\epsilon,s)
\end{bmatrix}.$$
Considering Equations \eqref{metricprofilecurve} for the metric, \eqref{principalcurvatures} for the principal curvatures of $f$ and the formula for $W[f]$ from Equation \eqref{WillmoreOperator}, we observe that $W$ only depends on $s$ and not on $\theta$. This justifies the notation $W[\gamma]$ for the very same function. Using the formulas for the normals of $\Sigma_\gamma$ and $\gamma$ from Equation \eqref{normalsprofilecurve} we further have
$\langle\Phi'(0),n\rangle=\langle\phi'(0), \nu\rangle$. 
Using the surface element from Equation \eqref{metricprofilecurve}, we can rewrite the surface integral in Equation \eqref{classicalformula} as
$$\int_{I\times\Sp^1}W[f]\langle\Phi'(0),n\rangle d\mu_f=2\pi\int_{s_0}^{s_1} W[\gamma]\langle\phi'(0), \nu\rangle \gamma_1|\dot\gamma|ds.$$
Let us now consider the boundary term. We assume that $\gamma$ is not closed -- that is $\gamma(s_0)\neq\gamma(s_1)$. Additionally, we assume that $\gamma_1(s_0)\neq 0\neq \gamma_1(s_1)$. In this case $\partial\Sigma_\gamma\neq\emptyset$ and
$$\partial\Sigma_\gamma=\set{f(s,\theta)\ |\ s\in\set{s_0,s_1},\ \theta\in[0,2\pi)}.$$
We focus on the boundary component at $s=s_0$. The interior conormal is given by
\begin{equation}\label{boundarycomputation1}
\eta=\frac1{|\dot\gamma(s)|}\frac\partial{\partial s}\bigg|_{s=s_0}
\hspace{.5cm}\textrm{and}\hspace{.5cm}
Df\eta=\frac1{|\dot\gamma|}\frac{\partial f}{\partial s}\bigg|_{s=s_0}.
\end{equation}
The tangential part of $\Phi'(0)$ is $\Phi'(0)-\langle\Phi'(0), n\rangle n$ and 
\begin{equation}\label{boundarycomputation2}
    \langle \Phi'(0)-\langle\Phi'(0), n\rangle n, Df\eta\rangle=\frac1{|\dot\gamma(s_0)|}\langle \Phi'(0), \frac{\partial f}{\partial s}\bigg|_{s=s_0}\rangle=\frac1{|\dot\gamma(s_0)|}\langle \phi'(0), \dot\gamma(s_0)\rangle.
\end{equation}
Combining Equations \eqref{boundarycomputation1} and \eqref{boundarycomputation2} and using the formula fom Equation \eqref{WillmoreBoundaryterm}, we get
\begin{align*}
\int_{s_0\times\Sp^1}\omega(\eta)dS_f=&\frac12\left[
\frac{\varphi(s_0) H'(s_0)-\varphi'(s_0) H(s_0)}{|\dot\gamma(s_0)|}
-
\frac12\frac{H(s_0)^2}{|\dot\gamma(s_0)|}\langle\dot\gamma(s_0), \phi'(0)\rangle 
\right]\int_0^{2\pi}\gamma_1(s_0)d\theta\\
=&\pi\left[
\frac{\varphi(s_0) H'(s_0)-\varphi'(s_0) H(s_0)
-\frac12H(s_0)^2\langle \dot\gamma(s_0),\phi'(0)\rangle}
{|\dot\gamma(s_0)|}
\right]\gamma_1(s_0).
\end{align*}

Computing the boundary term at $s=s_1$ is achieved by a similar computation. In total, we get the formula
\begin{equation}\label{variationformulastandard}
\begin{aligned}
    \delta\mathcal W[\gamma]\phi'(0)=&2\pi\int_{s_0}^{s_1} W[\gamma]\langle\phi,\nu\rangle \gamma_1|\dot\gamma|ds\\
    &-\pi\left[
\frac{\varphi(s) H'(s)-\varphi'(s) H(s)
-\frac12H^2(s)\langle \dot\gamma(s),\phi'(0)\rangle}
{|\dot\gamma(s)|}
\right]\gamma_1(s)\bigg|_{s=s_0}^{s=s_1}.
\end{aligned}
\end{equation}

\subsection{Existence of Variations}\label{variationexistenceappendix}
\begin{lemma}\label{variationexistence}
Let $\sigma\in(0,1)$, $\gamma\in\mathcal F_\sigma$ and $\eta\in C^\infty_0((0,1))$ such that $\delta\mathcal I[\gamma]\eta=0$. Then there exists $\Phi:(-\epsilon_0,\epsilon_0)\rightarrow\mathcal F'_\sigma$ such that $\Phi'(0)=\eta$.
\end{lemma}
\begin{proof}
Since $\sigma\in(0,1)$, Lemma \ref{psi0existencelemma} provides $\psi_0\in C^\infty_0((0,1))$ such that $\delta\mathcal I[\gamma]\psi_0=1$. For small $\epsilon,\rho\in\R$ we have $\gamma+\epsilon\eta+\rho\psi_0\in\mathcal P'$ as $\operatorname{supp}(\eta),\ \operatorname{supp}(\psi_0)\subset(0,1)$. Now, we define the operator 
$$Q(\epsilon,\rho):=\mathcal I[\gamma+\epsilon\eta+\rho\psi_0].$$
Clearly $Q(0,0)=\sigma$ and 
$$\frac d{d\rho}\bigg|_{\rho=0}Q(0,\rho)=\delta\mathcal I[\gamma]\psi_0=1\neq 0.$$
So, by the implicit function theorem, there exist $\epsilon_0, \rho_0>0$ and a smooth map $\kappa:(-\epsilon_0,\epsilon_0)\rightarrow (-\rho_0,\rho_0)$ such that for $\epsilon\in(-\epsilon_0,\epsilon_0)$ and $\rho\in(-\rho_0,\rho_0)$ we have
$$Q(\epsilon,\rho)=\sigma
\hspace{.5cm}\Leftrightarrow\hspace{.5cm}
\rho=\kappa(\epsilon).
$$
Moreover $\kappa'(0)=0$ as
$$0=\frac d{d\epsilon}\bigg|_{\epsilon=0}Q(\epsilon,\kappa(\epsilon))=\delta\mathcal I[\gamma]\eta+\kappa'(0)\delta\mathcal I[\gamma]\psi_0=\kappa'(0).$$
Therefore $\Phi(\epsilon):=\gamma+\epsilon\eta+\kappa(\epsilon)\psi_0$ provides a variation as claimed. 
\end{proof}

\noindent
\textit{Proof of Lemma \ref{eulerlagangeataxis}.}\ \\
Using $\varphi'(0)=0$ and $\operatorname{supp}(\psi_0)\subset(0,1)$, we have $\gamma+\epsilon\chi\cdot(\varphi\circ\gamma_1)e_2+\rho\psi_0\in\mathcal P'$. This allows the definition of the operator 
$$Q(\epsilon,\rho):=\mathcal I[\gamma+\epsilon\chi\cdot(\varphi\circ\gamma_1)e_2+\rho\psi_0].$$
Note $Q(0,0)=\sigma$ and 
$$\frac d{d\rho}\bigg|_{\rho=0} Q(0,\rho)=\delta\mathcal I[\gamma]\psi_0=1\neq 0.$$
So, by the implicit function theorem, there exist $\epsilon_0, \rho_0>0$ and a diffeomorphism $\kappa:(-\epsilon_0,\epsilon_0)\rightarrow (-\rho_0,\rho_0)$ such that for $\epsilon\in(-\epsilon_0,\epsilon_0)$ and $\rho\in(-\rho_0,\rho_0)$ we have
$$Q(\epsilon,\rho)=\sigma
\hspace{.5cm}\Leftrightarrow\hspace{.5cm}
\rho=\kappa(\epsilon).
$$
Moreover we have $\kappa'(0)=-\delta\mathcal I[\gamma]\left(\chi\cdot(\varphi\circ\gamma_1) e_2\right)$ as 
$$0=\frac d{d\epsilon}\bigg|_{\epsilon=0}Q(\epsilon,\kappa(\epsilon))=\delta\mathcal I[\gamma]\left(\chi\cdot(\varphi\circ\gamma_1) e_2\right)+\kappa'(0)\delta\mathcal I[\gamma]\psi_0.$$
Therefore $\Phi(\epsilon):=\gamma+\epsilon\chi\cdot(\varphi\circ\gamma_1) e_2+\kappa(\epsilon)\psi_0$ provides a variation as claimed. 
\qed

\subsection{Proof of Theorem \ref{schygulla}}\label{schygullapp}

Schygulla's proof in \cite{schygulla} is based on studying the scaled catenoid. For $a>0$ he considers 
$$g_a:\R\times[0,2\pi)\rightarrow\R^3,\ g_a(s,\theta):=\left(
a\cosh\left(\frac sa\right)\cos(\theta),a\cosh\left(\frac sa\right)\sin(\theta),s
\right).$$
Next, he introduces the inversion $I(x):=e_3+\frac{x-e_3}{|x-e_3|^2}$, defines $f_a:=I\circ g_a$ and puts 
$$\Sigma_a:=f_a(\R\times[0,2\pi))\cup\set{e_3}.$$

$\Sigma_a$ is smooth away from $e_3$. Additionally, there exist $R>0$, functions $u^\pm:B_R(0)\rightarrow\R$ and a neighbourhood of $e_3$ in which $\Sigma_a=\operatorname{graph}(u^+)\cup\operatorname{graph}(u^-)$. Direct computation shows $u^\pm\in C^{1,\alpha}(B_R(0))\cap W^{2,p}(B_R(0))$ for all $\alpha\in(0,1)$ and $p\in[1,\infty)$. Another direct computation shows $\mathcal W[\Sigma_a]=8\pi$.\\

We now sketch the proof of Theorem \ref{schygulla} by following Schygulla's arguments and commenting on the necessary alterations. For details, we refer to Lemma 1 in Schygulla's article \cite{schygulla}.
\begin{proof}
Let $\varphi_\pm\in C^\infty_0(B_R)$ be axially symmetric and smooth functions --  that is $\varphi(x)=\varphi(Tx)$ for all $T\in \operatorname{SO}(2)$ -- that satisfy $\varphi_{\pm}(0)=\pm 1$. For small $\epsilon>0$, we define variations
$$\Phi_\pm^\epsilon:B_R(0)\rightarrow\R^3,\ \Phi^\epsilon_\pm(x):=(x, u_\pm(x)+\epsilon\varphi_\pm(x)).$$
Following Schygulla's computations, it follows that there exists $c_0>0$ such that
$$\frac d{d\epsilon}\bigg|_{\epsilon=0}\mathcal W[\Phi_\pm^\epsilon]=\mp c_0\varphi_\pm(0).$$
We modify $\Sigma_a$ and construct a surface $\Sigma_a^\epsilon$ by replacing the graphs of $u^\pm$ with the graphs of $u_\pm+\epsilon\varphi_\pm$ with small $\epsilon>0$. Consequently $\mathcal W[\Sigma_a^\epsilon]<8\pi-\epsilon c_0$ for small $\epsilon>0$. Additionally, $\Sigma_a^\epsilon$ can be parameterized over $\Sp^2$ and is axially symmetric. Finally, we note $\mathcal I[\Sigma_a]\rightarrow 0$, as $a\rightarrow0^+$. Given $\sigma\in(0,1)$ we can therefore choose $a$ and $\epsilon$ so small such that simultaneously
\begin{align*}
    \mathcal W[\Sigma_a^\epsilon]&\leq \mathcal W[\Sigma_a]-\epsilon c_0=8\pi-\epsilon c_0,\\
    \mathcal I[\Sigma_a^\epsilon]&<\sigma.
\end{align*}
Approximating $\Sigma_a^\epsilon$ with smooth surfaces, we deduce that there exists a smooth, axially symmetric embedding $F_0:\Sp^2\rightarrow\R^3$ with $\mathcal W[F_0]<8\pi$ and $\mathcal I[F_0]<\sigma$. Theorem 5.2 in \cite{removability} allows us to deduce that the Willmore flow $F(t)$ with initial datum $F(0)=F_0$ exists for all times and converges to a sphere. As $F_0$ is axially symmetric and the solution to the Willmore flow is unique  (Proposition 1.1 in \cite{kuwert2002gradient}), it follows that $F(t)$ is axially symmetric for all $t\geq 0$. Since $\mathcal I[F(0)]<\sigma$ and $\mathcal I[F(t)]\rightarrow 1$ as $t\rightarrow\infty$, there exists a time $t_0>0$ such that $\mathcal I[F(t_0)]=\sigma$. Finally, by definitions of $\beta(\sigma)$ and the dissipation of Willmore energy along the Willmore flow, we get
$$\beta(\sigma)\leq \mathcal W[F(t_0)]\leq \mathcal W[F(0)]<8\pi.$$
\end{proof}

\subsection{The Monotonicity Formula}
Let $\Sigma$ be a smooth $2$-dimensional manifold without boundary, $f\in C^1(\Sigma,\R^3)$ be an immersion and $\mu_\Sigma$ denote the surface measure on $\Sigma$ that is induced by $f$. For $p\in\R^3$, the \emph{muliplicity} of $p$ is defined as 
\begin{equation}\label{multiplicityDef}
\theta^2_{f}(p):=\lim_{r\rightarrow0^+}\frac{\mu_{\Sigma}(B_r(p))}{\pi r^2}.
\end{equation}
If $\Sigma$ is compact, this expression is well-defined and
\begin{equation}\label{multiplicitiyisNumberPreImages}
\theta^2_f(p)=|\set{x\in\Sigma\ |\ f(x)=p}|.
\end{equation}

The following result is due Simon \cite{SimonMonotonicity}. Its applicability in the present setting has essentially been established by Kuwert and Schätzle \cite{removability}, Appendix A.
\begin{lemma}[Monotonicity Formula]\label{monotonicityformula}
Let $\gamma\in \mathcal P \cap  C^\infty((0,1),\mathcal H^2)$. The surface $f_\gamma:[0,1]\times[0,2\pi)\rightarrow\R^3$ as defined in Equation \eqref{fgammadef} is $C^1$-immersed. Putting $\theta^2(p):=\theta^2_{f_\gamma}(p)$, we have
$$\frac1\pi\int_{\Sp^2}\left|
\frac14\vec H+\frac{(f_\gamma(x)-p)^\perp}{|f_\gamma(x)-p|^2}
\right| d\mu_{\Sigma_\gamma}(x)
\leq 
\frac1{4\pi}\mathcal W[f_\gamma]-\theta^2(p).$$
Here $\vec H$ is the almost everywhere defined mean curvature vector of $f_\gamma$ and $(f_\gamma(x)-p)^\perp$ is the component of $f_\gamma(x)-p$ that is orthogonal to $Df_\gamma(x)( T_x \Sp^2)$. 
\end{lemma}

\begin{proof}
We put $f:=f_\gamma$ and $\mu:=\mu_{\Sigma_\gamma}$. We use the monotonicity formula proved in Appendix A of \cite{removability}. We may apply their results as the mean curvature $H$ is in $L^2(\mu)$. Indeed, by the definition of the class $\mathcal P$ (see Definition \ref{classPdef})
$$
\int_{\Sp^2} H^2 d\mu
\leq 2\int_{\Sp^2} k_1^2+k_2^2 d\mu
<\infty.
$$
Using Equation (A3) from \cite{removability}, we get 
\begin{align}
    &\sigma^{-2}\mu(B_\sigma (x))
    +\int_{f^{-1}(B_\rho(p)\backslash B_\sigma(p))}\left|\frac14 \vec H+\frac{(f(x)-p)^\perp}{|f(x)-p|^2}\right|^2 d\mu\nonumber\\
    =&
    \rho^{-2}\mu(B_\rho (x))
    +\frac1{16}\int_{f^{-1}(B_\rho(p)\backslash B_\sigma(p))} \vec H^2 d\mu
    +
    R_{p,\rho}-R_{p,\sigma}\label{monotonicityformulaeq1}
\end{align}
where for $\alpha=\sigma,\rho$
$$R_{p,\alpha}:=\frac1{2\alpha^2}\int_{f^{-1}(B_\alpha(p))}\langle(x-p),\vec H\rangle d\mu.$$
The Equation after (A6) in \cite{removability} implies $R_{p,\sigma}\rightarrow 0$ as $\sigma\rightarrow 0^+$. We note that $\mu(B_\rho(p))\leq A[f]<\infty$ by definition of the class $\mathcal P$. Hence $\rho^{-2}\mu_{\Sigma}(B_\rho(p))\rightarrow 0$ for $\rho\rightarrow\infty$. So Equation (A13) in \cite{removability} implies $\lim_{\rho\rightarrow \infty}R_{p,\rho}=0$. 
Letting $\rho\rightarrow\infty$ and $\sigma\rightarrow 0^+$ in Equation \eqref{monotonicityformulaeq1} then gives
$$\liminf_{\sigma\rightarrow 0^+}\left[
\sigma^{-2}\mu_{\Sigma}(B_\sigma (x))
    +\int_{\Sp^2\backslash f^{-1}(B_\sigma(p))}\left|\frac14 \vec H+\frac{(f(x)-p)^\perp}{|f(x)-p|^2}\right|^2 d\mu
\right]
\leq 
\frac14\mathcal W[f].
$$
The lemma follows by using Fatou's lemma and Equation \eqref{multiplicityDef}.
\end{proof}

\begin{lemma}\label{SphereOrInvertedCat}
Let $\gamma\in\mathcal P\cap C^\infty((0,1))$ and denote by $\vec H$ the almost everywhere defined mean curvature vector of $f_\gamma$. Suppose that there exists $p\in\R^3$ such that
\begin{equation}\label{ssumptioncritical}
\frac14\vec H(s)=-\frac{(f_\gamma(s,\theta)-p)^\perp}{|f_\gamma(s,\theta)- p|^2}.
\end{equation}
Then $\Sigma_\gamma$ is either a sphere or the inversion of a scaled and vertically translated catenoid. In particular, only the second case is possible if $\mathcal W[\Sigma_\gamma]=8\pi$.
\end{lemma}

\begin{proof}
By vertical translation of $\gamma$ and rotational invariance, we can assume $p=Re_1$ for some $R\geq 0$. Let $L:=L[\gamma]$. Using Equation \eqref{normalsprofilecurve}, we compute 
\begin{align*} 
(f_\gamma(s,\theta)-Re_1)\cdot n(s,\theta)&=\begin{bmatrix}
\gamma_1\cos\theta-R\\
\gamma_1\sin\theta\\
\gamma_2
\end{bmatrix}
\cdot
\frac1L\begin{bmatrix}
-\dot\gamma_2\cos\theta\\
-\dot\gamma_2\sin\theta\\
\dot\gamma_1
\end{bmatrix}
=\frac{-\gamma_1\dot\gamma_2+\dot\gamma_1\gamma_2+R\dot\gamma_2\cos\theta}L.
\end{align*}
A quick computation shows $|f_\gamma(s,\theta)-Re_1|^2=\gamma_1^2+\gamma_2^2+R^2-2R\gamma_1\cos\theta$. Inserting into Equation \eqref{ssumptioncritical}, we get
\begin{equation}\label{newequationLemmaA13}
\frac14 H(s)\left(\gamma_1^2+\gamma_2^2+R^2-2R\gamma_1\cos\theta\right)
=
-\frac{-\gamma_1\dot\gamma_2+\dot\gamma_1\gamma_2+R\dot\gamma_2\cos\theta}L.
\end{equation}
\paragraph{The case $R>0$}\ \\
Comparing the coefficients of $\cos\theta$ in Equation \eqref{newequationLemmaA13}and using $R\neq 0$ we derive\\ $\frac{\dot\gamma_2}L=\frac12H(s)\gamma_1(s)$. By using Equation \eqref{ArcLength_H_Identity}, we get
$$
\frac{2\dot\gamma_2^2}L
=H\gamma_1\dot\gamma_2
=-\frac{\gamma_1\ddot\gamma_1}L+\frac{\dot\gamma_2^2}{L}.
$$
We use $\dot\gamma_2^2=L^2-\dot\gamma_1^2$ and get $-\ddot\gamma_1\gamma_1=\dot\gamma_2^2=L^2-\dot\gamma_1^2$. Next, we put $x(t):=\gamma_1(\frac tL)$. $x$ satisfies the following problem
\begin{equation}\label{xproblem}
    \left\{\begin{aligned}
    \ddot x x&=\dot x^2-1,\\
    x(0)&=x(L)=0,\\
    \dot x(0)&=1,\\
    x(t)&>0\hspace{.5cm}\textrm{for }t\in(0,L).
\end{aligned}\right.
\end{equation}
Since $|\dot\gamma(t)|\leq L$ we have $|\dot x|\leq 1$. Suppose that there exists $t_0\in(0,L)$ such that $\dot x(t_0)^2=1$. Then by the Picard-Lindelöff theorem $x(t)=x(t_0)\pm (t-t_0)$. This, however, contradicts  $x(0)=x(L)=0$. So $x^2(t)<1$ for all $t\in(0,L)$ and hence the following computation is admissible for $t\in(0,L)$:
$$-\frac12\frac d{dt}\ln(1-\dot x^2)=\frac{\dot x\ddot x}{1-\dot x^2}=-\frac{\dot x}x=-\frac d{dt}\ln(x)$$
Integrating and subsequent exponentiation proves that there exists a positive constant $K$ such that 
$1-\dot x^2=K x^2$. Differentiating this identity gives $\dot x(\ddot x+Kx)=0$ Since $\dot x(0)=1$ we have $\dot x\neq 0$ for small times $t\in[0, t_0]$ and get 
$$
x(t)=a\sin(\sqrt K t)+b\cos(\sqrt K t)\hspace{.5cm}\textrm{for }t\in[0,t_0].$$
As $x(t)>0$, Equation \eqref{xproblem} and the Picard-Lindelöff theorem imply that this formula is valid not only on $[0, t_0]$ but for all $t\in[0,L]$. Since $x(0)=x(L)=0$ we get $b=0$ and $\sqrt K=L^{-1}n\pi$ for some $n\in\N$. Moreover, as $x>0$ on $(0,L)$ and $\dot x(0)=1$, we get 
$$x(t)=\frac L\pi\sin\left(\frac \pi L t\right)\hspace{.5cm}\textrm{for all $t\in[0, L]$}.$$
Let $y(t):=\gamma_2(\frac tL)$. As $\dot y^2=1-\dot x^2=\sin^2(\frac\pi L t)$, we have $\dot y(t)=\pm\sin(\frac\pi L t)$ and hence
$$y(t)=y(0)\mp\frac L\pi\cos\left(\frac \pi L t\right).$$
The formulas for $x(t)$ and $y(t)$ imply that $\gamma$ parameterises a circle of radius $\frac L\pi$ with center at $(0, y(0))$. In particular $\mathcal W[\gamma]=4\pi$.\\

\paragraph{The Case $R=0$}\ \\
Let $I:\R^3\backslash\set0\rightarrow\R^3\backslash\set 0$, $I(x):=\frac{x}{|x|^2}$. Then $I^*\delta=\frac1{|x|^4}\delta$ and  
$$H[I\circ f_\gamma]=H[f_\gamma, I^*\delta]=|f_\gamma|^2H[f_\gamma]+4\langle f_\gamma, n\rangle=0.$$
This shows that the surface defined by the profile curve $\Gamma(s):=\frac{\gamma(s)}{|\gamma(s)|^2}$ is a minimal surface. Let $\alpha$ denote the maximal solution to the problem
$$
\left\{
\begin{aligned}
    \dot\alpha(t)&=\frac{|\gamma(\alpha(t))|^2}{L},\\
    \alpha(0)&=\frac12.
\end{aligned}
\right.
$$
We define $\Gamma^*(t):=\Gamma(\alpha(t))$ and compute 
$$|\dot\Gamma^*(s)|=\left|\left(\frac{\dot\gamma}{|\gamma|^2}-2\frac{\langle\gamma,\dot\gamma\rangle}{|\gamma|^4}\gamma\right)\dot\alpha\right|=1.$$
Consequently, $\Gamma^*$ is parameterized by arc length and satisfies $H[\Gamma^*]=0$. Following the arguments after Equation \eqref{LimitH0Equations} we get $\Gamma_1^*(t)=\sqrt{a+bt+t^2}$ and $\dot\Gamma_2^*(t)\Gamma_1^*(t)=c$ for constants $a>0$ and $b,c\in\R$. By potentially redefining $\Gamma^*(t)\rightarrow\Gamma^*(-t)$ we may further assume $b\geq 0$. Since $|\dot\Gamma^*|=1$ we have 
$$1=(\dot\Gamma_1^*)^2+(\dot\Gamma_2^*)^2=\frac{c^2+(t+\frac b2)^2}{(\Gamma_1^*)^2}
\hspace{.5cm}\textrm{and hence}\hspace{.5cm}
t^2+bt+a=(t+\frac b2)^2+c^2.$$
If $c=0$ we deduce $\Gamma^*(t)=(t+\frac b2,\Gamma_2^*(0))$ for $t\geq -\frac b2$. Hence $\Sigma_{\Gamma^*}$ is part of a a plane $P$ orthogonal to the $x_3$-axis and consequently $\Sigma_\gamma$ is part of a sphere or, if $0\in P$, part of $P$. The second case is impossible. Indeed, if $\gamma(s)=(\gamma_1(s), c_0)$ for some $c_0\in\R$, we get $L=|\dot\gamma|=|\dot\gamma_1|$. Since $\gamma_1(0)=0$ and $\gamma_1(t)\geq 0$ for all $t\in[0,1]$ we get $\gamma_1(t)=Lt$, which contradicts $\gamma_1(1)=0$. So, if $c=0$, $\Sigma_\gamma$ is part of a sphere and by the definition of the class $\mathcal P$, it then follows that $\Sigma_\gamma$ is a sphere and hence $\mathcal W[\Sigma_\gamma]=4\pi$. If $c\neq 0$, we have 
\begin{align*}
\Gamma_1^*(t)&=\sqrt{\left(t+\frac b2\right)^2+c^2},\\
    \Gamma_2^*(t)&=\Gamma_2^*(0)+\int_0^t \frac {cdt}{\sqrt{(t+\frac b2)^2+c^2}}=\Gamma_2^*(0)-|c|\operatorname{asinh}(\frac b{2c})+|c|\operatorname{asinh}(\frac tc+\frac b{2c}).
\end{align*}
This is the profile curve of a scaled and vertically translated catenoid. Hence, $\Sigma_\gamma$ is the inversion of a scaled and vertically translated catenoid and $\mathcal W[\Sigma_\gamma]=8\pi$.
\end{proof}

\section{Technicalities regarding the Direct Method}

\subsection{Lower Semi Continuity}
The proof of Equations \eqref{Willmorelowersemicont} and \eqref{curvsqlowersemicont} is the consequence of two lemmas. First, we establish the following:
\begin{lemma}\label{lowersemiconLemma}
Let $\rho_n,\rho\in C^0([0,1],\R^+_0)$ and $f_n,f:[0,1]\rightarrow\R$ be measurable such that $\rho_n\rightarrow\rho$ in $C^0([0,1])$ and $f_n\rho_n\rightarrow f\rho$ weakly in $L^1((0,1))$. Further assume $f^2\rho\in L^1((0,1))$ and $\|f_n^2\rho_n\|_{L^1((0,1))}\leq C<\infty$ for all $n\in\N$. Then
$$\int_0^1 f^2(t)\rho(t)dt\leq\liminf_{n\rightarrow\infty}\int_0^1 f_n^2(t)\rho_n(t)dt.$$
\end{lemma}
All integrals are well-defined since $f^2\rho\geq 0$ and $f_n^2\rho_n\geq 0$. 
\begin{proof}
For $R>0$ we put $\phi_R:=f1_{|f|\leq R}\in L^\infty((0,1))$. Using $f_n\rho_n\rightarrow f\rho$ weakly in $L^1((0,1))$ and $\rho_n\rightarrow\rho$ in $C^0([0,1])$, we estimate
\begin{align*}
    \int_0^1 f(t)\phi_R(t)\rho(t)dt=&\lim_{n\rightarrow\infty}\int_0^1 f_n(t)\phi_R(t)\rho_n(t)dt\\
                                \leq &\liminf_{n\rightarrow\infty}\left[\left(\int_0^1 f_n(t)^2 \rho_n(t)dt\right)^{\frac12}\left(\int_0^1 \phi_R(t)^2\rho_n(t)dt\right)^{\frac12}\right]\\
                                \leq & \liminf_{n\rightarrow\infty}\left(\int_0^1 f_n(t)^2 \rho_n(t)dt\right)^{\frac12}\left(\int_0^1 \phi_R(t)^2\rho(t)dt\right)^{\frac12}.
\end{align*}    
Note that $\phi_R\rightarrow f$ pointwise and $|\phi_R|\leq |f|$. As $|f|^2\rho\in L^1((0,1))$ we can let $R\rightarrow\infty$ and get the claimed estimate. 
\end{proof}

 \begin{lemma}\label{lemmaprinciplacuvaturesconv}
Let $\gamma^{(n)},\gamma^*\in\mathcal P^w$ such that $\gamma^{(n)}\rightarrow \gamma^*$ in $\mathcal P^w$ and 
$$\sup_n\int_0^1 \left((k_1^ {(n)})^2+(k_2^ {(n)})^2\right) 2\pi \gamma_1^ {(n)}|\dot\gamma^ {(n)}|dt<\infty.$$
Then $k_i^{(n)}\gamma_1^{(n)}\rightarrow k_i^*\gamma_1^*$ weakly in $L^1((0,1))$ for $i=1,2$. 
\end{lemma}

\begin{proof}
Let $\phi\in L^ \infty((0,1))$, $L_n:=L[\gamma^{(n)}]$ and $L_*:=L[\gamma^*]$. Then $L_n\rightarrow L_*$ and, as $\gamma^*\in\mathcal P^ w$, we have $L_*>0$. We put $R=\begin{bmatrix}
0 & -1\\
1 & 0
\end{bmatrix}$ and estimate 
$$\left|\int_0^1 \langle\ddot\gamma^{(n)},R\dot\gamma^{(n)}-R\dot\gamma^{*}\rangle\gamma_1^{(n)}\phi dt\right|
\leq 
\|\phi\|_{L^\infty}\left[\int_0^1|\ddot\gamma^{(n)}|^2\gamma_1^{(n)}dt\right]^{\frac12}\left(\sup_n \gamma_1^{(n)}\right)^{\frac12}\|\dot\gamma^{(n)}-\dot\gamma^*\|_{L^2}.$$
$\||\ddot\gamma^ {(n)}|^2\gamma_1^ {(n)}\|_{L^1((0,1))}$ is bounded by combining the assumptions of the lemma with Lemma \ref{gamma1ddotgammainL1}. Using $\gamma^{(n)}\rightarrow\gamma^*$ in $C^0([0,1])$ and $W^ {1,2}((0,1))$, we deduce that 
\begin{equation}\label{erroerestiamtekconvergenceeq1}
\lim_{n\rightarrow\infty}\left|\int_0^1 \langle\ddot\gamma^{(n)},R\dot\gamma^{(n)}-R\dot\gamma^{*}\rangle\gamma_1^{(n)}\phi dt\right|=0.
\end{equation}
Note that $R\dot\gamma^*\phi\in L^\infty$. So, by weak convergence in $L^1$, we get 
$$\lim_{n\rightarrow\infty}\int_0^1\langle \ddot\gamma^{(n)} ,R\dot\gamma^{*}\rangle\gamma_1^{(n)}\phi dt
=
\lim_{n\rightarrow\infty}\int_0^1 \langle\ddot\gamma^{*}, R\dot\gamma^{*}\rangle\gamma_1^{*}\phi dt
.$$
In view of Equation \eqref{erroerestiamtekconvergenceeq1} as well as Equation \eqref{principalcurvatures}, we get $k_1^{(n)}\gamma_1^{(n)}\rightarrow k_1^*\gamma_1^*$ weakly in $L^1((0,1))$.\ \\

Next, we take care of $k_2^{(n)}$. As $\dot\gamma^{(n)}\rightarrow\dot\gamma^*$ in $L^2((0,1))$ we get 
 $$\int_0^1 k_2^{(n)}\gamma_1^{(n)}(t)\phi(t)dt=\frac{2\pi}{ L_n}\int_0^1 \dot\gamma_2^{(n)}(t) \phi(t)dt\rightarrow \int_0^1 k_2^*\gamma_1^*(t)\phi(t)dt.$$
\end{proof}

\noindent
\textit{Proof of Estimates \eqref{Willmorelowersemicont} and \eqref{curvsqlowersemicont}.}\ \\
   The proof for both estimates is essentially the same. To prove Estimate \eqref{Willmorelowersemicont}, we pass to a subsequence that realizes the $\liminf$ and assume without loss of generality that the $\liminf$ is finite. Consequently 
   \begin{equation}\label{liminfsuprealization}
   \sup_{n\in\N}\int_0^1\left(k_1^ {(n)}+k_2^ {(n)}\right)^22\pi|\dot\gamma^ {(n)}|\gamma_1^ {(n)}dt<\infty.\\
  \end{equation}
   We wish to apply Lemma \ref{lowersemiconLemma} with the functions $f_n=k_1^ {(n)}+k_2^ {(n)}$, $f=k_1^*+k_2^*$, $\rho_n=\gamma_1^ {(n)}$ and $\rho=\gamma_1^*$. To justify this choice we first use Lemma  \ref{lemmaprinciplacuvaturesconv} to get $(k_1^{(n)}+k_2^{(n)})\gamma_1^{(n)}\rightarrow (k_1^*+k_2^*)\gamma_1^*$ weakly in $L^1((0,1))$. Additionally we have Estimate \eqref{liminfsuprealization} and $(k_1^*+k_2^*)^2\gamma_1^*\in L^1((0,1))$ by definition of the class $\mathcal P^ w$.  Estimate \eqref{Willmorelowersemicont} now follows from Lemma \ref{lowersemiconLemma}. To get  Estimate \eqref{curvsqlowersemicont}, the only change is to take $f_n=k_i^ {(n)}$ and $f=k_i^*$ for $i=1,2$. \qed

\subsection{Compactness}\label{compactnessappendix}
The following lemma is due to Choski and Veneroni \cite{choskiveneroni} (see Lemma 5). We present a modified proof. 

\begin{lemma}\label{oscbound}
Let $\gamma\in\mathcal P^w$, $L:=L[\gamma]$ and $(a,b)\subset [0,1]$ such that $\gamma_1(t)>0$ for all $t\in(a,b)$. Then 
\begin{align*} 
|\dot\gamma_1(b^-)-\dot\gamma_1(a^+)|&\leq \frac L{2\pi}\int_a^b |k_1k_2|2\pi L\gamma_1(t)dt,\\
|\dot\gamma_2(b^-)-\dot\gamma_2(a^+)|&\leq\sqrt{\frac{L}{2\pi}}  \left(\int_a^b k_1^2 2\pi L\gamma_1 dt\right)^{\frac12}
    \left(\int_a^b\frac{\dot\gamma_1^2}{\gamma_1}dt\right)^{\frac12}.
\end{align*}
In  particular, if $\dot\gamma$ is continuous at $a$ and $b$, the one sided limits can be replaced by $\dot\gamma(a)$ and $\dot\gamma(b)$ respectively. 
\end{lemma}
Before presenting the proof, we have the following remark: In the proof, we use Theorem \ref{weakandstrongPconnection} to deduce that $\dot\gamma$ has one-sided limits at $a$ and $b$. The analog of Theorem \ref{weakandstrongPconnection} in \cite{choskiveneroni} is Lemma 6 and the proof uses Lemma 5 (the analog of our Lemma \ref{oscbound}). However, the continuity statement we use here is Corollary 2 in \cite{choskiveneroni} and proven without using the analog of our Lemma \ref{oscbound}. 
\begin{proof}
Let $\epsilon>0$ be small. Then $\gamma_1\geq c_0(\epsilon)>0$ on $[a+\epsilon, b-\epsilon]$ so that $\gamma\in W^{2,2}((a+\epsilon, b-\epsilon))$ by Definition \ref{PWeakDefinition}. 
Using Equation \eqref{ArcLength_K_Identity}, we compute 
$$
    \frac1{2\pi}\int_{a+\epsilon}^{b-\epsilon}|k_1k_2|2\pi L\gamma_1 dt
    \geq 
    \frac1L\left|\int_{a+\epsilon}^{b-\epsilon}\ddot\gamma_1(t)dt\right|
    =
    \frac{|\dot\gamma_1(b-\epsilon)-\dot\gamma_1(a+\epsilon)|}{L}.
$$
Now we let $\epsilon\rightarrow 0^+$. To do so, we use $|k_1k_2|2\pi L\gamma_1\in L^1((0,1))$ by definition of the class $\mathcal P^w$ and Theorem \ref{weakandstrongPconnection} to get the first estimate. For the second estimate, we argue similarly and write 
\begin{align*}
 \frac{|\dot\gamma_2(b-\epsilon)-\dot\gamma_2(a+\epsilon)|}{L}&=
 \frac1L\left|\int_{a+\epsilon}^{b-\epsilon}\ddot\gamma_2(t)dt\right|\\
 &=  \left|\int_{a+\epsilon}^{b-\epsilon}k_1 \dot\gamma_1 dt\right|\\
 &=  \frac1{\sqrt{2\pi L}}\left|\int_{a+\epsilon}^{b-\epsilon}k_1 \frac{\dot\gamma_1}{\sqrt{\gamma_1}} \sqrt{2\pi L\gamma_1} dt\right|\\
 &\leq  \frac1{\sqrt{2\pi L}}\left[
    \left(\int_a^b k_1^2 2\pi L\gamma_1 dt\right)^{\frac12}
    \left(\int_a^b\frac{\dot\gamma_1^2}{\gamma_1}dt\right)^{\frac12}
 \right].
\end{align*}
Note that while valid, we may have estimated by $+\infty$ in the last step. Letting $\epsilon\rightarrow 0^+$, the second estimate claimed in the lemma follows.
\end{proof}

For the proof of Theorem \ref{compactnesstheorem}, we require the space of bounded variations.
\begin{definition}
Let $\alpha,\beta\in\R$. For a measurbale function $f:[\alpha,\beta]\rightarrow\R$ we put
$$V_\alpha^\beta(f):=\sup\left\{\sum_{i=0}^{N-1}|f(t_{i+1})-f(t_i)|\ \bigg|\ N\in\N\textrm{ and }\alpha=t_0<t_1<...<t_{N-1}<t_N=\beta\right\}.$$
$f$ is said to have \emph{bounded variation} if $V_\alpha^\beta(f)<\infty$. If additionally $f\in L^1((\alpha,\beta))$, we define
$\|f\|_{\operatorname{BV}((\alpha,\beta))}:=\|f\|_{L^1((\alpha,\beta))}+V_\alpha^\beta(f)$. 
The space of all $f$ with $\|f\|_{\operatorname{BV}((\alpha,\beta))}<\infty$ is denoted by $\operatorname{BV}((\alpha,\beta))$. 
\end{definition}
We stress that to define $V_\alpha^\beta(f)$, we consider a function $f:[\alpha,\beta]\rightarrow\R$ and not the equivalence class of all functions $\tilde f$ such that $\tilde f=f$ almost everywhere. We require the following theorem. For a proof, we refer to Theorem 2.35 and Proposition 2.38 in \cite{leoni2017first}.
\begin{theorem}[Helly's Selection Theorem]\label{helly}
Let $(f_n)$ be a sequence of measurable functions $f_n:[\alpha,\beta]\rightarrow\R$ and $c\in[\alpha,\beta]$ such that $|f_n(c)|+V_\alpha^\beta(f_n)\leq C<\infty$ for all $n\in\N$. Then, there exists a subsequence $f_{n_a}$ and a measurable function $f:[\alpha,\beta]\rightarrow\R$ such that $f_{n_a}\rightarrow f$ pointwise and
$$V_\alpha^\beta(f)\leq\liminf_{a\rightarrow\infty}V_\alpha^\beta(f_{n_a}).$$
\end{theorem}
We establish the following corollary:
\begin{korollar}\label{BVCompactInL1}
Let $(f_n)\subset \operatorname{BV}((\alpha,\beta))$ be a sequence such that $\|f_n\|_{\operatorname{BV}((\alpha,\beta))}\leq C<\infty$. Then, there exist $f\in \operatorname{BV}((\alpha,\beta))$ and a subsequence  $(f_{n_a})$ such that $f_{n_a}\rightarrow f$ pointwise and in $L^1((\alpha,\beta))$. 
\end{korollar}
\begin{proof}
We first claim that $f_n(\alpha)$ is bounded. Once this is shown, we can use Helly's selection theorem to get a subsequence $f_{n_a}$ such that $f_{n_a}\rightarrow f$ pointwise. For $x\in(\alpha,\beta)$ we estimate 
\begin{align*}
|f_{n_a}(x)|&\leq |f_{n_a}(\alpha)|+|f_{n_a}(\alpha)-f_{n_a}(x)|\\
&\leq  |f_{n_a}(\alpha)|+|f_{n_a}(\alpha)-f_{n_a}(x)|+|f_{n_a}(x)-f_{n_a}(\beta)|\\
&\leq \sup_{a\geq 0}\left(|f_{n_a}(\alpha)|+V_\alpha^\beta(f_{n_a})\right).
\end{align*}
Therefore $f_{n_a}$ is bounded in $L^\infty((\alpha,\beta))$ and hence $f\in L^\infty((\alpha,\beta))$. Using Lebesgue's theorem, we obtain 
$$0=\lim_{a\rightarrow\infty}\int_\alpha^\beta |f(t)-f_{n_a}(t)|dt.$$
To prove that $f_n(\alpha)$ is bounded, let $x_n\in[\alpha,\beta]$ such that $|f_n(x_n)|<\inf_{[\alpha,\beta]}|f_n|+1$. We have the estimate 
$$C\geq \int_\alpha^\beta |f_n(x)|dx\geq \int_\alpha^\beta |f_n(x_n)|-1 dx= |f_n(x_n)|-(\beta-\alpha).$$
So, $f_n(x_n)$ is bounded and we can estimate
$$|f_n(\alpha)|\leq |f_n(x_n)|+|f_n(x_n)-f_n(\alpha)|+|f_n(\beta)-f_n(x_n)|\leq |f_n(x_n)|+V_\alpha^\beta(f_n)\leq C.$$
\end{proof}

Let $U\subset\R$ be open. Then there exist at most countably many and uniquely determined disjoint intervals $I_j=(\alpha_j,\beta_j)$ such that $U=\bigcup_j I_j$. This is e.g. shown in \cite{pugh2003real}, Chapter 2, Section 1, Theorem 9. Using this fact, we extend the notion of the space of bounded variations to general open sets of $\R$.

\begin{definition}
Let $U\subset\R$ be open, $m\in\N\cup\set\infty$ and $U=\bigcup_{j=1}^m(\alpha_j,\beta_j)$ be the unique decomposition of $U$ into disjoint, open intervals. For a measurable function $f:U\rightarrow\R$, we put 
$$V_U(f):=\sum_{j=1}^m V_{\alpha_j}^{\beta_j}(f).$$
$f$ is said to have \emph{bounded variation} if $V_U(f)<\infty$. If additionally $f\in L^1(U)$, we define $\|f\|_{\operatorname{BV}(U)}:=\|f\|_{L^1(U)}+V_U(f)$. The space of all $f$ with $\|f\|_{\operatorname{BV}(U)}<\infty$ is denoted by $\operatorname{BV}(U).$
\end{definition}

By a usual diagonal argument, we get the following corollary from Corollary \ref{BVCompactInL1}
\begin{korollar}\label{BVCompactInL1Mk2}
    Let $U\subset\R$ be open and $(f_n)\subset \operatorname{BV}(U)$ such that $\|f_n\|_{\operatorname{BV}(U)}\leq C<\infty$. Then there exists $f\in \operatorname{BV}(U)$ and a subsequence $(f_{n_a})$ such that $f_{n_a}\rightarrow f$ in $L^1(U)$. 
\end{korollar}

Finally, we establish Theorem \ref{compactnesstheorem}. The following proof is due to Choski and Veneroni \cite{choskiveneroni} (see Proposition 2) with only minor modifications.
\renewcommand{\proofname}{\textit{Proof of Theorem \ref{compactnesstheorem}.}}
\begin{proof}
We put $L_n:=L[\gamma^{(n)}]$. By the assumptions of Theorem \ref{compactnesstheorem}, $\gamma^{(n)}$ is bounded in $C^0([0,1])$ and combining the assumptions of Theorem \ref{compactnesstheorem} with Theorem \ref{choskiveneroniresult} shows that $\dot\gamma^ {(n)}$ is bounded in $L^\infty((0,1))$. In particular, we can assume without loss of generality that $L_n\rightarrow L_*\in[0,\infty)$.
\ \\
\ \\
\noindent
\textbf{Step 1 - Weak Limit}\ \\
Since $\gamma^{(n)}$ is bounded in $C^0$ and $|\dot\gamma^{(n)}|$ is bounded in $L^\infty((0,1))$ and therefore also in $L^2((0,1))$, we can use the Arzel\`a-Ascoli theorem to deduce that, after passing to a subsequence, $\gamma^{(n)}\rightarrow\gamma^*$ in $C^0([0,1])$ and $\dot\gamma^{(n)}\rightarrow\dot\gamma^* $ weakly in $L^2((0,1))$.
Next, we prove by contradiction that $L_*>0$. Suppose $L_*=0$, then, as $\gamma^ {(n)}\rightarrow\gamma^*$ in $C^ 0([0,1])$, we get
$$\theta\leq\liminf_{n\rightarrow\infty}A[\gamma^{(n)}]=2\pi\liminf_{n\rightarrow\infty}\int_0^1\gamma_1^{(n)}(t)L_n dt=0.$$
\ \\
\noindent
\textbf{Step 2 - Convergence in $W^ {1,2}$}\ \\
By Lemma \ref{oscbound}, the sequence $\dot\gamma^{(n)}_1$ is bounded in $\operatorname{BV}((0,1))$ since 
$$
V_0^1(\dot\gamma_1^{(n)})\leq \frac{L_n}{2\pi}\int_0^1 |k_1^{(n)}k_2^{(n)}|2\pi L_n\gamma_1^{(n)}(t) dt\leq C.
$$
Using Corollary \ref{BVCompactInL1}, we deduce that there exists $\sigma\in L^1((0,1))$ such that  $\dot\gamma^{(n)}_1\rightarrow\sigma$ strongly in $L^1((0,1))$ along a subsequence and since $\dot\gamma^{(n)}\rightarrow\dot\gamma^*$ weakly in $L^2((0,1))$, we get $\sigma=\gamma_1^*$. As $\dot\gamma^{(n)}_1$ is bounded in $L^\infty((0,1))$, we get, $\dot\gamma_1^ {(n)}\rightarrow\dot\gamma_1^*$ in $L^ p((0,1))$ for all $p\in[1,\infty)$. Indeed
$$\int_0^1 |\dot\gamma_1^ {(n)}-\dot\gamma_1^ {(m)}|^ p dt\leq \|\dot\gamma_1^ {(n)}-\dot\gamma_1^ {(m)}\|_{L^ \infty}^ {p-1}\|\dot\gamma_1^ {(n)}-\dot\gamma_1^ {(m)}\|_{L^1((0,1))}\leq C\|\dot\gamma_1^ {(n)}-\dot\gamma_1^ {(m)}\|_{L^1((0,1))}$$
shows that $\dot\gamma_1^ {(n)}$ is Cauchy in $L^ p((0,1))$ and hence converges to some limit. By uniqueness of the limit in $L^1((0,1))$ we get $\dot\gamma_1^ {(n)}\rightarrow\dot\gamma_1^*$ in $L^ p((0,1))$ for all $p\in[1,\infty)$. \\

Next, we claim $\gamma_1^*(t)>0$ almost everywhere. To see this, let $E:=\set{t\in[0,1]\ |\ \gamma_1^*(t)=0}$. Clearly $E$ is compact and, by uniform convergence, $\gamma_1^{(n)}\leq\epsilon$ on $E$ for all $n\geq n_0(\epsilon)$. Using the second assumption of Theorem \ref{compactnesstheorem}, we get
$$C\geq \int_0^1 \left(k_2^{(n)}\right)^22\pi L_n\gamma_1^{(n)}dt=\frac{2\pi}{L_n}\int_0^1 \frac{\left(\dot\gamma_2^{(n)}\right)^2}{\gamma_1^{(n)}}dt\geq \frac{2\pi}{L_n \epsilon}\int_E\left(\dot\gamma_2^{(n)}\right)^2dt.$$
This shows $\dot\gamma^{(n)}_2\rightarrow 0$ in $L^2(E)$. Combining this with $\dot\gamma_1^ {(n)}\rightarrow\dot\gamma_1^*$ in $L^2((0,1))$, we get
$$|E|L_*^2=\lim_{n\rightarrow\infty}|E|L_n^2=\lim_{n\rightarrow\infty}\int_E|\dot\gamma^ {(n)}|^2dt=\int_E(\dot\gamma_1^*)^2dt.$$ 
Now assume $|E|>0$. Then $|\dot\gamma_1^*(t)|=L_*$ for almost all $t\in E$. Using Rademacher's theorem, we deduce that $\gamma$ is differentiable at almost all $t\in E$. Since the classical and the weak derivative agree almost everywhere, we deduce that there exists $t_0\in E$ such that $\gamma$ is differentiable in $t_0$ and $|\dot\gamma^*_1(t_0)|=L_*>0$. But by definition of $E$ we also have $\gamma_1^*(t_0)=0$, which implies that there exists $t\in[0,1]$ close to $t_0$ such that $\gamma_1^*(t)<0$, which is a contradiction.\\

Next, we prove that along a subsequence $\dot\gamma_2^{(n)}\rightarrow\dot\gamma_2$ in $L^1((0,1))$. For $k\in\N$ we put 
$A_k:=\set{t\in[0,1]\ |\ \gamma_1(t)>2k^{-1}}$. For $n\geq n_0(k)$ we have $\gamma_1^{(n)}\geq\frac1k$ on $A_k$. Using  Lemma \ref{oscbound}, we deduce that $\dot\gamma_2^{(n)}$ is bounded in $\operatorname{BV}(A_k)$ for all $k$ with 
$$V_{A_k}(\dot\gamma_2^{(n)})\leq \sqrt{\frac{L_n}{2\pi}}\left(\int_0^1 (k_1^{(n)})^2 2\pi L_n\gamma_1^{(n)}dt\right)^{\frac12}k^{\frac12}\left(\int_0^1 (\dot\gamma_1^{(n)})^2dt\right)^{\frac12}\leq C\sqrt k.$$
After passing to a subsequence we deduce $\dot\gamma_2^{(n)}\rightarrow\dot\gamma_2^*$ in $L^1(A_k)$ for all $k\in\N$. We estimate
$$\int_0^1|\dot\gamma_2^{(n)}-\dot\gamma_2^*|dt\leq 
\int_{A_k}|\dot\gamma_2^{(n)}-\dot\gamma_2^*|dt+(L_n+L_*)|[0,1]\backslash A_k|.$$
$|[0,1]\backslash A_k|\rightarrow 0$ as $k\rightarrow\infty$ since $|E|=0$. Thus $\dot\gamma_2^{(n)}\rightarrow\dot\gamma_2^*$ in $L^1((0,1))$ and, using the $L^\infty$-bound, also in $L^p((0,1))$ for all $1\leq p<\infty$.\\

In particular we have $\dot\gamma^{(n)}\rightarrow\dot\gamma^*$ in $L^1((0,1))$, which implies that for all $0\leq a<b\leq 1$
$$L_*(b-a)=\lim_{n\rightarrow\infty}L_n(b-a)=\lim_{n\rightarrow\infty}\int_a^b |\dot\gamma^{(n)}|dt=\int_a^b|\dot\gamma^*|dt.$$
Taking $a=0$ and $b=1$ shows $L[\gamma^*]=L_*$ and taking $b=a+\delta$ and letting $\delta\rightarrow 0$ shows $|\dot\gamma^*(a)|=L_*$ for almost every $a\in[0,1]$. 
\ \\
\ \\
\noindent
\textbf{Step 3 - Limit in $\mathcal P^ w$}\ \\
Let $U\subset(0,1)$ be open such that $\gamma_1^*>0$ on $\bar U$. We prove that $\gamma^*\in W^ {2,2}(U)$. First, as $\gamma^ {(n)}\rightarrow\gamma^*$ in $C^ 0([0,1])$ we know that there exists $n_0(U)$ such that for all $n\geq n_0(U)$ we have $\inf_U \gamma_1^{(n)}\geq c_0>0$. By Lemma \ref{gamma1ddotgammainL1}, $(\ddot\gamma^ {(n)})\subset L^2(U)$ is a bounded sequence and hence there exists $\xi\in L^2(U)$ such that $\ddot\gamma^ {(n)}\rightarrow\xi$ weakly in $L^2(U)$ along a subsequence. For $\varphi\in C^ \infty_0(U)$ we get 
$$\int_U\xi\varphi dt=\lim_{n\rightarrow\infty}\int_U\ddot\gamma^ {(n)}\varphi dt=-\lim_{n\rightarrow\infty}\int_U\dot\gamma^ {(n)}\dot\varphi dt=-\int_U\dot\gamma^*\dot\varphi dt.$$
We deduce $\gamma^*\in W^ {2,2}_{\operatorname{loc}}([\gamma_1^*>0])$. It remains to prove that the curvature of $\gamma^*$ is bounded in $L^2$. As $\gamma^ {(n)}\rightarrow\gamma^*$ in $W^ {1,2}((0,1))$ we have $\dot\gamma^ {(n_a)}\rightarrow\dot\gamma^*$ pointwise almost everywhere along a suitable subsequence and clearly $\gamma^ {(n)}\rightarrow\gamma^*$ pointwise. Hence, by Fatou's lemma and the assumptions of Theorem \ref{compactnesstheorem}
\begin{align*}
    \int_0^1 (k_2^*)^22\pi L\gamma_1^* dt&=
    \int_0^1 \frac{(\dot\gamma^*_2)^2}{\gamma_1^* L_*}2\pi dt\\
    & \leq \liminf_{a\rightarrow\infty}\int_0^1\frac{(\dot\gamma^{(n_a)}_2)^2}{\gamma_1^{(n_a)} L_{n_a}}2\pi dt\\
    &= \liminf_{a\rightarrow\infty}\int_0^1 \left(k_2^ {(n_a)}\right)^2 2\pi L_{n_a}\gamma_1^ {(n_a)}dt\\
    &\leq C.
\end{align*}
Next, we consider $k_1^*$. Let $U\subset (0,1)$ such that $\bar U\cap E=\emptyset$. Since $(\ddot\gamma^{(n)})\subset L^2(U)$ is bounded and $\gamma_1^{(n)}\rightarrow\gamma_1^*$ in $C^0([0,1])$, we may estimate  
\begin{equation}\label{substiutionestimate}
\left|\int_U |\ddot\gamma^ {(n)}|^2(\gamma_1^*-\gamma_1^ {(n)})dt\right|\leq \|\gamma^ {(n)}_1-\gamma_1^*\|_{C^ 0}\sup_n \|\ddot\gamma^ {(n)}\|_{L^2(U)}^2\rightarrow 0.
\end{equation}
Along a subsequence $n_a$ we have $\ddot\gamma^ {(n_a)}\rightarrow\ddot\gamma^*$ weakly in $L^2(U)$. Let $\varphi\in C^ \infty_0(U)$. Using Estimate \eqref{substiutionestimate}, we can compute 
\begin{align*} 
\int_U|\ddot\gamma^*|^2|\gamma_1^*|dt
&=\lim_{a\rightarrow\infty}\int_U|\ddot\gamma^ {(n_a)}|^2\gamma_1^*dt
=\lim_{a\rightarrow\infty}\int_U|\ddot\gamma^ {(n_a)}|^2\gamma_1^ {(n_a)}dt
\leq \sup_n\int_0^1 |\ddot\gamma^ {(n)}|^2\gamma_1^ {(n)}dt.
\end{align*}
Consequently $|\ddot\gamma^*|^2\gamma_1^*\in L^1(U)$ with norm bounded independent of $U$. Using $|E|=0$ and Equation \eqref{ArcLength_K_Identity}, we get 
$$\int_0^1(k_1^*)^22\pi\gamma_1^*|\dot\gamma^*|dt
=\sup\left\{
\frac{2\pi}{L^3_*} \int_U |\ddot\gamma^*|^2\gamma_1^* dt\ \bigg|\ U\subset[0,1],\ \inf_U\gamma_1^*>0\right\}<\infty.$$
In total we have shown $\gamma^*\in\mathcal P^ w$.\ \\
\ \\
\noindent
\textbf{Step 4 - Convergence in $\mathcal P^ w$}\ \\
Considering Definition \ref{convergencedefinition} and that $\gamma^{(n)}\rightarrow\gamma^*$ in $C^0([0,1])$ and $W^{1,2}((0,1))$ has already been established, we are left with proving that $\ddot\gamma^{(n)}\gamma_1^{(n)}\rightarrow\ddot\gamma^*\gamma_1^*$ weakly in $L^1((0,1))$. Let $\xi\in \mathcal P^ w$. By Theorem \ref{weakandstrongPconnection}, there exist $0=t_0<t_1<...<t_N=1$ such that $\xi_1(t)=0$ precisely when $t=t_i$ for some $i$. For small $\epsilon>0$, we have $\xi\in W^{2,2}(t_i+\epsilon, t_{i+1}-\epsilon)$. So, for $\phi\in C^ \infty([0,1])$, we have
$$\int_{t_i+\epsilon}^ {t_{i+1}-\epsilon}\ddot\xi\xi_1\phi dt
=
\dot\xi\xi_1\phi\bigg|_{t_i+\epsilon}^ {t_{i+1}-\epsilon}
-
\int_{t_i+\epsilon}^ {t_{i+1}-\epsilon}\dot\xi(\dot\xi_1\phi+\xi_1\dot\phi)dt.$$
We sum over $i$ and wish to take $\epsilon\rightarrow 0^+$. As $\xi\in\mathcal P^ w$ we have $\ddot\xi\xi_1\in L^1((0,1))$ by Lemma \ref{gamma1ddotgammainL1}. Using $\phi,\dot\xi\in L^ \infty((0,1))$ and $\xi_1(t_i)=0$, the boundary term vanishes for $\epsilon\rightarrow 0^ +$. So, we deduce 
$$\int_{0}^ {1}\ddot\xi\xi_1\phi
=
-
\int_{0}^ 1\dot\xi_1(\dot\xi_1\phi+\xi_1\dot\phi)dt.$$
As $\gamma^ {(n)}\rightarrow\gamma^*$ in $C^ 0([0,1])$ and in $W^ {1,2}((0,1))$ we can compute 
\begin{align}
    &\lim_{n\rightarrow\infty}\int_{0}^ {1}\ddot\gamma^ {(n)}\gamma_1^ {(n)}\phi dt
=-\lim_{n\rightarrow\infty }
\int_{0}^ 1\dot\gamma^ {(n)}(\dot\gamma_1^ {(n)}\phi+\gamma_1^ {(n)}\dot\phi)dt\nonumber\\
=&-
\int_{0}^ 1\dot\gamma^ {*}(\dot\gamma_1^ {*}\phi+\gamma_1^ {*}\dot\phi)dt=\int_{0}^ {1}\ddot\gamma^ {*}\gamma_1^ {*}\phi dt.\label{identityforsmoothtests}
\end{align}
Finally, we prove that Equation \eqref{identityforsmoothtests} is also valid for $\phi\in L^2((0,1))$.  Let $\phi\in L^2((0,1))$ and $\epsilon>0$. We choose $\phi_\epsilon\in C^ \infty([0,1])$ such that $\|\phi_\epsilon-\phi\|_{L^2((0,1))}\leq \epsilon$ and estimate 
\begin{equation}\label{phiapprox1}
\left|\int_0^1 \ddot\gamma^*\gamma_1^*(\phi-\phi_\epsilon) dt\right|
\leq \left(\int_0^1|\ddot\gamma^*|^2\gamma_1^*dt\right)^ {\frac12}\left(\sup\gamma_1^*\right)^ {\frac12}\|\phi-\phi_\epsilon\|_{L^2((0,1))}\leq C\epsilon.
\end{equation}
The last step is justified by Lemma \ref{gamma1ddotgammainL1}. Similarly, using Lemma \ref{gamma1ddotgammainL1} and the assumptions of Theorem \ref{compactnesstheorem} we estimate 
\begin{align}
&\limsup_{n\rightarrow\infty}\left|\int_0^1 \ddot\gamma^{(n)}\gamma_1^{(n)}(\phi -\phi_\epsilon)dt\right|\nonumber\\
\leq &\sup_{n\in\N}\left[\left(\int_0^1|\ddot\gamma^{(n)}|^2\gamma_1^{(n)}dt\right)^ {\frac12}\left(\sup\gamma_1^{(n)}\right)^ {\frac12}\right]\|\phi-\phi_\epsilon\|_{L^2((0,1))}\leq C\epsilon.\label{phiapprox2}
\end{align}
Combining Estimates \eqref{phiapprox1} and \eqref{phiapprox2} and letting $\epsilon\rightarrow0^+$ shows that Equation \eqref{identityforsmoothtests} is valid for $\phi\in L^2((0,1))$ also. 
\end{proof}
\renewcommand{\proofname}{\textit{Proof.}}

\section{General Regularity Theory}\label{GeneralRegularityAppendix}
Let $U\subset\R$ be an open interval and $\gamma\in W^{1,\infty}(U,\mathcal H^2)\cap W^ {2,2}(U,\mathcal H^2)$. For $\alpha,\beta\in\R$ we define the functional 
\begin{equation}\label{IFUnctionalPDE}
I(\phi)
:=
\alpha\int_0^1 \gamma_1^2(t)\dot\phi_2(t)+2\gamma_1(t)\dot\gamma_2(t)\phi_1(t)dt
+
\beta\int_0^1|\dot\gamma(t)|\phi_1(t)+\frac{\gamma_1(t)\langle\dot\phi(t),\dot\gamma(t)\rangle}{|\dot\gamma(t)|}dt.
\end{equation}
We assume that $\gamma$ satisfies the equation
\begin{equation}\label{regappendixweakeq}
\delta\mathcal W[\gamma]\phi=I(\phi)\hspace{.5cm}\textrm{for all }\phi\in C^\infty_0(U).
\end{equation}
The goal of this section is to prove the following regularity theorem:

\begin{theorem}[Regularity Theorem]\label{GeneralRegularityTheorem}
    Let $\gamma\in W^{1,\infty}(U,\mathcal H^2)\cap W^{2,2}(U,\mathcal H^2)$ satisfy Equation \eqref{regappendixweakeq}. Further, we assume that $|\dot\gamma(t)|=L>0$ and
    \begin{align}
    &|\alpha|+|\beta|+L+L^{-1}+\|\gamma\|_{C^0(U)}+\|\ddot\gamma\|_{L^2(U)}+\mathcal W[\gamma]\leq M,\label{uniformestimate}\\
    &\kappa(\gamma;U):=\inf_{t\in U}\gamma_1(t)>0.\label{kappadef}
    \end{align}
    Then $\gamma\in C^\infty(U)$ and $\|\gamma\|_{C^m(U)}\leq C(M, \kappa(\gamma;U),U,m)$. The constant is increasing in the first and decreasing in the second slot. 
\end{theorem}

The proof is split into several lemmas. We remark that Eichmann and Grunau have carried out similar arguments in \cite{eichmann} for the Willmore equation without the isoperimetric Lagrange multiplier.

\begin{lemma}\label{regularitystep1}
    Under the assumptions of Theorem \ref{GeneralRegularityTheorem} $\gamma\in W^{2,\infty}(U)$, $H\in L^\infty(U)$ and 
    $$\|\gamma\|_{W^{2,\infty}(U)}+\|H\|_{L^\infty(U)}\leq C(M,\kappa(\gamma;U),U).$$
\end{lemma}
\begin{proof}
Let $U=(\tau_1,\tau_2)$. Throughout the proof $C$ will denote a constant that is allowed to depend on $M$, $\kappa(\gamma;U)$ and $U$ and that may change from line to line.\\
\ \\
\noindent
\textbf{Step 1: Constructing an Appropriate Test Function}\ \\
Given $\varphi=(\varphi_1,\varphi_2)\in C_0^\infty(U,\R^2)$ define 
$$\zeta(t):=\int_{\tau_1}^{t} \int_{\tau_1} ^s \varphi(r)dr ds.$$
Let $\tau_1<r_1<r_2<\tau_2$ such that $\operatorname{supp}(\varphi)\subset[r_1,r_2]$. Then $\zeta(t)=0$ when $t\in[\tau_1,r_1]$ and $\zeta(t)=a+bt$ when $t\in[r_2,\tau_2]$ where 
$$a=\int_{\tau_1}^{\tau_2}\int_{\tau_1}^s\varphi(r)dr ds-b\tau_2\hspace{.5cm}\textrm{and}\hspace{.5cm} b=\int_{\tau_1}^{\tau_2} \varphi(r)dr.$$
We choose a smooth function $\chi\equiv 0$ on $[\tau_1,\tau_1+\frac13(\tau_2-\tau_1)]$ and $\chi\equiv 1$ on $[\tau_1+\frac23(\tau_2-\tau_1),\tau_2]$ and define 
\begin{equation}\label{phiconstruction}
\phi(t):=\zeta(t)-\chi(t)(a+bt).
\end{equation}
It is easy to see that $\phi\in C^ \infty_0(U)$ and that $\|\phi\|_{C^1(U)}\leq C\|\varphi\|_{L^1(U)}$. \\

\noindent
\textbf{Step 2: Weak Euler-Lagrange Equation}\ \\
We compute 
\begin{equation}\label{weakeulerLagrange01}
    I(\phi)=\delta\mathcal W[\gamma]\phi=\pi\int_{\tau_1}^ {\tau_2} H\left(
    \delta H\phi
    \right)\gamma_1|\dot\gamma|dt
    +
    \frac\pi2\int_{\tau_1}^ {\tau_2} H^2 \left(
    |\dot\gamma|\phi_1+\gamma_1\frac{\langle\dot\gamma,\dot{\phi}\rangle}{|\dot\gamma|}
    \right)dt.
\end{equation}
Considering the definition of $I(\phi)$ it is easy to see that $|I(\phi)|\leq\|\phi\|_{C^1}$.
Using Estimates \eqref{uniformestimate} and \eqref{kappadef} we estimate 
\begin{align}
\left|\int_{\tau_1}^ {\tau_2} H^2
\left(|\dot\gamma|\phi_1+\gamma_1 \frac{\langle\dot\gamma,\dot{\phi}\rangle}{|\dot\gamma|} \right)dt\right|
    \leq
    &\left(L+\|\gamma_1\|_{C^0}\right) \| \phi\|_{C^1}\int_{\tau_1}^ {\tau_2} H^2  dt\nonumber\\
\leq & \frac{C\|\phi\|_{C^1}}{2\pi L\kappa(\gamma;U)}\int_{\tau_1}^ {\tau_2} H^22\pi\gamma_1 |\dot\gamma|dt\nonumber\\
\leq &C\|\phi\|_{C^1}.\label{int2estimate0}
\end{align}
We recall that $H=k_1+k_2$. Using the formulas for $k_1$ and $k_2$ from Equation \eqref{principalcurvatures} it is easy to derive the following equations:
\begin{align}
    \delta k_1\phi&=\frac{\ddot{\phi}_2\dot\gamma_1+\ddot\gamma_2\dot{\phi}_1-\ddot{\phi}_1\dot\gamma_2-\ddot\gamma_1\dot{\phi}_2}{L^3}
    -
    3\frac{\ddot\gamma_2\dot\gamma_1-\ddot\gamma_1\dot\gamma_2}{L^4}\frac{\langle\dot\gamma,\dot{\phi}\rangle}{L}\label{vark1}\\
    \delta k_2\phi&=\frac{\dot{\phi}_2}{\gamma_1 L}-\frac{\dot\gamma_2}{\gamma_1^2L}\phi_1-\frac{\dot\gamma_2}{\gamma_1L^2}\frac{\langle\dot\gamma,\dot{\phi}\rangle}{L}\label{vark2}
\end{align}
Next, we estimate the terms in $\delta H\phi$ that do not contain $\ddot{\phi}$. Using $|\dot\gamma|=L$ and Estimate \eqref{uniformestimate} we get
\begin{equation}\label{kestimates}
|\delta k_2\phi|
\leq
C\|\phi\|_{C^1}
\hspace{.5cm}\textrm{and}\hspace{.5cm}
\left|\frac{\ddot\gamma_2\dot{\phi}_1-\ddot\gamma_1\dot{\phi}_2}{L^3}
    -
    3\frac{\ddot\gamma_2\dot\gamma_1-\ddot\gamma_1\dot\gamma_2}{L^4}\frac{\langle\dot\gamma,\dot\phi\rangle}{L}\right|
\leq 
C\|\phi\|_{C^1}|\ddot\gamma|.
\end{equation}
Using Equation \eqref{weakeulerLagrange01} and inserting Estimates \eqref{int2estimate0} and \eqref{kestimates} yields
\begin{align} 
\left|\int_{\tau_1}^ {\tau_2} H(\ddot{\phi}_2\dot\gamma_1-\ddot{\phi}_1\dot\gamma_2)\gamma_1|\dot\gamma|dt\right|
\leq 
&C\|\phi\|_{C^1}\left[1+\int_{\tau_1}^ {\tau_2} |H|\gamma_1|\dot\gamma_1|dt+\int_{\tau_1}^{\tau_2} |H|\ |\ddot\gamma|\gamma_1|\dot\gamma_1|dt\right]\nonumber \\
\leq
&C\|\phi\|_{C^1}(1+\|\ddot\gamma\|_{L^2{((\tau_1,\tau_2))}})\nonumber \\
\leq 
&C\|\phi\|_{C^1}.\label{step2finalestimate0}
\end{align}
Here we have first used the Hölder inequality and the bound on the Willmore energy from Estimate \eqref{uniformestimate}.\\

\noindent
\textbf{Estimate Auxiliary Terms}\ \\
For the moment, we abbreviate $\xi:=\chi(a+bt)$. We use Estimate \eqref{uniformestimate} to estimate 
\begin{align}
&\left|\int_{\tau_1}^ {\tau_2} H(\ddot\xi_2\dot\gamma_1-\ddot\xi_1\dot\gamma_2)\gamma_1|\dot\gamma|dt\right|\nonumber
\leq 
2\|\ddot\xi\|_{C^0}\int_{\tau_1}^ {\tau_2}  |H|\ |\dot\gamma|^2\gamma_1dt\\
\leq &C\|\ddot\xi\|_{C^0}\sqrt{\mathcal W[\gamma]}
\leq C\|\xi\|_{C^2}.\label{smoothestimate}
\end{align}
We wish to use this estimate to replace $\phi$ in Estimate \eqref{step2finalestimate0} with $\zeta$. To do so, we recall $\phi=\zeta-\chi(a+bt)$. Using $\ddot\zeta=\varphi$, Estimate \eqref{smoothestimate} as well as the fact that $|a|+|b|\leq C\|\varphi\|_{L^1}$ we get
\begin{equation}\label{endresultstep2}
\left|\int_{\tau_1}^ {\tau_2} H(\varphi_2\dot\gamma_1-\varphi_1\dot\gamma_2)\gamma_1|\dot\gamma|dt\right|
\leq 
C\|\varphi\|_{L^1((0,1))}.
\end{equation}
In particular, choosing $\varphi=(\varphi_1,0)$ we find that for all $\varphi_1\in C^\infty_0((\tau_1,\tau_2))$ we have 
\begin{equation}\label{firstregest}
\left|\int_{\tau_1}^ {\tau_2} H\varphi_1\dot\gamma_2\gamma_1|\dot\gamma|dt\right|
\leq 
C\|\varphi_1\|_{L^1((\tau_1,\tau_2))}.
\end{equation}

\noindent
\textbf{Regularity Improvement}\ \\
Estimate \eqref{firstregest} allows us to extend the integral expression  $\int_{\tau_1}^ {\tau_2} H\dot\gamma_2\gamma_1|\dot\gamma|\varphi_1dt$ to a linear and continuous functional $\Lambda: L^1((\tau_1,\tau_2))\rightarrow\R$. As $(L^1((\tau_1,\tau_2)))'=L^\infty((\tau_1,\tau_2))$ we may deduce  
 $H\dot\gamma_2\gamma_1|\dot\gamma|\in L^\infty((\tau_1,\tau_2))$ with 
 $$\|H\dot\gamma_2 \gamma_1|\dot\gamma|\ \|_{L^\infty((\tau_1,\tau_2))}\leq C.$$
 As $\gamma_1\geq\kappa(\gamma;U)$ on $(\tau_1,\tau_2)$ and $|\dot\gamma|=L$ we obtain $\|H\dot\gamma_2\|_{L^\infty((\tau_1,\tau_2))}\leq C$. Using $\varphi=(0,\varphi_2)$ one can derive a similar estimate for $\|H\dot\gamma_1\|_{L^\infty((\tau_1,\tau_2))}$. Combining both estimates, we get
$$\|H\|_{L^\infty((\tau_1,\tau_2))}=\frac{\|H\dot\gamma\|_{L^\infty((\tau_1,\tau_2))}}{L}\leq\frac1L(\|H\dot\gamma_1\|_{L^\infty((\tau_1,\tau_2))}+\|H\dot\gamma_2\|_{L^\infty((\tau_1,\tau_2))})\leq C.$$
Using the identities from Equation \eqref{ArcLength_H_Identity}, we deduce $\gamma\in W^{2,\infty}((\tau_1,\tau_2))$ and the estimate $\|\ddot\gamma\|_{L^\infty((\tau_1,\tau_2))}\leq C$. 
\end{proof}

\begin{lemma}\label{regularitystep2}
        Under the assumptions of Theorem \ref{GeneralRegularityTheorem} $\gamma\in W^{3,\infty}(U)$, $H\in W^{1,\infty}(U)$ and 
    $$\|\gamma\|_{W^{3,\infty}(U)}+\|H\|_{W^{1,\infty}(U)}\leq C(M,\kappa(\gamma;U),U).$$
\end{lemma}
\begin{proof}
Let again $U=(\tau_1,\tau_2)$. As in the proof of Lemma \ref{regularitystep1} $C$ always refers to a constant depending on $M$, $\kappa(\gamma;U)$ and $U$. Also, the rest of the proof is very similar to the proof of Lemma \ref{regularitystep1} and we only sketch the parts that need to be altered. First, the test function has to get modified. Given $\varphi=(\varphi_1,\varphi_2)\in C_0^ \infty((\tau_1,\tau_2),\R^2)$ we put
$$\zeta(t):=\int_0^ t\varphi(r)dr.$$
Let $[r_1,r_2]=\operatorname{supp}\varphi$. Then $\zeta(t)=0$ for $t\in[\tau_1, r_1]$ and $\zeta(t)=a$ when $t\in[r_2,\tau_2]$ where
$$a=\int_{\tau_1}^ {\tau_2}\varphi(r) dr.$$
We choose the same smooth function $\chi$ as in the proof of Lemma \ref{regularitystep1} and put $\phi(t):=\zeta(t)-a\chi(t)$. It is easy to see that $\phi\in C^ \infty_0(U)$ and that $\|\phi\|_{C^ 0}\leq C\|\varphi\|_{L^1}$ and $\|\dot\phi\|_{L^1}\leq C\|\varphi\|_{L^1}$. 
Testing the weak Euler-Lagrange equation with $\phi$ as we did in Equation \eqref{weakeulerLagrange01} and using the estimates derived in Lemma \ref{regularitystep1} we get
\begin{equation}\label{step2estimatefehlthiername}
\left|\int_{\tau_1}^ {\tau_2} \hspace{-.3cm}H\left(\delta H\phi\right)\gamma_1|\dot\gamma|dt\right|\leq 
C\left[\int_{\tau_1}^ {\tau_2}\hspace{-.3cm} H^2\left(|\dot\gamma|\ |\phi_1|+\gamma_1 |\dot{\phi}|\right) dt+|I(\phi)|\right]
\leq
C\|\phi\|_{W^ {1,1}(U)}.
\end{equation}
Next, we take care of the terms in $\delta H$ that do not contain $\ddot{\phi}$. Recalling Equations \eqref{vark1} and \eqref{vark2} and using the result of Lemma \ref{regularitystep1}, we estimate 
\begin{align}
\left|\delta H\phi-\frac{\ddot{\phi}_2\dot\gamma_1-\ddot{\phi}_1\dot\gamma_2}{L^3}\right|
&=
|\delta k_2\phi|
+
\left|\frac{\ddot\gamma_2\dot{\phi}_1-\ddot\gamma_1\dot{\phi}_2}{L^3}
    -
    3\frac{\ddot\gamma_2\dot\gamma_1-\ddot\gamma_1\dot\gamma_2}{L^4}\frac{\langle\dot\gamma,\dot\phi\rangle}{L}\right|\nonumber\\
&\leq 
C\left(|\phi|+|\dot{\phi}|\right).\label{kestimates2}
\end{align}
Inserting into Estimate \eqref{step2estimatefehlthiername} and using the estimate for $\|H\|_{L^\infty((\tau_1,\tau_2))}$ from Lemma \ref{regularitystep1}, we get
$$\left|\int_{\tau_1}^ {\tau_2} H(\ddot{\phi}_2\dot\gamma_1-\ddot{\phi}_1\dot\gamma_2)\gamma_1|\dot\gamma|dt\right|
\leq 
C\|\phi\|_{W^ {1,1}(U)}.
$$
As in the proof of Lemma \ref{regularitystep1}, we now insert the formula for $\phi$ and estimate the auxiliary terms to get 
\begin{equation}\label{lemma2est}
\left|\int_{\tau_1}^ {\tau_2} H(\dot\varphi_2\dot\gamma_1-\dot\varphi_1\dot\gamma_2)\gamma_1|\dot\gamma|dt\right|
\leq 
C\|\phi\|_{W^ {1,1,}}\leq C\|\varphi\|_{L^1}.
\end{equation}
Now we argue similarly to as we did in the final step of the proof of Lemma \ref{regularitystep1}. First we take $\varphi=(\varphi_1,0)$ and consider the densely defined linear operator
$$
\Lambda:L^1((\tau_1,\tau_2))\rightarrow\R,
\hspace{.5cm}
\Lambda(\varphi_1):=\int_{\tau_1}^{\tau_2} H\dot\varphi_1\dot\gamma_2\gamma_1|\dot\gamma|dt.
$$
Estimate \eqref{lemma2est} allows us to extend $\Lambda$ to a continuous linear map on all of $L^1((\tau_1,\tau_2))$. Now, by the Riesz representation theorem (see e.g. \cite{hirzebruch}, Theorem 19.2), there exists a function $h\in L^ \infty((\tau_1,\tau_2))$ such that $\Lambda(\varphi)=\int_{\tau_1}^ {\tau_2} h \varphi dt$. Going back to $\varphi\in C_0^ \infty((\tau_1,\tau_2))$ we get 
$$
\int_{\tau_1}^ {\tau_2} h\varphi_1dt
=
\int_{\tau_1}^ {\tau_2} H\gamma_1\dot\gamma_2|\dot\gamma|\dot\varphi_1 dt\hspace{.5cm}\textrm{for all }\varphi_1\in C^ \infty_0((\tau_1,\tau_2)).$$
Thus $H\gamma_1\dot\gamma_2|\dot\gamma|$ is weakly differentiable and Estimate \eqref{lemma2est} implies 
$$\|(H\gamma_1\dot\gamma_2|\dot\gamma|)'\|_{L^ \infty((\tau_1,\tau_2))}\leq C.$$
Using the results from Lemma \ref{regularitystep1} as well as $\gamma_1\geq\kappa(\gamma;U)$ on $U=(\tau_1,\tau_2)$ we obtain $H\dot\gamma_2\in W^ {1,\infty}((\tau_1,\tau_2))$ with $\|H\dot\gamma_2\|_{W^ {1,\infty}((\tau_1,\tau_2))}\leq C$. Taking $\varphi=(0,\varphi_2)$ instead, we can derive a similar estimate for $H\dot\gamma_1$. The lemma follows as in the proof of Lemma \ref{regularitystep1} by using the identities in Equation \eqref{ArcLength_H_Identity}.
\end{proof}

\begin{lemma}
    Under the assumptions of Theorem \ref{GeneralRegularityTheorem} $\gamma\in W^{k+2,\infty}(U)$, $H\in W^{k,\infty(U)}$ for all $k\in\N_0$ and 
    $$\|\gamma\|_{W^{k+2,\infty}(U)}+\|H\|_{W^{k,\infty}(U)}\leq C(M,\kappa(\gamma;U),U,k).$$
\end{lemma}
\begin{proof}
As before, $C$ will always refer to a constant depending on $M$, $\kappa(\gamma; U)$ and $U$ and this time also $k$. The proof is inductive. The case $k=1$ follows from Lemma \ref{regularitystep2}. Therefore it suffices to prove the following implication for $k\geq 1$:
\begin{equation}\label{implication}
\left\{\begin{aligned} 
&H\in W^{k,\infty}(U)\\
&\|H\|_{W^ {k,\infty}(U)}\leq C\\
&\gamma\in W^{k+2,\infty}(U)\\
&\|\gamma\|_{W^ {k+2,\infty}(U)}\leq C
\end{aligned}\right\}
\hspace{.5cm}\Rightarrow \hspace{.5cm}
\left\{\begin{aligned} 
&H\in W^{k+1,\infty}(U)\\
&\|H\|_{W^ {k+1,\infty}(U)}\leq C\\
&\gamma\in W^{k+3,\infty}(U)\\
&\|\gamma\|_{W^ {k+3,\infty}(U)}\leq C
\end{aligned}\right\}
\end{equation}
To establish \eqref{implication} we take $\varphi=(\varphi_1,\varphi_2)\in C_0^ \infty(U,\R^2)$ and choose the test function $\phi:=\varphi^ {(k-1)}$
Testing the weak version of the Euler-Lagrange equation with $ \phi$ as we did in Equation \eqref{weakeulerLagrange01} gives 
\begin{align}
 &\left|\int_{\tau_1}^ {\tau_2} H\left(\delta H\phi\right)\gamma_1|\dot\gamma|dt\right|
\leq 
C\left|
\int_{\tau_1}^ {\tau_2}H^2\left(
|\dot\gamma|\phi_1+\gamma_1\frac{\langle\dot\gamma,\dot{\phi}\rangle}{|\dot\gamma|}
\right)dt
\right|+C|I(\phi)|\nonumber\\
    \leq & C\left|\int_{\tau_1}^ {\tau_2}H^2\left(|\dot\gamma|\varphi^{(k-1)}_1(t)+\gamma_1\frac{\langle\dot\gamma,\varphi^{(k)}(t)\rangle}{|\dot\gamma|}\right)dt\right|+C|I(\varphi^{(k-1)})|.\label{longestimate} 
\end{align}
 First, we establish a suitable estimate for $I(\varphi^ {(k-1)})$. By assumption $\gamma\in W^ {k+2,\infty}(U)$. Therefore we can integrate the integrals in Equation \eqref{IFUnctionalPDE} by parts to obtain 
$|I(\varphi^ {(k-1)})| \leq C\|\varphi\|_{L^1}$. 
Also using the assumption $H\in W^ {k,\infty}(U)$ and $\gamma\in W^ {k+2,\infty}(U)$, we can rewrite the integral on the right in Estimate \eqref{longestimate} by integrating by parts several times and estimate it by $C\|\varphi\|_{L^1}$. In total, we deduce
\begin{equation}\label{step3zwischenfehltname}
\left|
\int_{\tau_1}^ {\tau_2} H\left(\delta H\phi\right)\gamma_1|\dot\gamma|dt
\right|
\leq 
C\|\varphi\|_{L^1}.
\end{equation}
Next, we again estimate all terms in $\delta H\phi$ that do not contain $\ddot{\phi}$. Using Equations \eqref{vark1} and \eqref{vark2}, $\gamma\in W^ {k+2,\infty}$ and repeated integration by parts we get
\begin{align*}
   & \left|\int_{\tau_1}^ {\tau_2}H\left(
   \delta H\phi-\frac{\dot\gamma_1\ddot{\phi}_2-\dot \gamma_2\ddot{ \phi}_1}{L^3}
   \right)\gamma_1|\dot\gamma|dt\right|\\
   \leq & \left|\int_{\tau_1}^ {\tau_2}H
   \left(
        \delta k_2\phi
    +
    \frac{\ddot\gamma_2\dot{\phi}_1-\ddot\gamma_1\dot{\phi}_2}{L^3}
    -
    3\frac{\ddot\gamma_2\dot\gamma_1-\ddot\gamma_1\dot\gamma_2}{L^4}\frac{\langle\dot\gamma,\dot{\phi}\rangle}{L}\right)
   \gamma_1|\dot\gamma|dt\right|\\
   \leq & C\|\varphi\|_{L^1}.
\end{align*}
Combining with Estimate \eqref{step3zwischenfehltname}, we find 
$$\left|\int_{\tau_1}^ {\tau_2} H
\frac{\dot\gamma_1\varphi^ {(k+1)}_2-\dot\gamma_2\varphi^ {(k+1)}_1}{L^3}\gamma_1 |\dot\gamma| dt\right|
\leq
C\|\varphi\|_{L^1((\tau_1,\tau_2))}.
$$
As in the proof of Lemma \ref{regularitystep1} this now implies $H\in W^ {k+1,\infty}((\tau_1,\tau_2))$ and hence $\gamma\in W^ {k+3,\infty}((\tau_1,\tau_2))$. This establishes Implication \eqref{implication} and thereby the lemma.
\end{proof}

\section{Calculus Lemmas  }
\begin{lemma}\label{oderegularitylemma}
Let $n\in\N_0$, $u\in C^1((0,1))\cap C^0([0,1))$, $h\in C^n([0,1)$ and 
$$\frac{(x u(x))'}x=h(x).$$
Then $u\in C^{n+1}([0,1))$.
\end{lemma}
\begin{proof}
First, we consider the case $n=0$. We have 
\begin{equation}\label{odeRegLemmaEq1}
xu(x)=0\cdot u(0)+\int_0^x th(t)dt.
\end{equation}
So, for $x\in (0,1)$ we get 
$$u'(x)=\frac d{dx}\left[\frac1x\int_0^x th(t)dt\right]=h(x)-\frac1{x^2}\int_0^x th(t)dt.$$
We claim that $u'(x)\rightarrow \frac12 h(0)$ as $x\rightarrow 0$. Indeed
$$\frac1{x^2}\int_0^x h(t)tdt=\int_0^1 s h(sx) ds\rightarrow h(0)\int_0^1 s ds=\frac12 h(0).$$
Now we consider the case $n>0$. Rewriting Equation \eqref{odeRegLemmaEq1} we get
$$u(x)=x\int_0^1h(tx ) t dt$$
and deduce $u\in C^n([0,1))$. Next, compute 
\begin{align*}
    u'(x)=&\int_0^1 th(tx) dt+x\int_0^1 t^2 h'(tx) dt\\
        &\int_0^1 th(tx) dt
            +x\left[\frac {t^2h(tx)}x\bigg|_{t=0}^{t=1}-\int_0^1 2t\frac{h(tx)}xdt\right]\\
        =&h(x)-\int_0^1 t h(tx)dt.
\end{align*}
Clearly, this last formula is $C^n([0,1])$, which proves the lemma. 
\end{proof}

\begin{lemma}\label{uniformboundedLemma}
Let $r_0>0$, $f\in C^0([0, r_0])\cap C^1((0, r_0])$ and $w\in C^2([0, r_0])$ such that 
$$w''(r)+\left(\frac wr\right)'=f(r).$$
Then $|w''(r)|\leq \frac 43 \sup_{0\leq\rho\leq r}|f(\rho)|$ and consequently $|w'(r)|\leq |w'(r_0)|+\frac43 r_0\sup_{0\leq\rho\leq r_0}|f(\rho)|$.
\end{lemma}
\begin{proof}
    We compute 
    $$
    \left(\frac{(wr)'}r\right)'
    =
    \left(w'+\frac wr\right)'
    =
    f.
    $$
Integrating this equation from $r>0$ to $r_0$ shows 
$$-\frac{(w(r)r)'}r+w'(r_0)+\frac{w(r_0)}{r_0}=\int_r^{r_0} f(s)ds=\int_0^{r_0} f(s)ds-\int_0^rf(s)ds.$$
We put $c_0:=w'(r_0)+\frac{w(r_0)}{r_0}-\int_0^{r_0} f(s)ds$. Then 
$$\frac{(rw(r))'}r=c_0+\int_0^r f(s)ds.$$
We multiply by $r$. Another integration from $r$ to $r_0$ shows 
\begin{align*}
r_0 w(r_0)-rw(r)&=\frac12 c_0(r_0^2-r^2)+\int_r^{r_0}x\int_0^x f(s)dsdx\\
&=-\frac12 c_0 r^2 +\frac12 c_0 r_0^2+\int_0^{r_0}x\int_0^x f(s)dsdx-\int_0^{r}x\int_0^x f(s)dsdx.
\end{align*}
This shows that there exist constants $c_1,c_2\in\R$ such that
$$rw(r)=c_1 r^2+c_2+\int_0^r \int_0^x xf(s)dsdx.$$
Taking $r=0$ shows $c_2=0$. Next, we substitute $x=ry$ and obtain 
$$w(r)=c_1r+r\int_0^1\int_0^{ry}y f(s)dsdy.$$
We compute two $r$-derivatives. 
\begin{align*} 
w''(r)=&2\int_0^1 \frac\partial{\partial r}\int_0^{yr}yf(s)dsdy+r\int_0^1 \frac{\partial^2}{\partial r^2}\int_0^{ry} yf(s)dsdy\\
=&2\int_0^1 y^2f(yr)dy+r\int_0^1 y^3 f'(yr)dy\\
=&2\int_0^1 y^2f(yr)dy+\int_0^1 y^3\frac{\partial}{\partial y} f(yr)dy\\
=&2\int_0^1 y^2f(yr)dy+y^3 f(yr)\bigg|_{y=0}^{y=1}-\int_0^1 3y^2f(yr)dy\\
=& -\int_0^1 y^2f(yr)dy+ f(r)
\end{align*}
Consequently 
$$|w''(r)|
\leq \left(1+\int_0^1 y^2 dy\right)\sup_{0\leq\rho\leq r}|f(\rho)|
=\frac 43 \sup_{0\leq\rho\leq r}|f(\rho)|.$$
\end{proof}

\begin{lemma}\label{derivativeinL1}
Let $f\in C^ 1((0,\rho))$ such that $f'\in L^ 1((0,\rho))$. Then $f\in C^ 0([0,\rho])$.
\end{lemma}
\begin{proof}
We prove continuity at $x=0$. The case $x=\rho$ works similarly. Let $x<y\in(0,\rho]$. Then 
$$|f(x)-f(y)|\leq \int_{x}^ y|f'(t)|dt\leq \int_0^  y|f'(t)|dt.$$
This implies that $f(\epsilon_n)$ is Cauchy for any $\epsilon_n\rightarrow 0^ +$ and hence $f(0):=\lim_{x\rightarrow0^+}f(x)$ is well-defined. To see that $f$ is continuous at $x=0$, we use $f'\in L^1((0,1))$ to estimate
$$|f(0)-f(y)|\leq \int_0^y |f'(t)|dt\rightarrow 0,\hspace{.5cm}\textrm{as $y\rightarrow 0^ +$.}$$
\end{proof}

\section*{Acknowledgements}
The author would like to thank Ernst Kuwert for the suggestion of this interesting topic and the
helpful guidance as well as Marius Müller for the many helpful discussions.
\bibliographystyle{plain}
\bibliography{Quellen}

\begin{thebibliography}{10}

\bibitem{AK}
Roberta Alessandroni and Ernst Kuwert.
\newblock {L}ocal solutions to a free boundary problem for the {W}illmore
  functional.
\newblock {\em Calculus of Variations and Partial Differential Equations},
  55(2):1--29, 2016.

\bibitem{canham}
Peter~B. Canham.
\newblock The minimum energy of bending as a possible explanation of the
  biconcave shape of the human red blood cell.
\newblock {\em Journal of theoretical biology}, 26(1):61--81, 1970.

\bibitem{chenODE}
Jingyi Chen and Yuxiang Li.
\newblock Radially symmetric solutions to the graphic willmore surface
  equation.
\newblock {\em The Journal of Geometric Analysis}, 27(1):671--688, 2017.

\bibitem{choskiveneroni}
Rustum Choksi and Marco Veneroni.
\newblock Global minimizers for the doubly-constrained helfrich energy: the
  axisymmetric case.
\newblock {\em Calculus of Variations and Partial Differential Equations},
  48(3):337--366, 2013.

\bibitem{helfrichDeuling}
Heinz~J. Deuling and Wolfgang Helfrich.
\newblock Red blood cell shapes as explained on the basis of curvature
  elasticity.
\newblock {\em Biophysical journal}, 16(8):861--868, 1976.

\bibitem{dziuk2006error}
Gerhard Dziuk and Klaus Deckelnick.
\newblock Error analysis of a finite element method for the willmore flow of
  graphs.
\newblock {\em Interfaces and free boundaries}, 8(1):21--46, 2006.

\bibitem{eichmann}
Sascha Eichmann and Hans-Christoph Grunau.
\newblock Existence for willmore surfaces of revolution satisfying
  non-symmetric dirichlet boundary conditions.
\newblock {\em Advances in Calculus of Variations}, 12(4):333--361, 2019.

\bibitem{helfrich}
Wolfgang Helfrich.
\newblock Elastic properties of lipid bilayers: theory and possible
  experiments.
\newblock {\em Zeitschrift f{\"u}r Naturforschung}, 28(11-12):693--703, 1973.

\bibitem{hirzebruch}
Friedrich Hirzebruch and Winfried Scharlau.
\newblock {\em Einf{\"u}hrung in die Funktionalanalysis}.
\newblock Spektrum Akademischer Verlag, 1991.

\bibitem{KellerIso}
Laura Gioia~Andrea Keller, Andrea Mondino, and Tristan Rivi{\`e}re.
\newblock Embedded surfaces of arbitrary genus minimizing the willmore energy
  under isoperimetric constraint.
\newblock {\em Archive for Rational Mechanics and Analysis}, 212(2):645--682,
  2014.

\bibitem{kusnerISO}
Robert Kusner and Peter McGrath.
\newblock On the canham problem: Bending energy minimizers for any genus and
  isoperimetric ratio.
\newblock {\em Archive for Rational Mechanics and Analysis}, 247(1):10, 2023.

\bibitem{kuwertli}
Ernst Kuwert and Yuxiang Li.
\newblock Asymptotics of willmore minimizers with prescribed small
  isoperimetric ratio.
\newblock {\em SIAM Journal on Mathematical Analysis}, 50(4):4407--4425, 2018.

\bibitem{kuwert2002gradient}
Ernst Kuwert and Reiner Sch{\"a}tzle.
\newblock Gradient flow for the willmore functional.
\newblock {\em Communications in Analysis and Geometry}, 10(2):307--339, 2002.

\bibitem{removability}
Ernst Kuwert and Reiner Schätzle.
\newblock Removability of point singularities of willmore surfaces.
\newblock {\em Annals of Mathematics}, 160(1):315--357, 2004.

\bibitem{leoni2017first}
Giovanni Leoni.
\newblock {\em A first course in Sobolev spaces}.
\newblock American Mathematical Soc., 2017.

\bibitem{MayerBook}
Karl~H. Mayer.
\newblock {\em Algebraische {T}opologie}.
\newblock Birkh\"{a}user Verlag, Basel, 1989.

\bibitem{MondinoScharrerISO}
Andrea Mondino and Christian Scharrer.
\newblock A strict inequality for the minimization of the willmore functional
  under isoperimetric constraint.
\newblock {\em Advances in Calculus of Variations}, 2021.

\bibitem{Nagasawa}
Takeyuki Nagasawa and Izumi Takagi.
\newblock Bifurcating critical points of bending energy under constraints
  related to the shape of red blood cells.
\newblock {\em Calculus of Variations and Partial Differential Equations},
  16(1):63--111, 2003.

\bibitem{pugh2003real}
Charles~C. Pugh.
\newblock {\em Real Mathematical Analysis}.
\newblock Springer New York, 2003.

\bibitem{Ruppiso}
Fabian Rupp.
\newblock The willmore flow with prescribed isoperimetric ratio.
\newblock {\em arXiv preprint arXiv:2106.02579}, 2021.

\bibitem{scharrer2022embedded}
Christian Scharrer.
\newblock Embedded delaunay tori and their willmore energy.
\newblock {\em Nonlinear Analysis}, 223:113010, 2022.

\bibitem{schygulla}
Johannes Schygulla.
\newblock Willmore minimizers with prescribed isoperimetric ratio.
\newblock {\em Archive for Rational Mechanics and Analysis}, 203(3):901, 2012.

\bibitem{seifert}
Udo Seifert, Karin Berndl, and Reinhard Lipowsky.
\newblock Shape transformations of vesicles: Phase diagram for
  spontaneous-curvature and bilayer-coupling models.
\newblock {\em Physical review A}, 44(2):1182, 1991.

\bibitem{SimonMonotonicity}
Leon Simon.
\newblock Existence of surfaces minimizing the willmore functional.
\newblock {\em Communications in Analysis and Geometry}, 1(2):281--326, 1993.

\end{thebibliography}

\end{document}